\documentclass[12pt]{article}

\usepackage{amsmath,amsthm,amssymb}
\usepackage[colorlinks,
            linkcolor=red,
            anchorcolor=blue,
            citecolor=magenta
            ]{hyperref}
\usepackage{pdfsync}
\usepackage[numbers,sort&compress]{natbib}
\allowdisplaybreaks

\theoremstyle{definition}

\theoremstyle{plain}
\newtheorem{theorem}{Theorem}[section]

\newtheorem{remark}{Remark}[section]
\newtheorem{lemma}{Lemma}[section]

\newtheorem{proposition}{Proposition}[section]

\numberwithin{equation}{section}

\newcommand{\vs}{\vspace}

\begin{document}

\title{Normalized solutions for Sch\"odinger equations
with potential and general nonlinearities involving critical case on large convex domains\footnote{  This work was partially supported by NNSFC (No. 11671407 and No. 11801076), FDCT (No. 0091/2018/A3), Guangdong Special Support Program (No. 8-2015) and the key project of NSF of Guangdong Province (No. 2016A030311004).}}

\author{ Jun Wang$^{a}$, Zhaoyang Yin$^{a, b}$\footnote {Corresponding author. wangj937@mail2.sysu.edu.cn (J. Wang), mcsyzy@mail.sysu.edu.cn (Z. Yin)
} \\
{\small $^{a}$Department of Mathematics, Sun Yat-sen University, Guangzhou, 510275, China } \\
{\small $^{b}$Shenzhen Campus of Sun Yat-sen University, Shenzhen, 518107, China } \\
}

	\date{}

	\maketitle

\date{}

 \maketitle \vs{-.7cm}

  \begin{abstract}
In this paper, we
study the  following Schr\"odinger equations with potentials and general nonlinearities
\begin{equation*}
 \left\{\aligned
& -\Delta u+V(x)u+\lambda u=|u|^{q-2}u+\beta f(u),  \\
& \int |u|^2dx=\Theta,
\endaligned
\right.
\end{equation*}
both on $\mathbb{R}^N$ as well as on domains $r \Omega$ where $\Omega \subset \mathbb{R}^N$ is an open bounded convex domain and $r>0$ is large. The exponent satisfies $2+\frac{4}{N}\leq q\leq2^*=\frac{2 N}{N-2}$ and $f:\mathbb{R}\rightarrow \mathbb{R}$ satisfies $L^2$-subcritical or $L^2$-critical growth. This paper generalizes the conclusion of Bartsch et al. in \cite{TBAQ2023}(2023, arXiv preprint). Moreover, we consider the Sobolev critical case and $L^2$-critical case of the above problem.
\end{abstract}

{\footnotesize {\bf   Keywords:}  Sch\"odinger equations; Normalized solutions;  Variational methods, Mixed nonlinearity.

}

\section{ Introduction and main results}

This paper studies the existence of normalized solutions for the following Schr\"odinger equations with potentials and general nonlinearities
\begin{equation} \label{eq1.1}
 \left\{\aligned
& -\Delta u+V(x)u+\lambda u=|u|^{q-2}u+\beta f(u), &x\in \Omega, \\
& \int_{\Omega} |u|^2dx=\Theta,u\in H_0^1(\Omega), &x\in \Omega,
\endaligned
\right.
\end{equation}
where $\Omega \subset \mathbb{R}^N$ is either all of $\mathbb{R}^N$ or a bounded smooth convex domain, $N \geq 3$, $2+\frac{4}{N}\leq q\leq2^*=\frac{2 N}{N-2}$, the mass $\Theta>0$ and the parameter $\beta \in \mathbb{R}$ are prescribed. The frequency $\lambda$ is unknown and to be determined.

Such problems are motivated in particular by searching for solitary waves (stationary states) in nonlinear equations of the Schr\"odinger type. Specifically, consider the following nonlinear Schr\"odinger equation
\begin{equation*}
 \left\{\aligned
&-i \frac{\partial}{\partial t} \Psi=\Delta \Psi-V(x) \Psi+f\left(|\Psi|^2\right) \Psi=0,&(x, t) \in \mathbb{R}^N \times \mathbb{R}, \\
&\Psi=\Psi(x,t),&(x, t) \in \mathbb{C},
\endaligned
\right.
\end{equation*}
where $N \geq 1$. Researchers are interested in finding the existence of standing wave solutions to the above equations, that is, $\Psi(x, t)=$ $e^{i \lambda t} u(x), \lambda \in \mathbb{R}$, and $u: \mathbb{R}^N \rightarrow \mathbb{R}$, so we get the equation
$$
-\Delta u+(V(x)+\lambda) u=Q(u),\ x\in\mathbb{R}^N,
$$
where $Q(u)=f(|u|^2)u$. For physical reasons, we focus on the existence of normalized solutions for the following problem
\begin{equation} \label{eq1.2}
 \left\{\aligned
&-\Delta u+(V(x)+\lambda) u=Q(u),\ x\in \mathbb{R}^N, \\
& \int_{\mathbb{R}^N} |u|^2dx=\Theta,\ x\in \mathbb{R}^N.
\endaligned
\right.
\end{equation}
For more physical background about the above equation, please refer to \cite{{DCQN1999},{LE2010}}.

If potential $V(x)$ in \eqref{eq1.2} is constant, we call \eqref{eq1.2} is autonomous. In this case, recalling paper \cite{LJJ1997}, Jeanjean developed an approach based on the
Pohozaev identity which has been used successfully in recent years. The key to this method is to find a bounded Palais-Smale sequences by using the transformation $s*u(x) = e^{\frac{sN}{2}}u(e^sx)$. After that, by weakening the conditions in \cite{LJJ1997}, Jeanjean \cite{LJSL2020} and Bieganowski \cite{BBJM2021} improved these results. Of course, these articles only consider the problem of a single nonlinear term. Recently, there have been many studies on mixed nonlinear terms. For example, Soave \cite{{NSJDE2020},{NSJFA2020}} studied normalized solution of \eqref{eq1.2}
with mixed nonlinearity $f(|u|)u = \mu|u|^{q-2}u + |u|^{p-2}u$, $2 < p < 2 + \frac{4}{N} < q \leq2^*=\frac{2N}{N-2}$. Specifically, Soave in \cite{NSJDE2020} obtained many results of existence and non-existence. More precisely, if $2<q<p=2+\frac{4}{N}$, that is, the
leading nonlinearity is $L^2$-critical and a $L^2$-subcritical lower order term. \eqref{eq1.2} had a real-valued positive and radially symmetric solution for some $\lambda<0$ in $\mathbb{R}^N$ provided $\mu>0$ and $\Theta>0$ small enough. Moreover, if $\mu<0$, \eqref{eq1.2} had no solution. If $2+\frac{4}{N}=q<p<2^*$, that is, the
leading term is $L^2$-critical and $L^2$-supercritical, \eqref{eq1.2} had a real-valued positive, radially symmetric solution for some $\lambda<0$ in $\mathbb{R}^N$ provided $\mu>0$ and $\mu, \Theta$ satisfy the appropriate conditions. If $2<q< 2+\frac{4}{N}<p<2^*$, that is, the
leading term $L^2$-subcritical and $L^2$-supercritical, \eqref{eq1.2} also had a real-valued positive and radially symmetric solution for some $\lambda<0$ in $\mathbb{R}^N$ provided $\Theta>0, \mu<0$ and $\mu, \Theta$ satisfy the appropriate conditions. Soave in \cite{NSJFA2020} considered the Sobolev critical case and obtained some similar results. In particular, the Sobolev critical case also has been considered in \cite{{TAK2012},{TAK2013},{RKTO2017},{CMGX2013}}(see also the references therein). It is worth mentioning that many researchers are also interested in the existence of normalized multiple solutions. In \cite{LJTL2022}, Jeanjean et al. obtained the existence of normalized multiple solutions for Sobolev critical case in \eqref{eq1.2}. For more results on this aspect, please refer to \cite{{JWYW2022},{TBNS2017},{TBNS2019},{DBJC2019},{JBLT2013}} and its references.

If \eqref{eq1.2} is non-autonomous, Ikoma and Miyamoto in \cite{NIYM2020} considered question \eqref{eq1.2} with $V(x)\in C(\mathbb{R}^N), 0\not\equiv V(x)\leq0, V(x) \rightarrow0 (|x|\rightarrow\infty)$, they obtained some existence and non-existence results. After that, Ding and Zhong in \cite{YDXZ2022} proved the existence of normalized solutions to the following Schr\"odinger equation
$$
\left\{\begin{array}{l}
-\Delta u(x)+V(x) u(x)+\lambda u(x)=g(u(x)), \quad x\in\mathbb{R}^N, \\
0 \leq u(x) \in H^1(\mathbb{R}^N), N \geq 3,
\end{array}\right.
$$
where $g$ satisfies:
\begin{itemize}
\item[$(G1)$]\ $g: \mathbb{R} \rightarrow \mathbb{R}$ is continuous and odd.

\item[$(G2)$]\ There exists some $(\alpha, \beta) \in \mathbb{R}_{+}^2$ satisfying $2+\frac{4}{N}<\alpha \leq \beta<\frac{2 N}{N-2}$ such that
$$
\alpha G(s) \leq g(s) s \leq \beta G(s) \text { with } G(s)=\int_0^s g(t) d t .
$$

\item[$(G3)$]\ The functional defined by $\widetilde{G}(s):=\frac{1}{2} g(s) s-G(s)$ is of class $C^1$ and
$$
\widetilde{G}^{\prime}(s) s \geq \alpha \widetilde{G}(s), \forall s \in \mathbb{R},
$$
where $\alpha$ is given by $(G2)$.
\end{itemize}
Note that, $(G3)$ plays a crucial role in the uniqueness of $t_u$(see \cite{YDXZ2022} or \cite[Lemma 2.9]{{LJJ1997}}). However, we do not need this condition, since we directly perform scaling and complex calculations on energy functionals. Recently, Bartsch et al. in \cite{TBAQ2023} considered following Schr\"odinger equations with potentials and inhomogeneous nonlinearities on large convex domains
\begin{equation*}
 \left\{\aligned
& -\Delta u+V(x)u+\lambda u=|u|^{q-2}u+\beta|u|^{p-2}u,  \\
& \int |u|^2dx=\Theta,
\endaligned
\right.
\end{equation*}
they developed a robust method to study the existence of normalized solutions of nonlinear Schr\"odinger equations with potential. Under the stimulation of \cite{TBAQ2023}, our goal is to generalize its conclusion to general nonlinear terms and the Sobolev critical case.

In order to state our main results, we introduce some notations. Set $s_{+}=\max \{s, 0\}$, $s_{-}=\min \{s, 0\}$ for $s \in \mathbb{R}$. The Aubin-Talenti constant \cite{TA1976} is denoted by $S$, that is, $S$ is the best constant in the Sobolev embedding $\mathcal{D}^{1,2}(\mathbb{R}^N) \hookrightarrow L^{2^*}(\mathbb{R}^N)$, where $\mathcal{D}^{1,2}(\mathbb{R}^N)$ denotes the completion of $C_c^{\infty}(\mathbb{R}^N)$ with respect to the norm $\|u\|_{\mathcal{D}^{1,2}}:=$ $\|\nabla u\|_2$. It is well known \cite{GT1976} that the optimal constant is achieved by (any multiple of)
\begin{equation}\label{eq1.3}
  U_{\varepsilon, y}(x)=[N(N-2)]^{\frac{N-2}{4}}\left(\frac{\varepsilon}{\varepsilon^2+|x-y|^2}\right)^{\frac{N-2}{2}},\ \varepsilon>0,\ y \in \mathbb{R}^N,
\end{equation}
which are the only positive classical solutions to the critical Lane-Emden equation
$$
-\Delta w=w^{2^*-1}, \quad w>0 \quad \text { in } \mathbb{R}^N .
$$

Let $C_{N, s}$ be the best constant in the Gagliardo-Nirenberg inequality
$$
\|u\|_s^s \leq C_{N, s}\|u\|_2^{\frac{2 s-N(s-2)}{2}}\|\nabla u\|_2^{\frac{N(s-2)}{2}},\ 2<s<2^* .
$$

For some results, we expect that $V$ is $C^1$ and consider the function
$$
\widetilde{V}: \mathbb{R}^N \rightarrow \mathbb{R}, \quad \widetilde{V}(x)=\nabla V(x) \cdot x
$$
For $\Omega \subset \mathbb{R}^N$ and $r>0$, let
$$
\Omega_r=\left\{r x \in \mathbb{R}^N: x \in \Omega\right\}
$$
and
$$
S_{r, \Theta}:=S_\Theta \cap H_0^1(\Omega_r)=\left\{u \in H_0^1(\Omega_r):\|u\|_{L^2(\Omega_r)}^2=\Theta\right\} .
$$
From now on we assume that $\Omega \subset \mathbb{R}^N$ is a bounded smooth convex domain with $0 \in \Omega$.

Our assumptions on $V$ are:
\begin{itemize}
\item[$(V_0)$]\ $ V \in C^1(\mathbb{R}^N) \cap L^{\frac{N}{2}}(\mathbb{R}^N)$ is bounded and $\left\|V_{-}\right\|_{\frac{N}{2}}<S$.

\item[$(\widetilde{V_0})$]\ $ V \in C^1(\mathbb{R}^N) \cap L^{\frac{N}{2}}(\mathbb{R}^N)$ is bounded and $\left\|V_{-}\right\|_{\frac{N}{2}}<\frac{N(q-p_2)-2[N(p_2-2)-4]}{N(q-p_2)}S$.

\item[$(\widehat{V_0})$]\ $ V \in C^1(\mathbb{R}^N) \cap L^{\frac{N}{2}}(\mathbb{R}^N)$ is bounded and $\left\|V_{-}\right\|_{\frac{N}{2}}<\left(1-\frac{NC_N\Theta^{\frac{2}{N}}}{N+2}\right)S$.

\item[$(V_1)$]\ $V$ is of class $C^1, \lim\limits_{|x| \rightarrow \infty} V(x)=0$, and there exists $\rho \in(0,1)$ such that
$$
\liminf\limits_{|x| \rightarrow \infty} \inf _{y \in B(x, \rho|x|)}(x \cdot \nabla V(y)) e^{\tau|x|}>0 \quad \text { for any } \tau>0.
$$
\end{itemize}
\begin{remark}\label{R1.1}
In order to obtain the existence of normalized solutions in $\mathbb{R}^N$ by taking $\Omega=B_1$, the unit ball centered at the origin in $\mathbb{R}^N$, and analyzing the compactness of the solutions $u_{r, \Theta}$ established in Theorems \ref{t1.1}, \ref{t1.2} and \ref{t1.3} as $r$ tends to infinity, we require the condition $(V_1)$.
\end{remark}

Now, we make the following assumptions on the nonlinearity $f$:
\begin{itemize}
\item[$(f_1)$]\ $f\in C^1(\mathbb{R}, \mathbb{R})$ and $f$ is odd.

\item[$(f_2)$]\ There exists some $(p_1, p_2) \in \mathbb{R}_{+}^2$ satisfying $2< p_2 \leq p_1<2+\frac{4}{N}$ such that
$$
p_2 F(\tau) \leq f(\tau) \tau \leq p_1 F(\tau) \text { with } F(\tau)=\int_0^\tau f(t) d t .
$$

\item[($\widetilde{f_2})$]\ There exists some $(p_1, p_2) \in \mathbb{R}_{+}^2$ satisfying $2<p_2 < p_1=2+\frac{4}{N}$ such that
$$
p_2 F(\tau) \leq f(\tau) \tau \leq p_1 F(\tau).
$$
\end{itemize}
\begin{remark}\label{R1.2}
If $f(u)=\sum\limits_{i=1}^ma_i|u|^{\sigma_i-2}u$, where $a_i>0$ and $2<\sigma_i<2+\frac{4}{N}$, then the assumption $(f_1)$ can be weakened to $f\in C(\mathbb{R}, \mathbb{R})$ and $f$ is odd. In order to ensure the boundedness of Palais-Smale sequence under constraint conditions in Lemma \ref{L3.2}, we need to slightly strengthen the conditions for the nonlinear term $f$, that is, $f\in C^1(\mathbb{R}, \mathbb{R})$.
\end{remark}

The main results of this paper are as follows. Firstly, we consider the Sobolev subcritical case, that is, $2+\frac{4}{N}<q<2^*$.
\begin{theorem}\label{t1.1}(\textbf{case $\beta \leq 0$})
Assume $V$ satisfies $\left(V_0\right)$, is of class $C^1$ and $\widetilde{V}$ is bounded, $f$ satisfies $(f_1)-(f_2)$. There hold:

(i) For every $\Theta>0$, there exists $r_\Theta>0$ such that \eqref{eq1.1} on $\Omega_r$ with $r>r_\Theta$ has a mountain pass type solution $\left(\lambda_{r, \Theta}, u_{r, \Theta}\right)$ with $u_{r, \Theta}>0$ in $\Omega_r$ and positive energy $I_r\left(u_{r, \Theta}\right)>0$. Moreover, there exists $C_\Theta>0$ such that
$$
\limsup _{r \rightarrow \infty} \max _{x \in \Omega_r} u_{r, \Theta}(x)<C_\Theta .
$$

(ii) If in addition $\|\widetilde{V}_{+}\|_{\frac{N}{2}}<2 S$, then there exists $\widetilde{\Theta}>0$ such that
$$
\liminf _{r \rightarrow \infty} \lambda_{r, \Theta}>0 \text { for any } 0<\Theta<\widetilde{\Theta} .
$$
\end{theorem}

\begin{theorem}\label{t1.2}(\textbf{case $\beta > 0$})
Assume $V$ satisfies $\left(V_0\right)$, $f$ satisfies $(f_1)-(f_2)$ and set
$$
\Theta_V=\left[\frac{1-\left\|V_{-}\right\|_{\frac{N}{2}} S^{-1}}{2N(q-p_1)}\right]^{\frac{N}{2}}\left[\frac{q(4-N(p_1-2))}{C_{N, q}}\right]^{\frac{4-N(p_1-2)}{2 (q-p_1)}}\left[\frac{N(q-2)-4}{\alpha\beta C_{N, p_1}}\right]^{\frac{N(q-2)-4}{2 (q-p_1)}}.
$$
Then the following hold for $0<\Theta<\Theta_V$:

(i) There exists $r_\Theta>0$ such that \eqref{eq1.1} on $\Omega_r$ with $r>r_\Theta$ has a local minimum type solution $\left(\lambda_{r, \Theta}, u_{r, \Theta}\right)$ with $u_{r, \Theta}>0$ in $\Omega_r$ and negative energy $I_r\left(u_{r, \Theta}\right)<0$.

(ii) There exists $C_\Theta>0$ such that
$$
\limsup\limits_{r \rightarrow \infty} \max\limits_{x \in \Omega_r} u_{r, \Theta}(x)<C_\Theta, \quad \liminf\limits_{r \rightarrow \infty} \lambda_{r, \Theta}>0 .
$$
\end{theorem}

\begin{theorem}\label{t1.3}(\textbf{case $\beta > 0$})
Assume $V$ satisfies $\left(V_0\right)$, is of class $C^1$ and $\widetilde{V}$ is bounded, $f$ satisfies $(f_1)-(f_2)$. Set
$$
\widetilde{\Theta}_V=\frac{1}{2}\left(1-\left\|V_{-}\right\|_{\frac{N}{2}} S^{-1}\right)^{\frac{N}{2}}\left(\frac{C_{N,q}}{q}A_{p_1,q}+\frac{C_{N, q}}{q}\right)^{-\frac{N}{2}}\left(\frac{\alpha\beta q C_{N, p_1}}{C_{N, q}A_{p_1,q}}\right)^{\frac{N(q-2)-4}{2N(q-p_1)}},
$$
where
$$
A_{p_1, q}=\frac{(q-2)(N(q-2)-4)}{(p_1-2)(4-N(p_1-2))}.
$$
Then the following hold for $0<\Theta<\widetilde{\Theta}_V$:

(i) There exists $\widetilde{r}_\Theta>0$ such that \eqref{eq1.1} in $\Omega_r$ admits for $r>r_\Theta$ a mountain pass type solution $\left(\lambda_{r, \Theta}, u_{r, \Theta}\right)$ with $u_{r, \Theta}>0$ in $\Omega_r$ and positive energy $I_r\left(u_{r, \Theta}\right)>0$. Moreover, there exists $C_\Theta>0$ such that
$$
\limsup\limits_{r \rightarrow \infty} \max\limits_{x \in \Omega_r} u_{r, \Theta}(x)<C_\Theta .
$$

(ii) There exists $0<\bar{\Theta} \leq \widetilde{\Theta}_V$ such that
$$
\liminf\limits_{r \rightarrow \infty} \lambda_{r, \Theta}>0 \quad \text { for any } 0<\Theta \leq \bar{\Theta}.
$$
\end{theorem}

If $\Omega=\mathbb{R}^N$, $(V_1)$ is significant for obtaining the following results.

\begin{theorem}\label{t1.4}(\textbf{case $\beta > 0$})
Assume $V$ satisfies $(V_0)-(V_1)$. Then problem \eqref{eq1.1} with $\Omega=\mathbb{R}^N$ admits for any $0<\Theta<\Theta_V$, where $\Theta_V$ is as in Theorem \ref{t1.2}, a solution $\left(\lambda_\Theta, u_\Theta\right)$ with $u_\Theta>0, \lambda_\Theta>0$, and $I\left(u_\Theta\right)<0$.
\end{theorem}

\begin{theorem}\label{t1.5}(\textbf{case $\beta > 0$})
Assume $V$ satisfies $(V_0)-(V_1)$. Then \eqref{eq1.1} with $\Omega=\mathbb{R}^N$ admits for $0<\Theta<\bar{\Theta}, \bar{\Theta}>0$ as in Theorem \ref{t1.3} (ii), a solution $(\lambda_\Theta, u_\Theta)$ with $u_\Theta>0, \lambda_\Theta>0$, and $I(u_\Theta)>0$. Moreover, $\lim\limits_{\Theta \rightarrow 0} I(u_\Theta)=\infty$.
\end{theorem}

\begin{theorem}\label{t1.6}(\textbf{case $\beta \leq 0$})
Assume $V$ satisfies $(V_0)-(V_1)$, and $\|\widetilde{V}_{+}\|_{\frac{N}{2}}<2 S$. Then problem \eqref{eq1.1} with $\Omega=\mathbb{R}^N$ admits for $0<\Theta<\widetilde{\Theta}, \widetilde{\Theta}>0$ as in Theorem \ref{t1.1}, a solution $(\lambda_\Theta, u_\Theta)$ with $u_\Theta>0, \lambda_\Theta>0$, and $I(u_\Theta)>0$. Moreover, $\lim\limits_{\Theta\rightarrow 0} I\left(u_\Theta\right)=\infty$.
\end{theorem}
For the Sobolev critical case, that is $q=2^*$, we have the following results.

\begin{theorem}\label{t1.7}(\textbf{case $\beta > 0$})
Assume $V$ satisfies $\left(V_0\right)$, $f$ satisfies $(f_1)-(f_2)$. Set
$$
\Theta_V=\left( \frac{1}{N\alpha\beta C_{N, p_1}}\right)^{\frac{4}{2 p_1-N(p_1-2)}}\left(1-\left\|V_{-}\right\|_{\frac{N}{2}} S^{-1}\right)^{\frac{N}{2}}S^{\frac{N}{2}\cdot\frac{4-N(p_1-2)}{2 p_1-N(p_1-2)}}.
$$
Then the following hold for $0<\Theta<\Theta_V$:

(i) There exists $r_\Theta>0$ such that \eqref{eq7.2} on $\Omega_r$ with $r>r_\Theta$ has a local minimum type solution $\left(\lambda_{r, \Theta}, u_{r, \Theta}\right)$ with $u_{r, \Theta}>0$ in $\Omega_r$ and negative energy $\mathcal{I}_r\left(u_{r, \Theta}\right)<0$.

(ii) There exists $C_\Theta>0$ such that
$$
\limsup\limits_{r \rightarrow \infty} \max\limits_{x \in \Omega_r} u_{r, \Theta}(x)<C_\Theta, \quad \liminf\limits_{r \rightarrow \infty} \lambda_{r, \Theta}>0 .
$$
\end{theorem}

\begin{theorem}\label{t1.8}(\textbf{case $\beta \leq 0$})
Assume $V$ satisfies $\left(V_0\right)$, is of class $C^1$ and $\widetilde{V}$ is bounded, $f$ satisfies $(f_1)-(f_2)$. There hold:

(i) There exists $r_\Theta>0$ such that \eqref{eeq8.1} on $\Omega_r$ with $r>r_\Theta$ has a mountain pass type solution $\left(\lambda_{r, \Theta}, u_{r, \Theta}\right)$ with $u_{r, \Theta}\geq0$ in $\Omega_r$ and positive energy $\mathcal{I}_r\left(u_{r, \Theta}\right)>0$.

(ii) There exists $C_\Theta>0$ such that
$$
\limsup _{r \rightarrow \infty} \max _{x \in \Omega_r} u_{r, \Theta}(x)<C_\Theta .
$$
\end{theorem}

\begin{theorem}\label{t1.9}(\textbf{case $\beta > 0$})
Assume $V$ satisfies $\left(V_0\right)$, is of class $C^1$ and $\widetilde{V}$ is bounded, $f$ satisfies $(f_1)-(f_2)$. Set
$$
\widetilde{\Theta}_V=  \left(\frac{\alpha\beta C_{N, p_1}S^{\frac{2^*}{2}}}{A_{p_1}}\right)^{-\frac{4}{2p_1-N(p_1-2)}}\left[\frac{S^{\frac{2^*}{2}}}{2\cdot2^*}\left(1-\left\|V_{-}\right\|_{\frac{N}{2}} S^{-1}\right)(2^*A_{p_1}+1)\right]^{\frac{2[2\cdot2^*-N(p_1-2)]}{(2^*-2)[2p_1-N(p_1-2)]}}
$$
where
$$
A_{p_1 }=\frac{4(2^*-2)}{N(p_1-2)(4-N(p_1-2))}.
$$
Then the following hold for $0<\Theta<\widetilde{\Theta}_V$:

(i) There exists $\widetilde{r}_\Theta>0$ such that \eqref{eq9.1} in $\Omega_r$ admits for $r>r_\Theta$ a mountain pass type solution $\left(\lambda_{r, \Theta}, u_{r, \Theta}\right)$ with $u_{r, \Theta}\geq0$ in $\Omega_r$ and positive energy $\mathcal{I}_r\left(u_{r, \Theta}\right)>0$.

(ii) There exists $C_\Theta>0$ such that
$$
\limsup\limits_{r \rightarrow \infty} \max\limits_{x \in \Omega_r} u_{r, \Theta}(x)<C_\Theta .
$$
\end{theorem}

For the $L^2$-critical case, that is $p_1=2+\frac{4}{N}$ or $q=2+\frac{4}{N}$, we have the following results.

\begin{theorem}\label{t1.10}(\textbf{case $\beta > 0$ and $p_1=2+\frac{4}{N}$})
Assume $V$ satisfies $(\widetilde{V_0})$, $f$ satisfies $(f_1)$ and $(\widetilde{f_2})$. Set
$$
\widetilde{\Theta}_V=\left[\frac{N(q-p_2)-4}{N\alpha\beta(q-p_2)C_N}\right]^{\frac{N}{2}}.
$$
Then the following hold for $0<\Theta<\widetilde{\Theta}_V$:

(i) There exists $\widetilde{r}_\Theta>0$ such that \eqref{eq1.1} in $\Omega_r$ admits for $r>r_\Theta$ a mountain pass type solution $\left(\lambda_{r, \Theta}, u_{r, \Theta}\right)$ with $u_{r, \Theta}>0$ in $\Omega_r$ and positive energy $I_r\left(u_{r, \Theta}\right)>0$. Moreover, there exists $C_\Theta>0$ such that
$$
\limsup\limits_{r \rightarrow \infty} \max\limits_{x \in \Omega_r} u_{r, \Theta}(x)<C_\Theta .
$$

(ii) There exists $0<\bar{\Theta} \leq \widetilde{\Theta}_V$ such that
$$
\liminf\limits_{r \rightarrow \infty} \lambda_{r, \Theta}>0 \quad \text { for any } 0<\Theta \leq \bar{\Theta}.
$$
\end{theorem}

\begin{theorem}\label{t1.11}(\textbf{case $\beta \leq 0$ and $p_1=2+\frac{4}{N}$})
Assume $V$ satisfies $\left(V_0\right)$, is of class $C^1$ and $\widetilde{V}$ is bounded, $f$ satisfies $(f_1)$ and $(\widetilde{f_2})$. Set
$$
\widehat{\Theta}_V=\left[\frac{(N-2)q-2N}{2N\alpha\beta(q-p_2)C_N}\right]^{\frac{N}{2}}.
$$
Then the following hold for $0<\Theta<\widehat{\Theta}_V$:

(i) There exists $r_\Theta>0$ such that \eqref{eq1.1} on $\Omega_r$ with $r>r_\Theta$ has a mountain pass type solution $\left(\lambda_{r, \Theta}, u_{r, \Theta}\right)$ with $u_{r, \Theta}>0$ in $\Omega_r$ and positive energy $I_r\left(u_{r, \Theta}\right)>0$. Moreover, there exists $C_\Theta>0$ such that
$$
\limsup _{r \rightarrow \infty} \max _{x \in \Omega_r} u_{r, \Theta}(x)<C_\Theta .
$$

(ii) If in addition $\|\widetilde{V}_{+}\|_{\frac{N}{2}}<2 S$, then there exists $\widetilde{\Theta}>0$ such that
$$
\liminf _{r \rightarrow \infty} \lambda_{r, \Theta}>0 \text { for any } 0<\Theta<\widetilde{\Theta} .
$$
\end{theorem}

\begin{theorem}\label{t1.12}(\textbf{case $\beta > 0$ and $q =2+\frac{4}{N}$})
Assume $V$ satisfies $(\widehat{V_0})$, $f$ satisfies $(f_1)-(f_2)$ and set
$$
\Theta_V=\left(\frac{N+2}{NC_N}\right)^{\frac{N}{2}}.
$$
Then the following hold for $0<\Theta<\Theta_V$:

(i) There exists $r_\Theta>0$ such that \eqref{eq1.1} on $\Omega_r$ with $r>r_\Theta$ has a global minimum type solution $\left(\lambda_{r, \Theta}, u_{r, \Theta}\right)$ with $u_{r, \Theta}>0$ in $\Omega_r$ and negative energy $I_r\left(u_{r, \Theta}\right)<0$.

(ii) There exists $C_\Theta>0$ such that
$$
\limsup\limits_{r \rightarrow \infty} \max\limits_{x \in \Omega_r} u_{r, \Theta}(x)<C_\Theta, \quad \liminf\limits_{r \rightarrow \infty} \lambda_{r, \Theta}>0 .
$$
\end{theorem}

\begin{remark}\label{R1.3}
(i) Theorems \ref{t1.1}-\ref{t1.9} are valid if $2=p_2<p_1<2+\frac{4}{N}$ in $(f_2)$. Moreover, the proof of Theorems \ref{t1.4}-\ref{t1.6} is very similar to \cite{TBAQ2023}, so we omit it in this paper.

(ii) Our conclusion also applies to $p_1=p_2=2+\frac{4}{N}$ if $2+\frac{4}{N}<q<2^*$, such as $f(u)=|u|^{\frac{4}{N}}u$. Therefore, our results cover certain conclusions in \cite{NSJDE2020}.
\end{remark}
\begin{remark}\label{R1.4}
Theorems \ref{t1.2} and \ref{t1.7}(resp. Theorems \ref{t1.3} and \ref{t1.9}) both require some limitations on $\Theta_V$(resp. $\widetilde{\Theta}_V$), although their values are different, they all stem from changes in the geometric structure of the energy functional. In addition, there are still some unknown results for the Sobolev critical case, that is, $\liminf\limits_{r \rightarrow \infty} \lambda_{r, \Theta}>0$ may not necessarily hold when $\beta > 0$ or $\beta\leq0$. In fact, the methods and techniques in Theorem \ref{t1.2}(or Theorem \ref{t1.3}) cannot be applied to the Sobolev critical case since $\frac{(N-2) q-2 N}{2 N q}=0$, thus $\lambda_\Theta>0$ cannot be obtained $0<\Theta\leq\bar{\Theta}$.
\end{remark}

\begin{remark}\label{R1.5}
In this paper, whether in subcritical or critical situations, the monotonicity trick in \cite{LJJ1999} is one of the keys to get the conclusion. Proposition \ref{p3.1} does not ensure the existence of a mountain pass solution for the original problem obtained when $s=1$. However, it gives the existence of a sequence $s_n\rightarrow1^{-}$, with a corresponding sequence of
mountain pass critical points $u_{r,s_n}\in$ of $I_{r,s_n}$, constrained on $S_{r, \Theta}$. We aim to show that $u_{r,s_n}$ strongly converges to a constrained critical point of $I_r$. For this purpose, it is sufficient to prove that $u_{r,s_n}$ is bounded in $H_0^1(\Omega_r)$, thanks to Proposition 3.1 in \cite{SD2018}.
\end{remark}

The structure of this paper is arranged as follows. In section 2, we provide some ideas in the proof of main theorems. In the following three sections, we aim to prove the first three conclusions. After that, we consider the Sobolev critical case. Finally, we consider the $L^2$-critical case and give some comments.

\section{Preliminary}
Consider the problem
\begin{equation} \label{eq2.1}
 \left\{\aligned
& -\Delta u+V(x)u+\lambda u=|u|^{q-2}u+\beta f(u), &x\in \Omega_r, \\
& \int_{\Omega_r} |u|^2dx=\Theta,u\in H_0^1(\Omega_r), &x\in \Omega_r,
\endaligned
\right.
\end{equation}
where $N \geq 3$, $2+\frac{4}{N}\leq q\leq2^*=\frac{2 N}{N-2}$, the mass $\Theta>0$ and the parameter $\beta \in \mathbb{R}$ are prescribed. The frequency $\lambda$ is unknown and to be determined. The energy functional $I_r: H_0^1(\Omega_r) \rightarrow \mathbb{R}$ is defined by
\begin{equation} \label{eq2.2}
I_r(u)=\frac{1}{2} \int_{\Omega_r}|\nabla u|^2 d x+\frac{1}{2} \int_{\Omega_r} V(x) u^2 d x-\frac{1}{q} \int_{\Omega_r}|u|^q d x-\beta \int_{\Omega_r}F(u) d x
\end{equation}
and the mass constraint manifold is defined by
\begin{equation} \label{eq2.3}
S_{r, \Theta}=\left\{u \in H_0^1(\Omega_r):\|u\|_2^2=\Theta\right\}.
\end{equation}
If $\Omega=\mathbb{R}^N$, the energy functional $I: H_0^1\left(\Omega_r\right) \rightarrow \mathbb{R}$ is defined by
\begin{equation} \label{eq2.4}
I(u)=\frac{1}{2} \int_{\mathbb{R}^N}|\nabla u|^2 d x+\frac{1}{2} \int_{\mathbb{R}^N} V(x) u^2 d x-\frac{1}{q} \int_{\mathbb{R}^N}|u|^q d x-\beta \int_{\mathbb{R}^N}F(u) d x
\end{equation}
and the mass constraint manifold is defined by
\begin{equation} \label{eq2.5}
S_{\Theta}=\left\{u \in H_0^1(\mathbb{R}^N):\|u\|_2^2=\Theta\right\}.
\end{equation}

The proof idea of Theorem \ref{t1.1} is as follows. In order to find a mountain pass type solution $(\lambda_{r, \Theta}, u_{r, \Theta})$, we first need to analyze the geometric structure of the energy functional corresponding to the equation \eqref{eq3.1}. Note that $\beta>0$, there exist $0<T_\Theta<t_0$ such that $h\left(t_0\right)=0, h(t)<0$ for any $t>t_0, h(t)>0$ for any $0<t<t_0$ and $h\left(T_\Theta\right)=\max\limits_{t \in \mathbb{R}^{+}} h(t)$ in Lemma \ref{L3.1}. This indicates that the energy functional $I_{r,s}$ has a global maximum. Next, we obtain the bounded Palais-Smale sequence by using \cite[Theorem 1]{{JBXC2022}} and get a solution $\left(\lambda_{r, s}, u_{r, s}\right)$ for \eqref{eq3.1}. Finally, we consider Lagrange multiplier and establish an a priori estimate for the solutions of \eqref{eq1.1}. Theorem \ref{t1.2} is relatively simple because the energy functional has a local minimum, which can be proved using the method of constrained minimization. The proof of Theorem \ref{t1.3} is similar to Theorem \ref{t1.1}, but the geometric structures of the two cases are significantly different and require refined estimate of energy.

Note that, there are some differences between the proof of Lemma \ref{L5.3} and Lemma \ref{L3.3}, and we cannot directly use the method of Lemma \ref{L3.3}, even if $q$ can be reduced to $p_2$ according to condition $(f_2)$ and $p_2<2+\frac{4}{N}<q<2^*$. More precisely, it then follows from $\beta >0$ and $(f_2)$ that
\begin{eqnarray*}
&&\frac{1}{N} \int_{\Omega_r}|\nabla u|^2 d  x-\frac{1}{2 N} \int_{\partial \Omega_r}|\nabla u|^2(x \cdot \mathbf{n}) d \sigma-\frac{1}{2 N} \int_{\Omega_r}(\nabla V \cdot x) u^2 d x \\
& =&\frac{(q-2) s}{2 q} \int_{\Omega_r}|u|^q d x+ s\int_{\Omega_r}(\frac{\beta }{2  }f(u)u-\beta F(u))d x \\
& \geq& \frac{p_2-2}{2}\left(\frac{1}{2} \int_{\Omega_r}|\nabla u|^2 d x+\frac{1}{2} \int_{\Omega_r} V u^2 d x-m_{r, s}(\Theta)\right) .
\end{eqnarray*}
Consequently, we have
 \begin{eqnarray*}
\frac{p_2-2}{2} m_{r, s}(\Theta) &\geq & \frac{p_2-2}{2}\left(\frac{1}{2} \int_{\Omega_r}|\nabla u|^2 d x+\frac{1}{2} \int_{\Omega_r} V u^2 d x\right)-\frac{1}{N} \int_{\Omega_r}|\nabla u|^2 d x \\
&& +\frac{1}{2 N} \int_{\partial \Omega_r}|\nabla u|^2(x \cdot \mathbf{n}) d \sigma+\frac{1}{2 N} \int_{\Omega_r}(\nabla V \cdot x) u^2 d x \\
&\geq & \frac{N(p_2-2)-4}{4 N} \int_{\Omega_r}|\nabla u|^2 d x-\Theta\left(\frac{1}{2 N}\|\nabla V \cdot x\|_{\infty}+\frac{p_2-2}{4}\|V\|_{\infty}\right).
 \end{eqnarray*}
However, this method is useless because $\frac{N(p_2-2)-4}{4 N}<0$, we cannot obtain that $\int_{\Omega_r}|\nabla u|^2 d x$ is uniformly bounded in $s$ and $r$.

\section{Proof of Theorem \ref{t1.1}}
 In this section, we assume $\beta \leq 0$ and the assumptions of Theorem \ref{t1.1} hold. In order to obtain a bounded Palais-Smale sequence, we will use the monotonicity trick inspired by \cite{LJJ1999}. For $\frac{1}{2} \leq s \leq 1$, we define the functional $I_{r, s}: S_{r, \Theta} \rightarrow \mathbb{R}$ by
\begin{equation}\label{eq3.1}
  I_{r, s}(u)=\frac{1}{2} \int_{\Omega_r}|\nabla u|^2 d x+\frac{1}{2} \int_{\Omega_r} V  u^2 d x-\frac{s}{q} \int_{\Omega_r}|u|^q d x-\beta \int_{\Omega_r}F(u) d x .
\end{equation}
Note that if $u \in S_{r, \Theta}$ is a critical point of $I_{r, s}$, then there exists $\lambda \in \mathbb{R}$ such that $(\lambda, u)$ is a solution of the equation
\begin{equation}\label{eq3.2}
 \left\{\aligned
& -\Delta u+V u+\lambda u=s|u|^{q-2}u+\beta f(u), &x\in \Omega_r, \\
& \int_{\Omega_r} |u|^2dx=\Theta,u\in H_0^1(\Omega_r), &x\in \Omega_r.
\endaligned
\right.
\end{equation}

\begin{lemma}\label{L3.1}
For any $\Theta>0$, there exist $r_\Theta>0$ and $u^0, u^1 \in S_{r_\Theta, \Theta}$ such that

(i) $I_{r, s}(u^1) \leq 0$ for any $r>r_\Theta$ and $s \in\left[\frac{1}{2}, 1\right]$,
$$
\left\|\nabla u^0\right\|_2^2<\left[\frac{2 q}{N(q-2) C_{N, q}}\left(1-\left\|V_{-}\right\|_{\frac{N}{2}} S^{-1}\right) \Theta^{\frac{q(N-2)-2 N}{4}}\right]^{\frac{4}{N(q-2)-4}}<\left\|\nabla u^1\right\|_2^2
$$
and
$$
I_{r, s}\left(u^0\right)<\frac{(N(q-2)-4)\left(1-\left\|V_{-}\right\|_{\frac{N}{2}} S^{-1}\right)}{2 N(q-2)}\left[\frac{2 q\left(1-\left\|V_{-}\right\|_{\frac{N}{2}} S^{-1}\right)}{N(q-2) C_{N, q} \Theta^{\frac{2 q-N(q-2)}{4}}}\right]^{\frac{4}{N(q-2)-4}} .
$$

(ii) If $u \in S_{r, \Theta}$ satisfies
$$
\|\nabla u\|_2^2=\left[\frac{2 q}{N(q-2) C_{N, q}}\left(1-\left\|V_{-}\right\|_{\frac{N}{2}} S^{-1}\right) \Theta^{\frac{q(N-2)-2 N}{4}}\right]^{\frac{4}{N(q-2)-4}},
$$
then there holds
$$
I_{r, s}(u) \geq \frac{(N(q-2)-4)\left(1-\left\|V_{-}\right\|_{\frac{N}{2}} S^{-1}\right)}{2 N(q-2)}\left[\frac{2 q\left(1-\left\|V_{-}\right\|_{\frac{N}{2}} S^{-1}\right)}{N(q-2) C_{N, q} \Theta^{\frac{2 q-N(q-2)}{4}}}\right]^{\frac{4}{N(q-2)-4}} .
$$
(iii) Set
$$
m_{r, s}(\Theta)=\inf _{\gamma \in \Gamma_{r, \Theta}} \sup _{t \in[0,1]} I_{r, s}(\gamma(t))
$$
with
$$
\Gamma_{r, \Theta}=\left\{\gamma \in C\left([0,1], S_{r, \Theta}\right): \gamma(0)=u^0, \gamma(1)=u^1\right\} .
$$
Then
$$
\frac{(N(q-2)-4)\left(1-\left\|V_{-}\right\|_{\frac{N}{2}} S^{-1}\right)}{2 N(q-2)}\left[\frac{2 q\left(1-\left\|V_{-}\right\|_{\frac{N}{2}} S^{-1}\right)}{N(q-2) C_{N, q} \Theta^{\frac{2 q-N(q-2)}{4}}}\right]^{\frac{4}{N(q-2)-4}} \leq m_{r, s}(\Theta) \leq h(T_\Theta),
$$
where $h(T_\Theta)=\max\limits_{t \in \mathbb{R}^{+}} h(t)$, the function $h: \mathbb{R}^{+} \rightarrow \mathbb{R}$ being defined by
$$
h(t)=\frac{1}{2}\left(1+\|V\|_{\frac{N}{2}} S^{-1}\right) t^2 \theta \Theta-\alpha\beta C_{N, p_1} \Theta^{\frac{p_1}{2}} \theta^{\frac{N(p_1-2)}{4}} t^{\frac{N(p_1-2)}{2}}-\frac{1}{2 q} \Theta^{\frac{q}{2}}|\Omega|^{\frac{2-q}{2}} t^{\frac{N(q-2)}{2}} .
$$
Here $\theta$ is the principal eigenvalue of $-\Delta$ with Dirichlet boundary conditions in $\Omega$, and $|\Omega|$ is the volume of $\Omega$.
\end{lemma}
\begin{proof}
(i) Clearly, the set $S_{r, \Theta}$ is path connected. Since $v_1 \in S_{1, \Theta}$ be the positive eigenfunction associated to $\theta$ and note that $\theta$ is the principal eigenvalue of $-\Delta$, then
\begin{equation}\label{eq3.3}
  \int_{\Omega}\left|\nabla v_1\right|^2 d x=\theta \Theta.
\end{equation}
By the H\"older inequality, we know that
\begin{eqnarray*}
  \Theta=\int_{\Omega}|v_1(x)|^2dx  \leq\left(\int_{\Omega}|v_1(x)|^{q} dx\right)^{\frac{2}{q}}\cdot|\Omega|^{\frac{q-2}{q}},
\end{eqnarray*}
which implies
\begin{equation}\label{eq3.4}
  \int_{\Omega}|v_1(x)|^{q} dx\geq\Theta^{\frac{q}{2}}\cdot|\Omega|^{\frac{2-q}{2}}.
\end{equation}
According to $\left(f_2\right)$, there exists a constant $\alpha>0$ such that
\begin{equation}\label{eq3.5}
  F(\tau)\leq\alpha\tau^{p_1}.
\end{equation}
For $x \in \Omega_{\frac{1}{t}}$ and $t>0$, define $v_t(x):=t^{\frac{N}{2}} v_1(t x)$. Using \eqref{eq3.3}, \eqref{eq3.4}, \eqref{eq3.5} and $\frac{1}{2}\leq s\leq1$, it holds
\begin{eqnarray}\label{eq3.6}
&&I_{\frac{1}{t}, s}\left(v_t\right)\nonumber\\
&\leq&\frac{1}{2} \int_{\Omega_r}|\nabla v_t|^2 d x+\frac{1}{2} \int_{\Omega_r} V  v_t^2 d x-\frac{1}{2q} \int_{\Omega_r}|v_t|^q d x- \alpha\beta  \int_{\Omega_r}|v_t|^{p_1} d x\nonumber\\
&\leq&\frac{1}{2}\left(1+\|V\|_{\frac{N}{2}} S^{-1}\right) \int_{\Omega_r}|\nabla v_t|^2 d x-\alpha\beta C_{N, p_1} \Theta^{\frac{2 p_1-N(p_1-2)}{4}}\left( \int_{\Omega_r}\left|\nabla v_t\right|^2 d x\right)^{\frac{N(p_1-2)}{4}}\nonumber\\
&&-\frac{1}{2q} \int_{\Omega_r}|v_t|^q d x\nonumber\\
&\leq&\frac{1}{2}\left(1+\|V\|_{\frac{N}{2}} S^{-1}\right) t^2 \int_{\Omega}\left|\nabla v_1\right|^2 d x-\alpha\beta C_{N, p_1} \Theta^{\frac{2 p_1-N(p_1-2)}{4}}\left(t^2 \int_{\Omega}\left|\nabla v_1\right|^2 d x\right)^{\frac{N(p_1-2)}{4}} \nonumber\\
&&-\frac{1}{2q} t^{\frac{N(q-2)}{2}} \int_{\Omega}\left|v_1\right|^q d x \nonumber\\
&\leq & \frac{1}{2}\left(1+\|V\|_{\frac{N}{2}} S^{-1}\right) t^2 \theta\Theta-\alpha\beta C_{N, p_1} \Theta^{\frac{p_1}{2}}\theta^{\frac{N(p_1-2)}{4}}t^{\frac{N(p_1-2)}{2}}-\frac{1}{2q} t^{\frac{N(q-2)}{2}} \Theta^{\frac{q}{2}}\cdot|\Omega|^{\frac{2-q}{2}} \nonumber\\
&=: & h(t) .
\end{eqnarray}
Note that since $2<p_1<2+\frac{4}{N}<q<2^*$ and $\beta\leq0$ there exist $0<T_\Theta<t_0$ such that $h\left(t_0\right)=0, h(t)<0$ for any $t>t_0, h(t)>0$ for any $0<t<t_0$ and $h\left(T_\Theta\right)=\max\limits_{t \in \mathbb{R}^{+}} h(t)$. As a consequence, there holds
\begin{equation}\label{eq3.7}
  I_{r, s}\left(v_{t_0}\right)=I_{\frac{1}{t_0}, s}\left(v_{t_0}\right) \leq h\left(t_0\right)=0
\end{equation}
for any $r \geq \frac{1}{t_0}$ and $s \in\left[\frac{1}{2}, 1\right]$. Moreover, there exists $0<t_1<T_\Theta$ such that
\begin{equation}\label{eq3.8}
 h(t)<\frac{(N(q-2)-4)\left(1-\left\|V_{-}\right\|_{\frac{N}{2}} S^{-1}\right)}{2 N(q-2)}\left[\frac{2 q\left(1-\left\|V_{-}\right\|_{\frac{N}{2}} S^{-1}\right)}{N(q-2) C_{N, q} \Theta^{\frac{2 q-N(q-2)}{4}}}\right]^{\frac{4}{N(q-2)-4}}
\end{equation}
for $t \in\left[0, t_1\right]$. On the other hand, it follows from the Gagliardo-Nirenberg inequality and the H\"older inequality that
\begin{eqnarray}\label{eq3.9}
I_{r, s}(u)
&\geq&\frac{1}{2} \int_{\Omega_r}|\nabla u|^2 d x+\frac{1}{2} \int_{\Omega_r} V u^2 d x-\frac{1}{q} \int_{\Omega_r}|u|^q d x\nonumber\\
&\geq& \frac{\left(1-\left\|V_{-}\right\|_{\frac{N}{2}} S^{-1}\right)}{2} \int_{\Omega_r}|\nabla u|^2 d x-\frac{C_{N, q} \Theta^{\frac{2 q-N(q-2)}{4}}}{q}\left(\int_{\Omega_r}|\nabla u|^2 d x\right)^{\frac{N(q-2)}{4}}.
\end{eqnarray}
Define
$$
g(t):=\frac{1}{2}\left(1-\left\|V_{-}\right\|_{\frac{N}{2}} S^{-1}\right) t-\frac{C_{N, q} \Theta^{\frac{2 q-N(q-2)}{4}}}{q} t^{\frac{N(q-2)}{4}}
$$
and
$$
\widetilde{t}=\left[\frac{2 q}{N(q-2) C_{N, q}}\left(1-\left\|V_{-}\right\|_{\frac{N}{2}} S^{-1}\right) \Theta^{\frac{q(N-2)-2 N}{4}}\right]^{\frac{4}{N(q-2)-4}},
$$
it is easy to see that $g$ is increasing on $(0, \widetilde{t})$ and decreasing on $(\widetilde{t}, \infty)$, and
$$
g(\widetilde{t})=\frac{(N(q-2)-4)\left(1-\left\|V_{-}\right\|_{\frac{N}{2}} S^{-1}\right)}{2 N(q-2)}\left[\frac{2 q\left(1-\left\|V_{-}\right\|_{\frac{N}{2}} S^{-1}\right)}{N(q-2) C_{N, q} \Theta^{\frac{2 q-N(q-2)}{4}}}\right]^{\frac{4}{N(q-2)-4}} .
$$
For $r \geq \widetilde{r}_\Theta:=\max \left\{\frac{1}{t_1}, \sqrt{\frac{2 \theta \Theta}{\tilde{t}}}\right\}$, we have $v_{\frac{1}{\tilde{r}_\Theta}} \in S_{r, \Theta}$ and
\begin{eqnarray}\label{eq3.10}
\|\nabla v_{\frac{1}{\widetilde{r}_\Theta}}\|_2^2&=&\left(\frac{1}{\widetilde{r}_\Theta}\right)^2\left\|\nabla v_1\right\|_2^2\nonumber\\
&<&\left[\frac{2 q}{N(q-2) C_{N, q}}\left(1-\left\|V_{-}\right\|_{\frac{N}{2}} S^{-1}\right) \Theta^{\frac{q(N-2)-2 N}{4}}\right]^{\frac{4}{N(q-2)-4}} .
\end{eqnarray}
Moreover, there holds
\begin{equation}\label{eq3.11}
 I_{\widetilde{r}_\Theta, s}\left(v_{\frac{1}{\widetilde{r}_\Theta}}\right) \leq h\left(\frac{1}{\widetilde{r}_\Theta}\right) \leq h\left(t_1\right) .
\end{equation}
Setting $u^0=v_{\frac{1}{\tilde{r}_\Theta}}, u^1=v_{t_0}$ and
\begin{equation}\label{eq3.12}
  r_\Theta=\max \left\{\frac{1}{t_0}, \widetilde{r}_\Theta\right\}.
\end{equation}
Combining \eqref{eq3.7}, \eqref{eq3.8}, \eqref{eq3.10} and \eqref{eq3.11}, (i) holds.

(ii) By \eqref{eq3.9} and a direct calculation, (ii) holds.

(iii) Since $I_{r, s}\left(u^1\right) \leq 0$ for any $\gamma \in \Gamma_{r, \Theta}$, we have
$$
\|\nabla \gamma(0)\|_2^2<\tilde{t}<\|\nabla \gamma(1)\|_2^2.
$$
It then follows from \eqref{eq3.9} that
$$
\begin{aligned}
\max\limits_{t \in[0,1]} I_{r, s}(\gamma(t)) & \geq g(\widetilde{t}) \\
& =\frac{(N(q-2)-4)\left(1-\left\|V_{-}\right\|_{\frac{N}{2}} S^{-1}\right)}{2 N(q-2)}\left[\frac{2 q\left(1-\left\|V_{-}\right\|_{\frac{N}{2}} S^{-1}\right)}{N(q-2) C_{N, q} \Theta^{\frac{2 q-N(q-2)}{4}}}\right]^{\frac{4}{N(q-2)-4}}
\end{aligned}
$$
for any $\gamma \in \Gamma_{r, \Theta}$, hence the first inequality in (iii) holds. Now we define a path $\gamma \in \Gamma_{r, \Theta}$ by
$$
\gamma(\tau)(x)=\left(\tau t_0+(1-\tau) \frac{1}{\widetilde{r}_\Theta}\right)^{\frac{N}{2}} v_1\left(\left(\tau t_0+(1-\tau) \frac{1}{\widetilde{r}_\Theta}\right) x\right)
$$
for $\tau \in[0,1]$ and $x \in \Omega_r$. Then by \eqref{eq3.6} we have $m_{r, s}(\Theta) \leq h(T_\Theta)$, where $h(T_\Theta)=\max\limits_{t \in \mathbb{R}^{+}} h(t)$. Note that $T_\Theta$ is independent of $r$ and $s$.
\end{proof}

By using Lemma \ref{L3.1}, the energy functional $I_{r, s}$ possesses the mountain pass geometry. To obtain bounded Palais-Smale sequence, we recall a proposition from \cite{{JBXC2022},{XCLJ2022}}.

\begin{proposition}\label{p3.1}
{\em(see \cite[Theorem 1]{{JBXC2022}})} Let $(E,\langle\cdot, \cdot\rangle)$ and $(H,(\cdot, \cdot))$ be two infinite-dimensional Hilbert spaces and assume there are continuous injections
$$
E \hookrightarrow H \hookrightarrow E^{\prime} .
$$
Let
$$
\|u\|^2=\langle u, u\rangle, \quad|u|^2=(u, u) \quad \text { for } u \in E,
$$
and
$$
S_\mu=\left\{u \in E:|u|^2=\mu\right\}, \quad T_u S_\mu=\{v \in E:(u, v)=0\} \quad \text { for } \mu \in(0,+\infty) .
$$
Let $I \subset(0,+\infty)$ be an interval and consider a family of $C^2$ functionals $\Phi_\rho: E \rightarrow \mathbb{R}$ of the form
$$
\Phi_\rho(u)=A(u)-\rho B(u), \quad \text { for } \rho \in I,
$$
with $B(u) \geq 0$ for every $u \in E$, and
\begin{equation}\label{eq3.13}
  A(u) \rightarrow+\infty \quad \text { or } \quad B(u) \rightarrow+\infty \quad \text { as } u \in E \text { and }\|u\| \rightarrow+\infty.
\end{equation}
Suppose moreover that $\Phi_\rho^{\prime}$ and $\Phi_\rho^{\prime \prime}$ are $\tau$-H\"older continuous, $\tau \in(0,1]$, on bounded sets in the following sense: for every $R>0$ there exists $M=M(R)>0$ such that
\begin{equation}\label{eq3.14}
\left\|\Phi_\rho^{\prime}(u)-\Phi_\rho^{\prime}(v)\right\| \leq M\|u-v\|^\tau \quad \text { and } \quad\left\|\Phi_\rho^{\prime \prime}(u)-\Phi_\rho^{\prime \prime}(v)\right\| \leq M\|u-v\|^\tau
\end{equation}
for every $u, v \in B(0, R)$. Finally, suppose that there exist $w_1, w_2 \in S_\mu$ independent of $\rho$ such that
$$
c_\rho:=\inf _{\gamma \in \Gamma} \max _{t \in[0,1]} \Phi_\rho(\gamma(t))>\max \left\{\Phi_\rho\left(w_1\right), \Phi_\rho\left(w_2\right)\right\} \quad \text { for all } \rho \in I,
$$
where
$$
\Gamma=\left\{\gamma \in C\left([0,1], S_\mu\right): \gamma(0)=w_1, \gamma(1)=w_2\right\} .
$$
Then for almost every $\rho \in I$, there exists a sequence $\left\{u_n\right\} \subset S_\mu$ such that

(i) $\Phi_\rho\left(u_n\right) \rightarrow c_\rho$,

(ii) $\left.\Phi_\rho^{\prime}\right|_{S_\mu}\left(u_n\right) \rightarrow 0$,

(iii) $\left\{u_n\right\}$ is bounded in $E$.
\end{proposition}

\begin{lemma}\label{L3.2}
For any $\Theta>0$, let $r>r_\Theta$, where $r_\Theta$ is defined in Lemma \ref{L3.1}. Then problem \eqref{eq3.1} has a solution $\left(\lambda_{r, s}, u_{r, s}\right)$ for almost every $s \in\left[\frac{1}{2}, 1\right]$. Moreover, $u_{r, s} \geq 0$ and $I_{r, s}\left(u_{r, s}\right)=m_{r, s}(\Theta)$.
\end{lemma}
\begin{proof}
By Proposition \ref{p3.1}, it follows that
$$
A(u)=\frac{1}{2} \int_{\Omega_r}|\nabla u|^2 d x+\frac{1}{2} \int_{\Omega_r} V(x) u^2 d x-\beta \int_{\Omega_r}F(u) d x \quad \text { and } \quad B(u)=\frac{1}{q} \int_{\Omega_r}|u|^q d x.
$$
Note that the assumptions in Proposition \ref{p3.1} hold due to $\beta \leq 0$ and Lemma \ref{L3.1}. Hence, for almost every $s \in\left[\frac{1}{2}, 1\right]$, there exists a bounded Palais-Smale sequence $\left\{u_n\right\}$ satisfying
$$
I_{r, s}\left(u_n\right) \rightarrow m_{r, s}(\Theta) \quad \text { and }\left.\quad I_{r, s}^{\prime}\left(u_n\right)\right|_{T_{u_n} S_{r, \Theta}} \rightarrow 0,
$$
where $T_{u_n} S_{r, \Theta}$ denotes the tangent space of $S_{r, \Theta}$ at $u_n$. Then
$$
\lambda_n=-\frac{1}{\Theta}\left(\int_{\Omega_r}\left|\nabla u_n\right|^2 d x+\int_{\Omega_r} V(x) u_n^2 d x-\beta \int_{\Omega_r}f(u_n)u_n d x-s \int_{\Omega_r}\left|u_n\right|^q d x\right)
$$
is bounded and
\begin{equation}\label{eq3.15}
  I_{r, s}^{\prime}\left(u_n\right)+\lambda_n u_n \rightarrow 0 \quad \text { in } H^{-1}\left(\Omega_r\right) .
\end{equation}
Moreover, since $\left\{u_n\right\}$ is a bounded Palais-Smale sequence, there exist $u_0 \in H_0^1\left(\Omega_r\right)$ and $\lambda \in \mathbb{R}$ such that, up to a subsequence,
\begin{eqnarray*}
\lambda_n &\rightarrow& \lambda \ \text { in }\ \mathbb{R},\\
u_n &\rightharpoonup& u_0 \  \text { in }\ H_0^1(\Omega_r),\\
 u_n &\rightarrow& u_0 \  \text { in }\ L^t(\Omega_r) \text { for all } 2 \leq t<2^*,
\end{eqnarray*}
where $u_0$ satisfies
$$
\left\{\begin{array}{l}
-\Delta u_0+V u_0+\lambda u_0=s\left|u_0\right|^{q-2} u_0+\beta f(u_0) \quad \text { in } \Omega_r \\
u_0 \in H_0^1\left(\Omega_r\right), \quad \int_{\Omega_r}\left|u_0\right|^2 d x=\Theta.
\end{array}\right.
$$
Using \eqref{eq3.15}, we have
$$
I_{r, s}^{\prime}\left(u_n\right) u_0+\lambda_n \int_{\Omega_r} u_n u_0 d x \rightarrow 0 \text { as } n \rightarrow \infty
$$
and
\begin{equation*}
  I_{r, s}^{\prime}\left(u_n\right) u_n+\lambda_n \Theta\rightarrow 0 \ \text{as }\  n \rightarrow \infty  .
\end{equation*}
Note that
\begin{eqnarray*}
\lim _{n \rightarrow \infty} \int_{\Omega_r} V(x) u_n^2 d x&=&\int_{\Omega_r} V(x) u_0^2 d x,\\
\lim _{n \rightarrow \infty} \int_{\Omega_r} f(u_n)u_n d x&=&\int_{\Omega_r} f(u_0)u_0 d x,\\
\lim _{n \rightarrow \infty} \int_{\Omega_r} f(u_n)u_0 d x&=&\int_{\Omega_r} f(u_0)u_0  d x,
\end{eqnarray*}
so we get $u_n \rightarrow u_0$ in $H_0^1(\Omega_r)$, hence $I_{r, s}(u_0)=m_{r, s}(\Theta)$.

Now, we show that $u_{r, s} \geq 0$. In order to obtain it, we only need to modify the proof of Proposition \ref{p3.1}. In fact, for almost every $s \in\left[\frac{1}{2}, 1\right]$, the derivative $m_{r, s}^{\prime}$ with respect to $s$ is well defined since the function $s \mapsto m_{r, s}$ is nonincreasing, where $m_{r, s}$ denotes $m_{r, s}(\Theta)$ for fixed $\Theta$. Let $s$ be such that $m_{r, s}^{\prime}$ exists and $\left\{s_n\right\} \subset\left[\frac{1}{2}, 1\right]$ be a monotone increasing sequence converging to s. Similar to the proof of Proposition \ref{p3.1}, there exist $\left\{\gamma_n\right\} \subset \Gamma_{r, \Theta}$ and $K=K\left(m_{r, s}^{\prime}\right)$ such that:

(i) if $I_{r, s}\left(\gamma_n(t)\right) \geq m_{r, s}-\left(2-m_{r, s}^{\prime}\right)\left(s-s_n\right)$, then $\int_{\Omega_r}\left|\nabla \gamma_n(t)\right|^2 d x \leq K$.

(ii) $\max\limits_{t \in[0,1]} I_{r, s}\left(\gamma_n(t)\right) \leq m_{r, s}-\left(2-m_{r, s}^{\prime}\right)\left(s-s_n\right)$.

Letting $\widetilde{\gamma}_n(t)=\left|\gamma_n(t)\right|$ for any $t \in[0,1]$, it follows that $\left\{\widetilde{\gamma}_n\right\} \subset \Gamma_{r, \Theta}$. Observe that $\left\|\nabla|u|\|_2^2 \leq\right.$ $\|\nabla u\|_2^2$ for any $u \in H^1(\mathbb{R}^N)$. Now we have:

(I) if $I_{r, s}\left(\widetilde{\gamma}_n(t)\right) \geq m_{r, s}-\left(2-m_{r, s}^{\prime}\right)\left(s-s_n\right)$, then $I_{r, s}\left(\gamma_n(t)\right) \geq m_{r, s}-\left(2-m_{r, s}^{\prime}\right)\left(s-s_n\right)$. By (i), there holds $\int_{\Omega_r}\left|\nabla \gamma_n(t)\right|^2 d x \leq K$, and hence $\int_{\Omega_r}\left|\nabla \widetilde{\gamma}_n(t)\right|^2 d x \leq K$. Thus (i) also holds for $\widetilde{\gamma}_n$.

(II) $\max\limits_{t \in[0,1]} I_{r, s}\left(\widetilde{\gamma}_n(t)\right) \leq \max\limits_{t \in[0,1]} I_{r, s}\left(\gamma_n(t)\right) \leq m_{r, s}-\left(2-m_{r, s}^{\prime}\right)\left(s-s_n\right)$.

By replacing $\gamma_n$ with $\widetilde{\gamma}_n$ in the proof of Proposition \ref{p3.1}, we obtain a nonnegative bounded Palais-Smale sequence $\left\{u_n\right\}$. Consequently, there exists a nonnegative normalized solution to \eqref{eq3.1} for almost every $s \in\left[\frac{1}{2}, 1\right]$ as above.
\end{proof}

In order to obtain a solution of \eqref{eq1.1}, we need to prove a uniform estimate for the solutions of \eqref{eq3.1} established in Lemma \ref{L3.2}.

\begin{lemma}\label{L3.3}
If $(\lambda_{r, s}, u_{r, s}) \in \mathbb{R} \times S_{r, \Theta}$ is a solution of \eqref{eq3.1} established in Lemma \ref{L3.2} for some $r$ and $s$, then
$$
\int_{\Omega_r}|\nabla u|^2 d x \leq \frac{4 N}{N(q-2)-4}\left[\frac{q-2}{2} h(T_\Theta)+\Theta\left(\frac{1}{2 N}\|\widetilde{V}\|_{\infty}+\frac{q-2}{4}\|V\|_{\infty}\right)\right],
$$
where the constant $h(T_\Theta)$ is defined in (iii) of Lemma \ref{L3.1}  and is independent of $r$ and $s$.
\end{lemma}
\begin{proof} For simplicity, we denote $(\lambda_{r, s}, u_{r, s})$ as $(\lambda, u)$ in this lemma. Since $u$ is a solution of \eqref{eq3.1}, we have
\begin{equation}\label{eq3.16}
  \int_{\Omega_r}|\nabla u|^2 d x+\int_{\Omega_r} V(x) u^2 d x=s \int_{\Omega_r}|u|^q d x+\beta \int_{\Omega_r}f(u)u d x-\lambda \int_{\Omega_r}|u|^2 d x .
\end{equation}
The Pohozaev identity implies
\begin{eqnarray*}
&&\frac{N-2}{2 N} \int_{\Omega_r}|\nabla u|^2 d x+\frac{1}{2 N} \int_{\partial \Omega_r}|\nabla u|^2(x \cdot \mathbf{n}) d \sigma+\frac{1}{2 N} \int_{\Omega_r}\widetilde{V}(x) u^2dx+\frac{1}{2} \int_{\Omega_r} V u^2 d x \\
&=&-\frac{\lambda}{2} \int_{\Omega_r}|u|^2 d x+\frac{s}{q} \int_{\Omega_r}|u|^q d x+\beta \int_{\Omega_r}F(u) d x,
\end{eqnarray*}
where $\mathbf{n}$ denotes the outward unit normal vector on $\partial \Omega_r$. It then follows from $\beta \leq 0$ and $(f_2)$ that
\begin{eqnarray*}
&&\frac{1}{N} \int_{\Omega_r}|\nabla u|^2 d  x-\frac{1}{2 N} \int_{\partial \Omega_r}|\nabla u|^2(x \cdot \mathbf{n}) d \sigma-\frac{1}{2 N} \int_{\Omega_r}(\nabla V \cdot x) u^2 d x \\
& =&\frac{(q-2) s}{2 q} \int_{\Omega_r}|u|^q d x+ \int_{\Omega_r}(\frac{\beta }{2  }f(u)u-\beta F(u))d x \\
& \geq& \frac{(q-2) s}{2 q} \int_{\Omega_r}|u|^q d x+  \frac{\beta (q-2) }{2  }  \int_{\Omega_r}F(u) d x \\
& =&\frac{q-2}{2}\left(\frac{1}{2} \int_{\Omega_r}|\nabla u|^2 d x+\frac{1}{2} \int_{\Omega_r} V u^2 d x-m_{r, s}(\Theta)\right) .
\end{eqnarray*}
Consequently, we have
 \begin{eqnarray*}
\frac{q-2}{2} m_{r, s}(\Theta) &\geq & \frac{q-2}{2}\left(\frac{1}{2} \int_{\Omega_r}|\nabla u|^2 d x+\frac{1}{2} \int_{\Omega_r} V u^2 d x\right)-\frac{1}{N} \int_{\Omega_r}|\nabla u|^2 d x \\
&& +\frac{1}{2 N} \int_{\partial \Omega_r}|\nabla u|^2(x \cdot \mathbf{n}) d \sigma+\frac{1}{2 N} \int_{\Omega_r}(\nabla V \cdot x) u^2 d x \\
&\geq & \frac{N(q-2)-4}{4 N} \int_{\Omega_r}|\nabla u|^2 d x-\Theta\left(\frac{1}{2 N}\|\nabla V \cdot x\|_{\infty}+\frac{q-2}{4}\|V\|_{\infty}\right),
 \end{eqnarray*}
where the last inequality holds since $x \cdot \mathbf{n}(x) \geq 0$ for any $x \in \partial \Omega_r$ due to the convexity of $\Omega_r$. Using Lemma \ref{L3.1}, we have
\begin{eqnarray*}
 \frac{N(q-2)-4}{4 N} \int_{\Omega_r}|\nabla u|^2 d x-\Theta\left(\frac{1}{2 N}\|\nabla V \cdot x\|_{\infty}+\frac{q-2}{4}\|V\|_{\infty}\right)
 &\leq&\frac{q-2}{2}h(T_\Theta),
\end{eqnarray*}
which implies
\begin{equation*}
  \int_{\Omega_r}|\nabla u|^2 d x \leq \frac{4 N}{N(q-2)-4}\left[\frac{q-2}{2} h(T_\Theta)+\Theta\left(\frac{1}{2 N}\|\widetilde{V}\|_{\infty}+\frac{q-2}{4}\|V\|_{\infty}\right)\right].
\end{equation*}
This completes the proof of lemma.
\end{proof}

Now, we obtain a solution of \eqref{eq1.1} by letting $s \rightarrow 1$.

\begin{lemma}\label{L3.4}
For every $\Theta>0$, problem \eqref{eq1.1} has a solution $\left(\lambda_r, u_r\right)$ provided $r>r_\Theta$ where $r_\Theta$ is as in Lemma \ref{L3.1}. Moreover, $u_r \geq 0$ in $\Omega_r$.
\end{lemma}
\begin{proof}
By using Lemma  \ref{L3.2}, there is a nonnegative solution $(\lambda_{r, s}, u_{r, s})$ to \eqref{eq3.1} for almost every $s \in\left[\frac{1}{2}, 1\right]$. In view of Lemma  \ref{L3.3}, $\left\{u_{r, s}\right\}$ is bounded. By an argument similar to that in Lemma  \ref{L3.2}, there exist $u_r \in S_{r, \Theta}$ and $\lambda_r$ such that, going if necessary to a subsequence,
$$
\lambda_{r, s} \rightarrow \lambda_r \quad \text { and } \quad u_{r, s} \rightarrow u_r \quad \text { in } H_0^1\left(\Omega_r\right) \quad \text { as } s \rightarrow 1 .
$$
Hence $u_r$ is a nonnegative solution of problem \eqref{eq1.1}.
\end{proof}

Next, we will consider the Lagrange multiplier. we first establish an a priori estimate for the solutions of \eqref{eq1.1}.

\begin{lemma}\label{L3.5}
If $\left\{\left(\lambda_r, u_r\right)\right\}$ is a family of nonnegative solutions of \eqref{eq1.1} such that $\left\|u_r\right\|_{H^1} \leq$ $C$ with $C>0$ independent of $r$, then $\limsup\limits_{r \rightarrow \infty}\left\|u_r\right\|_{\infty}<\infty$.
\end{lemma}
\begin{proof}
Using the regularity theory of elliptic partial differential equations, we know that $u_r \in C(\Omega_r)$. Assume to the contrary that there exist a sequence, for simplicity denoted by $\left\{u_r\right\}$, and $x_r \in \Omega_r$ such that
$$
M_r:=\max _{x \in \Omega_r} u_r(x)=u_r\left(x_r\right) \rightarrow \infty \quad \text { as } r \rightarrow \infty.
$$
Suppose without loss of generality that, up to a subsequence, $\lim\limits_{r \rightarrow \infty} \frac{x_r}{\left|x_r\right|}=(1,0, \ldots, 0)$. Set
$$
v_r(x)=\frac{u_r\left(x_r+\tau_r x\right)}{M_r} \text { for } x \in \Sigma^r:=\left\{x \in \mathbb{R}^N: x_r+\tau_r x \in \Omega_r\right\},
$$
where $\tau_r=M_r^{\frac{2-q}{2}}$. Then $\tau_r \rightarrow 0$ as $r \rightarrow \infty$, $\left\|v_r\right\|_{L^{\infty}\left(\Sigma^r\right)} \leq 1$, and $v_r$ satisfies
\begin{equation}\label{eq3.17}
-\Delta v_r+\tau_r^2 V\left(x_r+\tau_r x\right) v_r+\tau_r^2 \lambda_r v_r=\left|v_r\right|^{q-2} v_r+\beta M_r^{1-q}f( M_rv_r) \quad \text { in } \Sigma^r .
\end{equation}
In fact, since $u_r$ is a nonnegative solution of \eqref{eq1.1}, we obtain
\begin{eqnarray*}
&&-\Delta u_r\left(x_r+\tau_r x\right)+ V\left(x_r+\tau_r x\right) u_r\left(x_r+\tau_r x\right)+ \lambda_r u_r\left(x_r+\tau_r x\right)\\
&=&\left|u_r\left(x_r+\tau_r x\right)\right|^{q-2} u_r\left(x_r+\tau_r x\right)+\beta f( u_r\left(x_r+\tau_r x\right)) \quad \text { in } \Omega_r ,
\end{eqnarray*}
then by a direct calculation and the definition of $v_r(x)$, $\tau_r$, we know that \eqref{eq3.17} holds.
In view of \eqref{eq1.1}, the Gagliardo-Nirenberg inequality and $\left\|u_r\right\|_{H^1} \leq C$ with $C$ independent of $r$, we infer that the sequence $\left\{\lambda_r\right\}$ is bounded. It then follows from the regularity theory of elliptic partial differential equations and the Arzela-Ascoli theorem that there exists $v$ such that, up to a subsequence
$$
v_r \rightarrow v \quad \text { in } H_0^1(\Sigma) \quad \text { and } \quad v_r \rightarrow v \quad \text { in } C_{l o c}^\beta(\Sigma) \text { for some } \beta \in(0,1),
$$
where $\Sigma:=\lim\limits_{r \rightarrow \infty} \Sigma^r$.

Similar to the proof of \cite[Lemma 2.7]{{TBAQ2023}}, we have
 \begin{equation*}
   \liminf _{r \rightarrow \infty} \frac{\operatorname{dist}\left(x_r, \partial \Omega_r\right)}{\tau_r}=\liminf _{r \rightarrow \infty} \frac{\left|y_r-x_r\right|}{\tau_r} \geq d>0,
 \end{equation*}
where $y_r \in \partial \Omega_r$ is such that $\operatorname{dist}\left(x_r, \partial \Omega_r\right)=\left|y_r-x_r\right|$ for any large $r$. As a result, by letting $r \rightarrow \infty$ in \eqref{eq3.17}, we obtain that $v \in H_0^1(\Sigma)$ is a nonnegative solution of
$$
-\Delta v=|v|^{q-2} v \quad \text { in } \Sigma,
$$
where
$$
\Sigma= \begin{cases}\mathbb{R}^N & \text { if } \liminf\limits_{r \rightarrow \infty} \frac{\operatorname{dist}\left(x_r, \partial \Omega_r\right)}{\tau_r}=\infty, \\ \left\{x \in \mathbb{R}^N: x_1>-d\right\} & \text { if } \liminf\limits_{r \rightarrow \infty} \frac{\operatorname{dist}\left(x_r, \partial \Omega_r\right)}{\tau_r}>0 .\end{cases}
$$
It then follows from the Liouville theorems (see \cite{MEPL1982}) that $v=0$ in $H_0^1(\Sigma)$, which contradicts $v(0)=\lim\limits_{r \rightarrow \infty} v_r(0)=1$.
\end{proof}

Clearly, the proof of Lemma \ref{L3.5} does not depend on $\beta$.

\begin{lemma}\label{L3.6}
Let $\left(\lambda_{r, \Theta}, u_{r, \Theta}\right)$ be the solution of \eqref{eq1.1} from Lemma \ref{L3.4}. If $\|\widetilde{V}_{+}\|_{\frac{N}{2}} < 2 S$, then there exists $\bar{\Theta}>0$ such that
$$
\liminf\limits_{r \rightarrow \infty} \lambda_{r, \Theta}>0 \quad \text { for } 0<\Theta<\bar{\Theta} .
$$
\end{lemma}
\begin{proof}
Let $\left(\lambda_{r, \Theta}, u_{r, \Theta}\right)$ be the solution of \eqref{eq1.1} established in Theorem \ref{L3.4}. By the regularity theory of elliptic partial differential equations, we have $u_{r, \Theta} \in C\left(\Omega_r\right)$. Using Lemma \ref{L3.5}, it holds
$$
\limsup _{r \rightarrow \infty} \max\limits_{\Omega_r} u_{r, \Theta}<\infty .
$$
Setting
$$
Q(\Theta)=\liminf\limits_{r \rightarrow \infty} \max\limits_{\Omega_r} u_{r, \Theta},
$$
we claim that there is $\Theta_1>0$ such that $Q(\Theta)>0$ for any $0<\Theta<\Theta_1$. Assume to the contrary that there exists a sequence $\left\{\Theta_k\right\}$ tending to 0 as $k \rightarrow \infty$ such that $Q\left(\Theta_k\right)=0$ for any $k$, that is,
\begin{equation}\label{eq3.20}
  \liminf\limits_{r \rightarrow \infty} \max\limits_{\Omega_r} u_{r, \Theta_k}=0 \quad \text {for any}\ k.
\end{equation}
As a consequence of (iii) in Lemma \ref{L3.1}, for any $r>r_{\Theta_k}$, we have
\begin{equation}\label{eq3.21}
I_r\left(u_{r, \Theta_k}\right)=m_{r, 1}\left(\Theta_k\right) \rightarrow \infty \quad \text { as } k \rightarrow \infty .
\end{equation}
For any given $k$, it follows from \eqref{eq3.20} and $u_{r, \Theta_k} \in S_{r, \Theta_k}$ that, up to a subsequence,
\begin{equation}\label{eq3.22}
\int_{\Omega_r}\left|u_{r, \Theta_k}\right|^s d x=\int_{\Omega_r}\left|u_{r, \Theta_k}\right|^{s-2}\left|u_{r, \Theta_k}\right|^2 d x \leq\left|\max _{\Omega_r} u_{r, \Theta_k}\right|^{s-2} \Theta_k \rightarrow 0 \quad \text { as } r \rightarrow \infty
\end{equation}
for any $s>2$. Hence, for any given large $k$, there exists $\bar{r}_k>r_{\Theta_k}$ such that
$$
\left|\frac{1}{q} \int_{\Omega_r} |u_{r, \Theta_k}|^q d x+\beta \int_{\Omega_r}f(u_{r, \Theta_k}) d x \right|<\frac{m_{r, 1}\left(\Theta_k\right)}{2} \text { for any } r \geq \bar{r}_k .
$$
In view of \eqref{eq3.21} and $I_r\left(u_{r, \Theta_k}\right)=m_{r, 1}\left(\Theta_k\right)$, we further have
\begin{equation}\label{eq3.23}
\int_{\Omega_r}\left|\nabla u_{r, \Theta_k}\right|^2 d x+\int_{\Omega_r} V(x) u_{r, \Theta_k}^2 d x \geq \frac{m_{r, 1}\left(\Theta_k\right)}{2} \text { for any large } k \text { and } r \geq \bar{r}_k .
\end{equation}
It follows from \eqref{eq3.20}, \eqref{eq3.22} and \eqref{eq3.23} that there exists $r_k \geq \bar{r}_k$ with $r_k \rightarrow \infty$ as $k \rightarrow \infty$ such that
\begin{equation}\label{eq3.24}
   \lim _{k \rightarrow \infty} \max _{\Omega_{r_k}} u_{r_k, \Theta_k}=0,
\end{equation}
\begin{equation}\label{eq3.25}
   \int_{\Omega_{r_k}}\left|u_{r_k, \Theta_k}\right|^s d x \leq\left|\max _{\Omega_{r_k}} u_{r_k, \Theta_k}\right|^{s-2} \Theta_k \rightarrow 0 \text { as } k \rightarrow \infty \text { for any } s>2
\end{equation}
and
\begin{equation}\label{eq3.26}
\int_{\Omega_{r_k}}\left|\nabla u_{r_k, \Theta_k}\right|^2 d x+\int_{\Omega_r} V u_{r_k, \Theta_k}^2 d x \rightarrow \infty \quad \text { as } k \rightarrow \infty .
\end{equation}
By \eqref{eq1.1}, \eqref{eq3.25} and \eqref{eq3.26}, we have
\begin{equation}\label{eq3.27}
\lambda_{r_k,\Theta_k} \rightarrow-\infty \quad \text { as } k \rightarrow \infty .
\end{equation}
Now \eqref{eq1.1} implies
\begin{eqnarray*}
-\Delta u_{r_k, \Theta_k}+V(x)u_{r_k, \Theta_k}+\lambda_{r_k,\Theta_k} u_{r_k, \Theta_k}=|u_{r_k, \Theta_k}|^{q-2}u_{r_k, \Theta_k}+\beta f(u_{r_k, \Theta_k}),
\end{eqnarray*}
so
$$
-\Delta u_{r_k, \Theta_k}+\left(\|V\|_{\infty}+\frac{\lambda_{r_k, \Theta_k}}{2}\right) u_{r_k, \Theta_k} \geq-\frac{\lambda_{r_k, \Theta_k}}{2}u_{r_k, \Theta_k}+\left|u_{r_k, \Theta_k}\right|^{q-2}u_{r_k, \Theta_k}+\beta f(u_{r_k, \Theta_k}) .
$$
Using \eqref{eq3.27} and \eqref{eq3.24}, it follows that
$$
-\Delta u_{r_k, \Theta_k}+\left(\|V\|_{\infty}+\frac{\lambda_{r_k, \Theta_k}}{2}\right) u_{r_k, \Theta_k} \geq 0
$$
for large $k$. Let $\theta_{r_k}$ be the principal eigenvalue of $-\Delta$ with Dirichlet boundary condition in $\Omega_{r_k}$, and $v_{r_k}>0$ be the corresponding normalized eigenfunction. It follows that
$$
\left(\theta_{r_k}+\|V\|_{\infty}+\frac{\lambda_{r_k, \Theta_k}}{2}\right) \int_{\Omega_{r_k}} u_{r_k, \Theta_k} v_{r_k} d x \geq 0.
$$
Since $\int_{\Omega_{r_k}} u_{r_k, \Theta_k} v_{r_k} d x>0$, we have
$$
\theta_{r_k}+\|V\|_{\infty}+\frac{\lambda_{r_k, \Theta_k}}{2} \geq 0,
$$
which contradicts \eqref{eq3.27} for large $k$. Hence the claim holds, that is, there exists $\Theta_1>0$ such that
\begin{equation}\label{eqq3.28}
  Q(\Theta)=\liminf\limits_{r \rightarrow \infty} \max\limits_{\Omega_r} u_{r, \Theta}>0
\end{equation}
for any $0<\Theta<\Theta_1$.

We consider $H^1(\Omega_r)$ as a subspace of $H^1(\mathbb{R}^N)$ for any $r>0$. It follows from Lemma \ref{L3.3} that the set of solutions $\left\{u_{r, \Theta}: r>r_\Theta\right\}$ established in Lemma \ref{L3.4} is bounded in $H^1(\mathbb{R}^N)$, so there exist $u_\Theta \in H^1(\mathbb{R}^N)$ and $\lambda_\Theta \in \mathbb{R}$ such that up to a subsequence:
\begin{eqnarray*}
  \lambda_{r, \Theta} &\rightarrow& \lambda_\Theta,\\
u_{r, \Theta} &\rightharpoonup& u_\Theta \ \text {in}\ H^1(\mathbb{R}^N),\\
u_{r, \Theta} &\rightarrow& u_\Theta \ \text{in}\ L_{l o c}^k(\mathbb{R}^N) \ \text{for all}\ 2 \leq k<2^*,\\
u_{r, \Theta} &\rightarrow& u_\Theta \ \text {a.e. in}\ \mathbb{R}^N
\end{eqnarray*}
and $u_\Theta$ is a solution of the equation
\begin{equation*}
  -\Delta u+V(x)u+\lambda_\Theta u=|u|^{q-2} u+\beta f(u) \ \text {in}\ \mathbb{R}^N.
\end{equation*}
Hence,
\begin{equation}\label{eq3.28}
\int_{\mathbb{R}^N}\left|\nabla u_\Theta\right|^2 d x+\int_{\mathbb{R}^N} V(x) u_\Theta^2 d x+\lambda_\Theta \int_{\mathbb{R}^N} u_\Theta^2 d x=\int_{\mathbb{R}^N}\left|u_\Theta\right|^q d x+\beta \int_{\mathbb{R}^N}f(u_\Theta)u_\Theta d x
\end{equation}
and the Pohozaev identity gives
\begin{eqnarray}\label{eq3.29}
& &\frac{N-2}{2 N} \int_{\mathbb{R}^N}\left|\nabla u_\Theta\right|^2 d x+\frac{1}{2 N} \int_{\mathbb{R}^N} \widetilde{V} u_\Theta^2dx+\frac{1}{2} \int_{\mathbb{R}^N} V(x) u_\Theta^2 d x+\frac{\lambda_\Theta}{2} \int_{\mathbb{R}^N} u_\Theta^2 d x \nonumber\\
& =&\frac{1}{q} \int_{\mathbb{R}^N}\left|u_\Theta\right|^q d x+\beta \int_{\mathbb{R}^N}F(u_\Theta) d x .
\end{eqnarray}
It follows from \eqref{eq3.28}, \eqref{eq3.29}, $(f_2)$, the Gagliardo-Nirenberg inequality and the fact $\beta\leq0$ that
\begin{eqnarray*}
  & &\frac{1}{ N} \int_{\mathbb{R}^N}\left|\nabla u_\Theta\right|^2 d x+\frac{1}{2 N} \int_{\mathbb{R}^N} \widetilde{V}(x) u_\Theta^2dx  \nonumber\\
  &=&\left(\frac{1}{2}-\frac{1}{q}\right) \int_{\mathbb{R}^N}\left|u_\Theta\right|^q d x+\frac{\beta}{2} \int_{\mathbb{R}^N}(f(u_\Theta)u_\Theta-2F(u_\Theta)) d x  \nonumber\\
  &\leq& \frac{C_{N, q}(q-2)}{2 q}\left(\int_{\mathbb{R}^N} u_\Theta^2 d x\right)^{\frac{2 q-N(q-2)}{4}}\left(\int_{\mathbb{R}^N}\left|\nabla u_\Theta\right|^2 d x\right)^{\frac{N(q-2)}{4}}.
\end{eqnarray*}
By using the H\"older inequality, we have
\begin{equation*}
  \left(\frac{1}{N}-\frac{\|\widetilde{V}_{+}\|_{\frac{N}{2}} S^{-1}}{2 N}\right) \int_{\mathbb{R}^N}\left|\nabla u_\Theta\right|^2 d x\leq\frac{1}{ N} \int_{\mathbb{R}^N}\left|\nabla u_\Theta\right|^2 d x+\frac{1}{2 N} \int_{\mathbb{R}^N} \widetilde{V}(x) u_\Theta^2dx.
\end{equation*}
Therefore,
\begin{eqnarray*}
   &&\left(\frac{1}{N}-\frac{\|\widetilde{V}_{+}\|_{\frac{N}{2}} S^{-1}}{2 N}\right) \int_{\mathbb{R}^N}\left|\nabla u_\Theta\right|^2 d x\nonumber\\
   &\leq&\frac{C_{N, q}(q-2)}{2 q}\left(\int_{\mathbb{R}^N} u_\Theta^2 d x\right)^{\frac{2 q-N(q-2)}{4}}\left(\int_{\mathbb{R}^N}\left|\nabla u_\Theta\right|^2 d x\right)^{\frac{N(q-2)}{4}}.
\end{eqnarray*}
If $u_\Theta \neq 0$, Using $\|\widetilde{V}_{+}\|_{\frac{N}{2}} < 2 S$, we obtain that
\begin{equation}\label{eq3.30}
  \int_{\mathbb{R}^N}\left|\nabla u_\Theta\right|^2 d x \geq\left[\frac{q\left(2-\|\widetilde{V}_{+}\|_{\frac{N}{2}} S^{-1}\right)}{N C_{N, q}(q-2)}\right]^{\frac{4}{N(q-2)-4}} \Theta^{\frac{q(N-2)-2 N}{N(q-2)-4}} .
\end{equation}
Next, it follows from \eqref{eq3.5}, \eqref{eq3.28}, \eqref{eq3.29}, \eqref{eq3.30}, $(f_2)$ and $2+\frac{4}{N}<q<2^*$ that
\begin{eqnarray*}
 \left(\frac{1}{q}-\frac{1}{ 2}\right)\lambda_\Theta \int_{\mathbb{R}^N} u_\Theta^2 d x&=& \left(\frac{N-2}{2N}-\frac{1}{q}\right)\int_{\mathbb{R}^N}\left|\nabla u_\Theta\right|^2 d x+\frac{1}{2 N} \int_{\mathbb{R}^N} \widetilde{V}(x) u_\Theta^2 d x \\
&&+\left(\frac{1}{2  }-\frac{1}{q}\right) \int_{\mathbb{R}^N} V(x) u_\Theta^2 d x-\frac{\beta}{q} \int_{\mathbb{R}^N}\left(qF(u_\Theta)-f(u_\Theta)u_\Theta)\right)d x \\
& \leq& \frac{(N-2) q-2 N}{2 N q} \int_{\mathbb{R}^N}\left|\nabla u_\Theta\right|^2 d x+\frac{\|\widetilde{V}\|_{\infty}}{2 N} \Theta+\frac{(q-2)\|V\|_{\infty}}{2 q} \Theta \\
&&-\frac{\beta(q-p_2)\alpha}{q}C_{N,p_1}   \Theta^{ \frac{2 p_1-N(p_1-2)}{4}}\left(\int_{\mathbb{R}^N}\left|\nabla u_\Theta\right|^2 d x\right)^{\frac{N(p_1-2)}{4}} \\
& \rightarrow&-\infty \ \text { as } \Theta \rightarrow 0,
\end{eqnarray*}
since $\frac{(N-2) q-2 N}{2 N q}<0$. Therefore, if $u_\Theta \neq 0$ for $\Theta>0$ small there exists $\Theta_0>0$ such that $\lambda_\Theta>0$ for $0<\Theta<\Theta_0$.

In order to complete the proof, we consider the case that there is a sequence $\Theta_k \rightarrow 0$ such that $u_{\Theta_k}=0$ for any $k$. Assume without loss of generality that $u_\Theta=0$ for any $\Theta \in(0, \Theta_1)$. Let $x_{r, \Theta} \in \Omega_r$ be such that $u_{r, \Theta}\left(x_{r, \Theta}\right)=\max\limits_{\Omega_r} u_{r, \Theta}$. In view of \eqref{eqq3.28}, there holds $\left|x_{r, \Theta}\right| \rightarrow \infty$
as $r \rightarrow \infty$. Otherwise, there exists $x_0 \in \mathbb{R}$ such that, up to a subsequence, $x_{r, \Theta} \rightarrow x_0$, and hence $u_\Theta(x_0) \geq d_\Theta>0$. This contradicts $u_\Theta=0$. We claim that $\operatorname{dist}(x_{r, \Theta}, \partial \Omega_r) \rightarrow \infty$ as $r \rightarrow \infty$. Arguing by contradiction we assume that $\liminf\limits_{r \rightarrow \infty} \operatorname{dist}(x_{r, \Theta}, \partial \Omega_r)=l<\infty$. It follows from \eqref{eqq3.28} that $l>0$. Let $w_r(x)=u_{r, \Theta}(x+x_{r, \Theta})$ for any $x \in \Sigma^r:=\{x \in \mathbb{R}^N: x+x_{r, \Theta} \in \Omega_r\}$. Then $w_r$ is bounded in $H^1(\mathbb{R}^N)$, and there is $w \in H^1(\mathbb{R}^N)$ such that $w_r \rightharpoonup w$ as $r \rightarrow \infty$. By the regularity theory of elliptic partial equations and $\liminf\limits_{r \rightarrow \infty} u_{r, \Theta}\left(x_{r, \Theta}\right)>d_\Theta>0$, we infer that $w(0) \geq d_\Theta>0$. Assume without loss of the generality that, up to a subsequence,
$$
\lim\limits_{r \rightarrow \infty} \frac{x_{r, \Theta}}{\left|x_{r, \Theta}\right|}=e_1.
$$
Setting
$$
\Sigma=\left\{x \in \mathbb{R}^N: x \cdot e_1<l\right\}=\left\{x \in \mathbb{R}^N: x_1<l\right\},
$$
we have $\phi(\cdot-x_{r, \Theta}) \in C_c^{\infty}(\Omega_r)$ for any $\phi \in C_c^{\infty}(\Sigma)$ and $r$ large enough. It then follows that
\begin{eqnarray}\label{eq3.32}
&&\int_{\Omega_r} \nabla u_{r, \Theta} \nabla \phi\left(\cdot-x_{r, \Theta}\right) d x+\int_{\Omega_r} V u_{r, \Theta} \phi\left(\cdot-x_{r, \Theta}\right) d x+\lambda_{r, \Theta} \int_{\Omega_r} u_{r, \Theta} \phi\left(\cdot-x_{r, \Theta}\right) d x \nonumber\\
&=&\int_{\Omega_r}\left|u_{r, \Theta}\right|^{q-2} u_{r, \Theta} \phi\left(\cdot-x_{r, \Theta}\right) d x+\beta \int_{\Omega_r}f(u_{r, \Theta}) \phi\left(\cdot-x_{r, \Theta}\right) d x .
\end{eqnarray}
Since $\left|x_{r, \Theta}\right| \rightarrow \infty$ as $r \rightarrow \infty$, it holds
\begin{eqnarray}\label{eq3.33}
\left|\int_{\Omega_r} V u_{r, \Theta} \phi\left(\cdot-x_{r, \Theta}\right) d x\right| & \leq& \int_{\text {Supp } \phi}\left|V\left(\cdot+x_{r, \Theta}\right) w_r \phi\right| d x \nonumber\\
& \leq&\left\|w_r\right\|_{2^*}\|\phi\|_{2^*}\left(\int_{\text {Supp } \phi}\left|V\left(\cdot+x_{r, \Theta}\right)\right|^{\frac{N}{2}} d x\right)^{\frac{2}{N}} \nonumber\\
& \leq&\left\|w_r\right\|_{2^*}\|\phi\|_{2^*}\left(\int_{\mathbb{R}^N \backslash B_{\frac{|x_{r, \Theta}|}{2}}}|V|^{\frac{N}{2}} d x\right)^{\frac{2}{N}}\nonumber\\
& \rightarrow& 0 \text { as } r \rightarrow \infty .
\end{eqnarray}
Letting $r \rightarrow \infty$ in \eqref{eq3.32}, we obtain for $\phi \in C_c^{\infty}(\Sigma)$:
$$
\int_{\Sigma} \nabla w \cdot \nabla \phi d x+\lambda_\Theta \int_{\Sigma} w \phi d x=\int_{\Sigma}|w|^{q-2} w \phi d x+\beta \int_{\Sigma}f(w) \phi d x .
$$
Thus $w \in H_0^1(\Sigma)$ is a weak solution of the equation
\begin{equation}\label{eq3.34}
 -\Delta w+\lambda_\Theta w=|w|^{q-2} w+\beta f(w) \quad \text { in } \Sigma .
\end{equation}
Hence we obtain a nontrivial nonnegative solution of \eqref{eq3.34} on a half space which is impossible by the Liouville theorem (see \cite{MEPL1982}). This proves that dist $\left(x_{r, \Theta}, \partial \Omega_r\right) \rightarrow \infty$ as $r \rightarrow \infty$. A similar argument as above shows that \eqref{eq3.34} holds for $\Sigma=\mathbb{R}^N$. Now we argue as in the case $u_\Theta \neq 0$ above that there exists $\Theta_2$ such that $\lambda_\Theta>0$ for any $0<\Theta<\Theta_2$.

Setting $\bar{\Theta}=\min \left\{\Theta_0, \Theta_1, \Theta_2\right\}$, the proof is complete.
\end{proof}
\noindent\textbf{Proof of Theorem \ref{t1.1}} The proof is an immediate consequence of Lemmas \ref{L3.4}, \ref{L3.5} and \ref{L3.6}.

\section{Proof of Theorem \ref{t1.2}}

In this section, we assume that the assumptions of Theorem \ref{t1.2} hold.
Since $\beta>0$,
\begin{eqnarray*}
I_r(u)
&\geq&\frac{1}{2}\left(1-\left\|V_{-}\right\|_{\frac{N}{2}} S^{-1}\right) \int_{\Omega_r}|\nabla u|^2 d x-\frac{C_{N, q} \Theta^{\frac{2 q-N(q-2)}{4}}}{q}\left(\int_{\Omega_r}|\nabla u|^2 d x\right)^{\frac{N(q-2)}{4}}\\
&&-\alpha\beta C_{N, p_1} \Theta^{\frac{2 p_1-N(p_1-2)}{4}}\left(\int_{\Omega_r}|\nabla u|^2 d x\right)^{\frac{N(p_1-2)}{4}}\\
&=&h_1(t),
\end{eqnarray*}
where
\begin{eqnarray*}
h_1(t)&:=&\frac{1}{2}\left(1-\left\|V_{-}\right\|_{\frac{N}{2}} S^{-1}\right) t^2-\frac{C_{N, q} \Theta^{\frac{2 q-N(q-2)}{4}}}{q}t^{\frac{N(q-2)}{2}}-\alpha\beta C_{N, p_1} \Theta^{\frac{2 p_1-N(p_1-2)}{4}}t^{\frac{N(p_1-2)}{2}}\\
&=&t^{\frac{N(p_1-2)}{2}}\left[\frac{1}{2}\left(1-\left\|V_{-}\right\|_{\frac{N}{2}} S^{-1}\right) t^{\frac{4-N(p_1-2)}{2}}-\frac{C_{N, q} \Theta^{\frac{2 q-N(q-2)}{4}}}{q}t^{\frac{N(q-p_1)}{2}}\right]\\
&&-\alpha\beta C_{N, p_1} \Theta^{\frac{2 p_1-N(p_1-2)}{4}}t^{\frac{N(p_1-2)}{2}}.
\end{eqnarray*}
Consider
$$
\psi(t):=\frac{1}{2}\left(1-\left\|V_{-}\right\|_{\frac{N}{2}} S^{-1}\right) t^{\frac{4-N(p_1-2)}{2}}-\frac{C_{N, q} \Theta^{\frac{2 q-N(q-2)}{4}}}{q}t^{\frac{N(q-p_1)}{2}}.
$$
Note that $\psi$ admits a unique maximum at
$$
\bar{t}=\left[\frac{q(4-N(p_1-2))\left(1-\|V_{-}\|_{\frac{N}{2}} S^{-1}\right)}{ 2N(q-p_1) C_{N, q}}\right]^{\frac{2}{N(q-2)-4}} \Theta^{\frac{N(q-2)-2 q}{2(N(q-2)-4)}} .
$$
By a direct calculation, we obtain
\begin{eqnarray*}
\psi(\bar{t})=\left[\frac{1-\left\|V_{-}\right\|_{\frac{N}{2}} S^{-1}}{2N(q-p_1)}\right]^{\frac{N(q-p_1)}{N(q-2)-4}}\left[\frac{q(4-N(p_1-2))}{C_{N, q}}\right]^{\frac{4-N(p_1-2)}{N(q-2)-4}}\left[N(q-2)-4\right].
\end{eqnarray*}
Hence,
$$
\psi(\bar{t})>\alpha\beta C_{N, p_1}\Theta^{\frac{2 p_1-N(p_1-2)}{4}}
$$
as long as
\begin{equation*}
  \Theta_V=\left[\frac{1-\left\|V_{-}\right\|_{\frac{N}{2}} S^{-1}}{2N(q-p_1)}\right]^{\frac{N}{2}}\left[\frac{q(4-N(p_1-2))}{C_{N, q}}\right]^{\frac{4-N(p_1-2)}{2 (q-p_1)}}\left[\frac{N(q-2)-4}{\alpha\beta C_{N, p_1}}\right]^{\frac{N(q-2)-4}{2 (q-p_1)}}.
\end{equation*}
Now, let $0<\Theta<\Theta_V$ be fixed, we obtain
\begin{equation}\label{eq4.1}
  \psi(\bar{t})>\alpha\beta C_{N, p_1}\Theta^{\frac{2 p_1-N(p_1-2)}{4}}
\end{equation}
and $h_1(\bar{t})>0$. In view of $2<p_1<2+\frac{4}{N}<q<2^*$ and \eqref{eq4.1}, there exist $0<R_1<T_\Theta<R_2$ such that $h_1(t)<0$ for $0<t<R_1$ and for $t>R_2, h_1(t)>0$ for $R_1<t<R_2$, and $h_1\left(T_\Theta\right)=\max\limits_{t \in \mathbb{R}^{+}} h_1(t)>0$.

Define
$$
\mathcal{V}_{r, \Theta}=\left\{u \in S_{r, \Theta}:\|\nabla u\|_2^2 \leq T_\Theta^2\right\} .
$$
Let $\theta$ be the principal eigenvalue of operator $-\Delta$ with Dirichlet boundary condition in $\Omega$, and let $|\Omega|$ be the volume of $\Omega$.
\begin{lemma}\label{L4.1}

(i) If $r<\frac{\sqrt{C \Theta}}{T_\Theta}$, then $\mathcal{V}_{r, \Theta}=\emptyset$.

(ii) If
$$
r> \max \left\{\frac{\sqrt{C\Theta}}{T_\Theta},\left(\frac{ \theta\left(1+\|V\|_{\frac{N}{2}} S^{-1}\right)}{2 \alpha_1\beta} \Theta^{\frac{2-p_2}{2}}|\Omega|^{\frac{p_2-2}{2}}\right)^{\frac{2}{N(p_2-2)+4}}\right\}
$$
then $\mathcal{V}_{r, \Theta} \neq \emptyset$ and
$$
e_{r, \Theta}:=\inf _{u \in \mathcal{V}_{r, \Theta}} I_r(u)<0
$$
is attained at some interior point $u_r>0$ of $\mathcal{V}_{r, \Theta}$. As a consequence, there exists a Lagrange multiplier $\lambda_r \in \mathbb{R}$ such that $\left(\lambda_r, u_r\right)$ is a solution of \eqref{eq2.1}. Moreover $\liminf\limits_{r \rightarrow \infty} \lambda_r>0$ holds true.
\end{lemma}
\begin{proof}
(i) The Poincar$\mathrm{\acute{e}}$ inequality implies there exists a positive constant $C$(only depend on $\Omega$) such that
$$
\int_{\Omega_r}|\nabla u|^2 d x=\frac{1}{r^2}\int_{\Omega}|\nabla u|^2 d x \geq \frac{C}{r^2}\int_{\Omega}| u|^2 d x  =\frac{C \Theta}{r^2}
$$
for any $u \in S_{r, \Theta}$. Since $T_\Theta$ is independent of $r$, there holds $\mathcal{V}_{r, \Theta}=\emptyset$ if and only if $r<\frac{\sqrt{C \Theta}}{T_\Theta}$.

(ii) Let $v_1 \in S_{1, \Theta}$ be the positive normalized eigenfunction corresponding to $\theta$. Setting
\begin{equation}\label{eq4.2}
  r_\Theta=\max \left\{\frac{\sqrt{C\Theta}}{T_\Theta},\left[\frac{ \theta\left(1+\|V\|_{\frac{N}{2}} S^{-1}\right)}{2 \alpha_1\beta} \Theta^{\frac{2-p_2}{2}}|\Omega|^{\frac{p_2-2}{2}}\right]^{\frac{2}{N(p_2-2)+4}}\right\}.
\end{equation}
Now, we construct for $r>r_\Theta$ a function $u_r \in S_{r, \Theta}$ such that $u_r \in \mathcal{V}_{r, \Theta}$ and $I_r\left(u_r\right)<0$. Clearly,
$$
\int_{\Omega}\left|\nabla v_1\right|^2 d x=\theta \Theta,\ \Theta=\int_{\Omega}\left|v_1\right|^2 d x \leq\left(\int_{\Omega}\left|v_1\right|^{p_2} d x\right)^{\frac{2}{p_2}}|\Omega|^{\frac{p_2-2}{p_2}} .
$$
Define $u_r \in S_{r, \Theta}$ by $u_r(x)=r^{-\frac{N}{2}} v_1\left(r^{-1} x\right)$ for $x \in \Omega_r$. Then
\begin{equation}\label{eq4.3}
  \int_{\Omega_r}\left|\nabla u_r\right|^2 d x=r^{-2} \theta \Theta \quad \text { and } \quad \int_{\Omega_r}\left|u_r\right|^{p_2} d x \geq r^{\frac{N(2-p_2)}{2}} \Theta^{\frac{p_2}{2}}|\Omega|^{\frac{2-p_2}{2}}.
\end{equation}
According to $\left(f_2\right)$, there exists a constant $\alpha_1>0$ such that
\begin{equation}\label{eq4.4}
  F(\tau)\geq\alpha_1\tau^{p_2}.
\end{equation}
By \eqref{eq4.2}, \eqref{eq4.3}, \eqref{eq4.4}, $2<p_2<2+\frac{4}{N}$ and a direct calculation we have $u_r \in \mathcal{V}_{r, \Theta}$ and
\begin{eqnarray*}
I_r\left(u_r\right)
& \leq&\frac{1}{2}\left(1+\|V\|_{\frac{N}{2}} S^{-1}\right) r^{-2} \theta \Theta-\alpha_1\beta r^{\frac{N(2-p_2)}{2}} \Theta^{\frac{p_2}{2}}|\Omega|^{\frac{2-p_2}{2}} \\
&<& 0 .
\end{eqnarray*}
It then follows from the Gagliardo-Nirenberg inequality that
\begin{eqnarray}\label{eq4.5}
I_r\left(u_r\right)
&\geq& \frac{1}{2}\left( 1-\|V_{-}\|_{\frac{N}{2}} S^{-1}\right) \int_{\Omega_r}|\nabla u|^2 d x-C_{N, p_1} \beta \Theta^{\frac{2 p_1-N(p_1-2)}{4}}\left(\int_{\Omega_r}|\nabla u|^2 d x\right)^{\frac{N(p_1-2)}{4}} \nonumber\\
& &-\frac{C_{N, q}}{q} \Theta^{\frac{2 q-N(q-2)}{4}}\left(\int_{\Omega_r}|\nabla u|^2 d x\right)^{\frac{N(q-2)}{4}} .
\end{eqnarray}
As a consequence $I_r$ is bounded from below in $\mathcal{V}_{r, \Theta}$. By the Ekeland principle there exists a sequence $\left\{u_{n, r}\right\} \subset \mathcal{V}_{r, \Theta}$ such that
$$
I_r(u_{n, r}) \rightarrow \inf\limits_{u \in \mathcal{V}_{r, \Theta}} I_r(u),\  I_r^{\prime}(u_{n, r})|_{T_{u_n, r} S_{r, \Theta}} \rightarrow 0 \ \text {as}\ n \rightarrow \infty
$$
Consequently there exists $u_r \in H_0^1(\Omega_r)$ such that $u_{n, r} \rightharpoonup u_r$ in $H_0^1(\Omega_r)$ and
\begin{equation*}
  u_{n, r} \rightarrow u_r \ \text{in}\ L^k(\Omega_r) \ \text{for all}\ 2 \leq k<2^*.
\end{equation*}
Moreover,
$
\left\|\nabla u_r\right\|_2^2 \leq \liminf\limits_{n \rightarrow \infty}\left\|\nabla u_{n, r}\right\|_2^2 \leq T_\Theta^2,
$
that is, $u_r \in \mathcal{V}_{r, \Theta}$. Note that
$$
\int_{\Omega_r} V u_{n, r}^2 d x \rightarrow \int_{\Omega_r} V u_r^2 d x \ \text {as}\ n \rightarrow \infty,
$$
hence
$$
e_{r, \Theta} \leq I_r(u_r) \leq \liminf\limits_{n \rightarrow \infty} I_r(u_{n, r})=e_{r, \Theta}.
$$
It follows that $u_{n, r} \rightarrow u_r$ in $H_0^1\left(\Omega_r\right)$, so $I_r(u_r)<0$. Therefore $u$ is an interior point of $\mathcal{V}_{r, \Theta}$ because $I_r(u) \geq h_1(T_\Theta)>0$ for any $u \in \partial \mathcal{V}_{r, \Theta}$ by \eqref{eq4.5}. The Lagrange multiplier theorem implies that there exists $\lambda_r \in \mathbb{R}$ such that $(\lambda_r, u_r)$ is a solution of \eqref{eq2.1}. Moreover,
\begin{eqnarray}\label{eq4.6}
\lambda_r \Theta & =&\int_{\Omega_r}\left|u_r\right|^q d x+\beta \int_{\Omega_r}f(u_r)u_r d x-\int_{\Omega_r}\left|\nabla u_r\right|^2 d x-\int_{\Omega_r} V u_r^2 d x \nonumber\\
& =&\int_{\Omega_r}\left|u_r\right|^q d x+\beta \int_{\Omega_r}f(u_r)u_r d x-\frac{2}{q} \int_{\Omega_r}|u_r|^q d x-2\beta \int_{\Omega_r}F(u_r) d x-2 I_r(u_r) \nonumber\\
& >&-2 I_r(u_r)=-2 e_{r, \Theta} .
\end{eqnarray}
It follows from the definition of $e_{r, \Theta}$ that $e_{r, \Theta}$ is nonincreasing with respect to $r$. Hence, $e_{r, \Theta} \leq e_{r_\Theta, \Theta}<0$ for any $r>r_\Theta$ and $0<\Theta<\Theta_V$. In view of \eqref{eq4.6}, we have $\liminf\limits_{r \rightarrow \infty} \lambda_r>0$. Finally, the strong maximum principle implies $u_r>0$.
\end{proof}

\noindent\textbf{Proof of Theorem \ref{t1.2}} The proof is a direct consequence of Lemma \ref{L4.1} and Lemma \ref{L3.5}.

\section{Proof of Theorem \ref{t1.3}}

In this subsection we assume that the assumptions of Theorem \ref{t1.3} hold. For $s \in\left[\frac{1}{2}, 1\right]$, $\beta>0$, we define the functional $J_{r, s}: S_{r, \Theta} \rightarrow \mathbb{R}$ by
$$
J_{r, s}(u)=\frac{1}{2} \int_{\Omega_r}|\nabla u|^2 d x+\frac{1}{2} \int_{\Omega_r} V u^2 d x-s\left(\frac{1}{q} \int_{\Omega_r}|u|^q d x+\beta \int_{\Omega_r}F(u) d x\right) .
$$
Note that if $u \in S_{r, \Theta}$ is a critical point of $J_{r, s}$ then there exists $\lambda \in \mathbb{R}$ such that $(\lambda, u)$ is a solution of the problem
\begin{equation}\label{eq5.1}
 \left\{\aligned
& -\Delta u+V u+\lambda u=s|u|^{q-2}u+s\beta f(u), &x\in \Omega_r, \\
& \int_{\Omega_r} |u|^2dx=\Theta,u\in H_0^1(\Omega_r), &x\in \Omega_r,
\endaligned
\right.
\end{equation}
\begin{lemma}\label{L5.1}
For $0<\Theta<\widetilde{\Theta}_V$ where $\widetilde{\Theta}_V$ is defined in Theorem \ref{t1.3}, there exist $\widetilde{r}_\Theta>0$ and $u^0, u^1 \in S_{r_\Theta, \Theta}$ such that

(i) For $r>\widetilde{r}_\Theta$ and $s \in\left[\frac{1}{2}, 1\right]$ we have $J_{r, s}\left(u^1\right) \leq 0$ and
$$
J_{r, s}\left(u^0\right)<\frac{N(q-2)-4}{4}\left(\frac{2\left(1-\left\|V_{-}\right\|_{\frac{N}{2}} S\right)}{N(q-2)}\right)^{\frac{N(q-2)}{N(q-2)-4}} A^{\frac{4}{4-N(q-2)}} \Theta^{\frac{N(q-2)-2 q}{N(q-2)-4}},
$$
where
$$
A=\left(\frac{C_{N,q}(q-2)(N(q-2)-4)}{q(p_1-2)(4-N(p_1-2))} +\frac{C_{N, q}}{q}\right).
$$
Moreover,
$$
\left\|\nabla u^0\right\|_2^2<\left[\frac{2\left(1-\left\|V_{-}\right\|_{\frac{N}{2}} S^{-1}\right)}{N(q-2) A}\right]^{\frac{4}{N(q-2)-4}} \Theta^{\frac{N(q-2)-2 q}{N(q-2)-4}}
$$
and
$$
\left\|\nabla u^1\right\|_2^2>\left[\frac{2\left(1-\left\|V_{-}\right\|_{\frac{N}{2}} S^{-1}\right)}{N(q-2) A}\right]^{\frac{4}{N(q-2)-4}} \Theta^{\frac{N(q-2)-2 q}{N(q-2)-4}}.
$$

(ii) If $u \in S_{r, \Theta}$ satisfies
$$
\|\nabla u\|_2^2=\left[\frac{2\left(1-\left\|V_{-}\right\|_{\frac{N}{2}} S^{-1}\right)}{N(q-2) A}\right]^{\frac{4}{N(q-2)-4}} \Theta^{\frac{N(q-2)-2 q}{N(q-2)-4}},
$$
then there holds
$$
J_{r, s}(u) \geq \frac{N(q-2)-4}{4}\left[\frac{2\left(1-\left\|V_{-}\right\|_{\frac{N}{2}} S\right)}{N(q-2)}\right]^{\frac{N(q-2)}{N(q-2)-4}} A^{\frac{4}{4-N(q-2)}} \Theta^{\frac{N(q-2)-2 q}{N(q-2)-4}}.
$$

(iii) Let
$$
m_{r, s}(\Theta)=\inf _{\gamma \in \Gamma_{r, \Theta}} \sup\limits_{t \in[0,1]} J_{r, s}(\gamma(t)),
$$
where
$$
\Gamma_{r, \Theta}=\left\{\gamma \in C\left([0,1], S_{r, \Theta}\right): \gamma(0)=u^0, \gamma(1)=u^1\right\}.
$$
Then
$$
m_{r, s}(\Theta) \geq \frac{N(q-2)-4}{4}\left[\frac{2\left(1-\left\|V_{-}\right\|_{\frac{N}{2}} S\right)}{N(q-2)}\right]^{\frac{N(q-2)}{N(q-2)-4}} A^{\frac{4}{4-N(q-2)}} \Theta^{\frac{N(q-2)-2 q}{N(q-2)-4}}
$$
and
\begin{equation*}
  m_{r, s}(\Theta) \leq \frac{N(q-2)-4}{2}\left(\frac{\theta\left(1+\|V\|_{\frac{N}{2}} S^{-1}\right)}{N(q-2)}\right)^{\frac{N(q-2)}{N(q-2)-4}}(4 q)^{\frac{4}{N(q-2)-4}}|\Omega|^{\frac{2(q-2)}{N(q-2)-4}} \Theta^{\frac{N(q-2)-2 q}{N(q-2)-4}}.
\end{equation*}
where $\theta$ is the principal eigenvalue of $-\Delta$ with Dirichlet boundary condition in $\Omega$.
\end{lemma}
\begin{proof}
Let $v_1 \in S_{1, \Theta}$ be the positive normalized eigenfunction of $-\Delta$ with Dirichlet boundary condition in $\Omega$ associated to $\theta$, then we have
\begin{equation}\label{eq5.2}
  \int_{\Omega}\left|\nabla v_1\right|^2 d x=\theta \Theta.
\end{equation}
By the H\"older inequality, we know
\begin{equation}\label{eq5.3}
  \int_{\Omega}|v_1(x)|^{p_2} dx\geq\Theta^{\frac{p_2}{2}}\cdot|\Omega|^{\frac{2-p_2}{2}}.
\end{equation}
Setting $v_t(x)=t^{\frac{N}{2}} v_1(t x)$ for $x \in B_{\frac{1}{t}}$ and $t>0$. Using \eqref{eq4.4}, \eqref{eq5.2}, \eqref{eq5.3} and $\frac{1}{2}\leq s\leq1$, we get
\begin{eqnarray}\label{eq5.4}
J_{\frac{1}{t}, s}\left(v_t\right)
&\leq & \frac{1}{2}\left(1+\|V\|_{\frac{N}{2}} S^{-1}\right) t^2 \theta\Theta- \frac{\beta}{2} \alpha_1 t^{\frac{N(p_2-2)}{2}}   \int_{\Omega}|v_1|^{p_2} d x-\frac{1}{2q} t^{\frac{N(q-2)}{2}} \Theta^{\frac{q}{2}}\cdot|\Omega|^{\frac{2-q}{2}} \nonumber\\
&\leq: & h_2(t) .
\end{eqnarray}
where
\begin{equation*}
  h_2(t)=\frac{1}{2}\left(1+\|V\|_{\frac{N}{2}} S^{-1}\right) t^2 \theta\Theta-\frac{1}{2q} t^{\frac{N(q-2)}{2}} \Theta^{\frac{q}{2}}\cdot|\Omega|^{\frac{2-q}{2}}
\end{equation*}
A simple computation shows that $h_2(t_0)=0$ for
$$
t_0:=\left[\left(1+\|V\|_{\frac{N}{2}} S^{-1}\right) q \theta \Theta^{\frac{2-q}{2}}|\Omega|^{\frac{q-2}{2}}\right]^{\frac{2}{N(q-2)-4}}
$$
and $h_2(t)<0$ for any $t>t_0, h_2(t)>0$ for any $0<t<t_0$. Moreover, $h_2(t)$ achieves its maximum at
$$
t_\Theta=\left[\frac{4 q\left(1+\|V\|_{\frac{N}{2}} S^{-1}\right) \theta}{N(q-2)} \Theta^{\frac{2-q}{2}}|\Omega|^{\frac{q-2}{2}}\right]^{\frac{2}{N(q-2)-4}} .
$$
This implies
\begin{equation}\label{eq5.5}
  J_{r, s}(v_{t_0})=J_{\frac{1}{t_0}, s}(v_{t_0}) \leq h_2(t_0)=0
\end{equation}
for any $r \geq \frac{1}{t_0}$ and $s \in\left[\frac{1}{2}, 1\right]$. There exists $0<t_1<t_\Theta$ such that for any $t \in\left[0, t_1\right]$,
\begin{equation}\label{eq5.6}
h_2(t)<\frac{N(q-2)-4}{4}\left(\frac{2\left(1-\left\|V_{-}\right\|_{\frac{N}{2}} S\right)}{N(q-2)}\right)^{\frac{N(q-2)}{N(q-2)-4}} A^{\frac{4}{4-N(q-2)}} \Theta^{\frac{N(q-2)-2 q}{N(q-2)-4}} .
\end{equation}
On the other hand, it follows from \eqref{eq3.5}, the Gagliardo-Nirenberg inequality and the H\"older inequality that
\begin{eqnarray}\label{eq5.7}
J_{r, s}(u)
&\geq&\frac{1}{2}\left(1-\left\|V_{-}\right\|_{\frac{N}{2}} S^{-1}\right) \int_{\Omega_r}|\nabla u|^2 d x-\frac{C_{N, q} \Theta^{\frac{2 q-N(q-2)}{4}}}{q}\left(\int_{\Omega_r}|\nabla u|^2 d x\right)^{\frac{N(q-2)}{4}}\nonumber\\
&&-\alpha\beta C_{N, p_1} \Theta^{\frac{2 p_1-N(p_1-2)}{4}}\left(\int_{\Omega_r}|\nabla u|^2 d x\right)^{\frac{N(p_1-2)}{4}}.
\end{eqnarray}
Define
\begin{eqnarray*}
g_1(t)&:=&\frac{1}{2}\left(1-\left\|V_{-}\right\|_{\frac{N}{2}} S^{-1}\right) t-\frac{C_{N, q} \Theta^{\frac{2 q-N(q-2)}{4}}}{q} t^{\frac{N(q-2)}{4}}-\alpha\beta C_{N, p_1} \Theta^{\frac{2 p_1-N(p_1-2)}{4}}t^{\frac{N(p_1-2)}{4}}\\
&=&t^{\frac{N(p_1-2)}{4}}\left[\frac{1}{2}\left(1-\left\|V_{-}\right\|_{\frac{N}{2}} S^{-1}\right) t^{\frac{4-N(p_1-2)}{4}}-\frac{C_{N, q} \Theta^{\frac{2 q-N(q-2)}{4}}}{q}t^{\frac{N(q-p_1)}{4}}\right]\\
&&-\alpha\beta C_{N, p_1} \Theta^{\frac{2 p_1-N(p_1-2)}{4}}t^{\frac{N(p_1-2)}{4}}
\end{eqnarray*}
In view of $2<p<2+\frac{2}{N}<q<2^*$ and the definition of $\widetilde{\Theta}_V$, there exist $0<l_1<l_M<l_2$ such that $ g_1(t)<0$ for any $0<t<l_1$ and $t>l_2, g_1(t)>0$ for $l_1<t<l_2$ and $g_1\left(l_M\right)=\max\limits_{t \in \mathbb{R}^{+}} g_1(t)>0$. Let
$$
t_2=\left(\frac{\alpha\beta q C_{N, p_1}(p_1-2)(4-N(p_1-2))}{C_{N, q}(q-2)(N(q-2)-4)}\right)^{\frac{4}{N(q-p_1)}} \Theta^{\frac{N-2}{N}} .
$$
Then by a direct calculation, we have $g_1^{\prime \prime}(t) \leq 0$ if and only if $t \geq t_2$. Hence
$$
\max\limits_{t \in \mathbb{R}^{+}} g_1(t)=\max _{t \in\left[t_2, \infty\right)} g_1(t).
$$
Note that for any $t \geq t_2$,
\begin{eqnarray}\label{eq5.8}
g_1(t) & =&\frac{1}{2}\left(1-\left\|V_{-}\right\|_{\frac{N}{2}} S^{-1}\right) t-\frac{C_{N, q} \Theta^{\frac{2 q-N(q-2)}{4}}}{q} t^{\frac{N(q-2)}{4}}-\alpha\beta C_{N, p_1} \Theta^{\frac{2 p_1-N(p_1-2)}{4}}t^{\frac{N(p_1-2)}{4}} \nonumber\\
& =&\frac{1}{2}\left(1-\left\|V_{-}\right\|_{\frac{N}{2}} S^{-1}\right) t-\alpha\beta C_{N, p_1} \Theta^{\frac{(q-p_1)(N-2)}{4}}\cdot\Theta^{\frac{2 q-N(q-2)}{4}}t^{\frac{N(p_1-2)}{4}}\nonumber\\
&&-\frac{C_{N, q} \Theta^{\frac{2 q-N(q-2)}{4}}}{q} t^{\frac{N(q-2)}{4}} \nonumber\\
&\geq& \frac{1}{2}\left(1-\left\|V_{-}\right\|_{\frac{N}{2}} S^{-1}\right) t-\left(\frac{C_{N,q}(q-2)(N(q-2)-4)}{q(p_1-2)(4-N(p_1-2))} +\frac{C_{N, q}}{q}\right) \Theta^{\frac{2 q-N(q-2)}{4}} \nonumber\\
&&\cdot t^{\frac{N(q-2)}{4}}\nonumber\\
& =:& g_2(t) .
\end{eqnarray}

Now, we will determine the value of $\widetilde{\Theta}_V$. In fact, $g_1\left(l_M\right)=\max\limits_{t \in \mathbb{R}^{+}} g_1(t)>0$ as long as $g_2(t_2)>0$, that is,
\begin{eqnarray*}
g_2(t_2)
&=&\frac{1}{2}\left(1-\left\|V_{-}\right\|_{\frac{N}{2}} S^{-1}\right)\left(\frac{\alpha\beta q C_{N, p_1}}{C_{N, q}A_{p_1,q}}\right)^{\frac{4}{N(q-p_1)}} \Theta^{\frac{N-2}{N}}-\left(\frac{C_{N,q}}{q}A_{p_1,q}+\frac{C_{N, q}}{q}\right)\cdot\Theta\\
&&\cdot\left(\frac{\alpha\beta q C_{N, p_1}}{C_{N, q}A_{p_1,q}}\right)^{\frac{q-2}{q-p_1}}\\
&>&0,
\end{eqnarray*}
where
$$
A_{p_1, q}=\frac{(q-2)(N(q-2)-4)}{(p_1-2)(4-N(p_1-2))}.
$$
Hence, we take
\begin{equation*}
 \widetilde{\Theta}_V= \frac{1}{2}\left(1-\left\|V_{-}\right\|_{\frac{N}{2}} S^{-1}\right)^{\frac{N}{2}}\left(\frac{C_{N,q}}{q}A_{p_1,q}+\frac{C_{N, q}}{q}\right)^{-\frac{N}{2}}\left(\frac{\alpha\beta q C_{N, p_1}}{C_{N, q}A_{p_1,q}}\right)^{\frac{N(q-2)-4}{2N(q-p_1)}}.
\end{equation*}
Let
$$
A=\left(\frac{C_{N,q}(q-2)(N(q-2)-4)}{q(p_1-2)(4-N(p_1-2))} +\frac{C_{N, q}}{q}\right)
$$
and
$$
t_g=\left[\frac{2\left(1-\left\|V_{-}\right\|_{\frac{N}{2}} S^{-1}\right)}{N(q-2) A}\right]^{\frac{4}{N(q-2)-4}} \Theta^{\frac{N(q-2)-2 q}{N(q-2)-4}},
$$
so that $t_g>t_2$ by the definition of $\widetilde{\Theta}_V, \max\limits_{t \in\left[t_2, \infty\right)} g_2(t)=g_2(t_g)$ and
\begin{eqnarray*}
\max\limits_{t \in \mathbb{R}^{+}} g_1(t) & \geq& \max\limits_{t \in\left[t_2, \infty\right)} g_2(t) \\
& =&\frac{(N(q-2)-4)}{4}\left[\frac{2\left(1-\left\|V_{-}\right\|_{\frac{N}{2}} S\right)}{N(q-2)}\right]^{\frac{N(q-2)}{N(q-2)-4}} A^{\frac{4}{4-N(q-2)}} \Theta^{\frac{N(q-2)-2 q}{N(q-2)-4}} .
\end{eqnarray*}
Set $\bar{r}_\Theta=\max \left\{\frac{1}{t_1}, \sqrt{\frac{2 \theta \Theta}{t_g}}\right\}$, then $v_{\frac{1}{\bar{r}_\Theta}} \in S_{r, \Theta}$ for any $r>\bar{r}_\Theta$, and
\begin{equation}\label{eq5.9}
  \left\|\nabla v_{\frac{1}{\bar{r}_\Theta}}\right\|_2^2=\left(\frac{1}{\bar{r}_\Theta}\right)^2\left\|\nabla v_1\right\|_2^2<t_g=\left[\frac{2\left(1-\left\|V_{-}\right\|_{\frac{N}{2}} S^{-1}\right)}{N(q-2) A}\right]^{\frac{4}{N(q-2)-4}} \Theta^{\frac{N(q-2)-2 q}{N(q-2)-4}}.
\end{equation}
Moreover,
\begin{equation}\label{eq5.10}
J_{\bar{r}_\Theta, s}\left(v_{\frac{1}{\bar{r}_\Theta}}\right) \leq h_2\left(\frac{1}{\bar{r}_\Theta}\right) \leq h_2\left(t_1\right) .
\end{equation}
Let $u^0=v_{\frac{1}{\bar{\tau}_\Theta}}, u^1=v_{t_0}$ and
$$
\widetilde{r}_\Theta=\max \left\{\frac{1}{t_0}, \bar{r}_\Theta\right\} .
$$
Then the statement (i) holds by \eqref{eq5.5}, \eqref{eq5.6}, \eqref{eq5.9}, \eqref{eq5.10}.

(ii) holds by \eqref{eq5.8} and a direct calculation.

(iii) In view of $J_{r, s}\left(u^1\right) \leq 0$ for any $\gamma \in \Gamma_{r, \Theta}$ and the definition of $t_0$, we have
$$
\|\nabla \gamma(0)\|_2^2<t_g<\|\nabla \gamma(1)\|_2^2 .
$$
It then follows from \eqref{eq5.8} that
\begin{eqnarray*}
\max\limits_{t \in[0,1]} J_{r, s}(\gamma(t)) & \geq& g_2\left(t_g\right) \\
& =&\frac{N(q-2)-4}{4}\left[\frac{2\left(1-\left\|V_{-}\right\|_{\frac{N}{2}} S^{-1}\right)}{N(q-2)}\right]^{\frac{N(q-2)}{N(q-2)-4}} A^{\frac{4}{4-N(q-2)}} \Theta^{\frac{N(q-2)-2 q}{N(q-2)-4}}
\end{eqnarray*}
for any $\gamma \in \Gamma_{r, \Theta}$, hence the first inequality in (iii) holds. We define a path $\gamma:[0,1] \rightarrow S_{r, \Theta}$ by
$$
\gamma(t): \Omega_r \rightarrow \mathbb{R}, \quad x \mapsto\left(\tau t_0+(1-\tau) \frac{1}{\widetilde{r}}_\Theta\right)^{\frac{N}{2}} v_1\left(\left(\tau t_0+(1-\tau) \frac{1}{\widetilde{r}_\Theta}\right) x\right) .
$$
Then $\gamma \in \Gamma_{r, \Theta}$, and the second inequality in (iii) follows from \eqref{eq5.4}.
\end{proof}
\begin{lemma}\label{L5.2}
Assume $0<\Theta<\widetilde{\Theta}_V$ where $\widetilde{\Theta}_V$ is given in Theorem \ref{t1.3}. Let $r>\widetilde{r}_\Theta$, where $\widetilde{r}_\Theta$ is defined in Lemma \ref{L5.1}. Then problem \eqref{eq5.1} admits a solution $\left(\lambda_{r, s}, u_{r, s}\right)$ for almost every $s \in\left[\frac{1}{2}, 1\right]$. Moreover, there hold $u_{r, s}>0$ and $J_{r, s}\left(u_{r, s}\right)=m_{r, s}(\Theta)$.
\end{lemma}
\begin{proof}
The proof is similar to the Lemma \ref{L3.2}. We omit it here.
\end{proof}

\begin{lemma}\label{L5.3}
For fixed $\Theta>0$ the set of solutions $u \in S_{r, \Theta}$ of \eqref{eq5.1} is bounded uniformly in $s$ and $r$.
\end{lemma}
\begin{proof}
Since $u$ is a solution of \eqref{eq5.1}, we have
\begin{equation*}
  \int_{\Omega_r}|\nabla u|^2 d x+\int_{\Omega_r} V u^2 d x=s \int_{\Omega_r}|u|^q d x+s\beta \int_{\Omega_r}f(u)u d x-\lambda \int_{\Omega_r}|u|^2 d x .
\end{equation*}
The Pohozaev identity implies
\begin{eqnarray*}
&&\frac{N-2}{2 N} \int_{\Omega_r}|\nabla u|^2 d x+\frac{1}{2 N} \int_{\partial \Omega_r}|\nabla u|^2(x \cdot \mathbf{n}) d \sigma+\frac{1}{2 N} \int_{\Omega_r}\widetilde{V}(x) u^2+\frac{1}{2} \int_{\Omega_r} V u^2 d x \\
&=&-\frac{\lambda}{2} \int_{\Omega_r}|u|^2 d x+\frac{s}{q} \int_{\Omega_r}|u|^q d x+s\beta \int_{\Omega_r}F(u) d x.
\end{eqnarray*}
It then follows from $\beta >0$ and $(f_2)$ that
\begin{eqnarray*}
&&\frac{1}{N} \int_{\Omega_r}|\nabla u|^2 d  x-\frac{1}{2 N} \int_{\partial \Omega_r}|\nabla u|^2(x \cdot \mathbf{n}) d \sigma-\frac{1}{2 N} \int_{\Omega_r}(\nabla V \cdot x) u^2 d x \\
& =&\frac{(q-2) s}{2 q} \int_{\Omega_r}|u|^q d x+ s\int_{\Omega_r}(\frac{\beta }{2  }f(u)u-\beta F(u))d x \\
& \geq&\frac{q-2}{2}\left(\frac{1}{2} \int_{\Omega_r}|\nabla u|^2 d x+\frac{1}{2} \int_{\Omega_r} V u^2 d x-m_{r, s}(\Theta)\right)+  s\frac{\beta (p_2-q) }{2  }  \int_{\Omega_r}F(u) d x .
\end{eqnarray*}
Using Gagliardo-Nirenberg inequality, \eqref{eq3.5} and (iii) in Lemma \ref{L5.1}, we have
 \begin{eqnarray*}
\frac{q-2}{2} m_{r, s}(\Theta)
&\geq & \frac{N(p_2-2)-4}{4 N} \int_{\Omega_r}|\nabla u|^2 d x-\Theta\left(\frac{1}{2 N}\|\nabla V \cdot x\|_{\infty}+\frac{p_2-2}{4}\|V\|_{\infty}\right)\\
&&+\frac{s\alpha\beta (p_2-q) }{2  }C_{N, p_1} \Theta^{\frac{2 p_1-N(p_1-2)}{4}}\left(\int_{\Omega_r}|\nabla u|^2 d x\right)^{\frac{N(p_1-2)}{4}}.
 \end{eqnarray*}
Since $2<p_1<2+\frac{4}{N}$, we can bound $\int_{\Omega_r}|\nabla u|^2 d x$ uniformly in $s$ and $r$.
\end{proof}

\begin{lemma}\label{L5.4}
Assume $0<\Theta<\widetilde{\Theta}_V$, where $\widetilde{\Theta}_V$ is given in Theorem \ref{t1.3}, and let $r>\widetilde{r}_\Theta$, where $\widetilde{r}_\Theta$ is defined in Lemma \ref{L5.1}. Then the following hold:

(i) Equation \eqref{eq2.1} admits a solution $\left(\lambda_{r, \Theta}, u_{r, \Theta}\right)$ for every $r>\widetilde{r}_\Theta$ such that $u_{r, \Theta}>0$ in $\Omega_r$.

(ii) There is $0<\bar{\Theta} \leq \widetilde{\Theta}_V$ such that
$$
\liminf\limits_{r \rightarrow \infty} \lambda_{r, \Theta}>0 \ \text {for any}\ 0<\Theta<\bar{\Theta} .
$$
\end{lemma}
\begin{proof}
The proof of $(i)$ is similar to that of Lemma \ref{L3.4}, we omit it. As be consider $H_0^1\left(\Omega_r\right)$ as a subspace of $H^1(\mathbb{R}^N)$ for every $r>0$. In view of Lemma \ref{L5.3}, there are $\lambda_\Theta$ and $u_\Theta \in H^1(\mathbb{R}^N)$ such that, up to a subsequence,
$$
u_{r, \Theta} \rightharpoonup u_\Theta \ \text {in}\ H^1(\mathbb{R}^N)\ \text {and}\  \lim _{r \rightarrow \infty} \lambda_{r, \Theta} \rightarrow \lambda_\Theta .
$$
Arguing by contradiction, we assume that $\lambda_{\Theta_n} \leq 0$ for some sequence $\Theta_n \rightarrow 0$. Let $\theta_r$ be the principal eigenvalue of $-\Delta$ with Dirichlet boundary condition in $\Omega_r$ and let $v_r>0$ be the corresponding normalized eigenfunction. Testing \eqref{eq2.1} with $v_r$, it holds
$$
\left(\theta_r+\lambda_{r, \Theta_n}\right) \int_{\Omega_r} u_{r, \Theta_n} v_r d x+\int_{\Omega_r} V u_{r, \Theta_n} v_r d x \geq 0 .
$$
In view of $\int_{\Omega_r} u_{r,\Theta_n} v_r d x>0$ and $\theta_r=r^{-2} \theta_1$, there holds
$$
\max\limits_{x \in \mathbb{R}^N} V+\lambda_{r,\Theta_n}+r^{-2} \theta_1 \geq 0 .
$$
Hence there exists $C>0$ independent of $n$ such that $\left|\lambda_{\Theta_n}\right| \leq C$ for any $n$.

\textbf{Case 1} There is subsequence denoted still by $\left\{\Theta_n\right\}$ such that $u_{\Theta_n}=0$. We first claim that there exists $d_n>0$ for any $n$ such that
\begin{equation}\label{eq5.11}
  \liminf\limits_{r \rightarrow \infty} \sup\limits_{z \in \mathbb{R}^N} \int_{B(z, 1)} u_{r, \Theta_n}^2 d x \geq d_n .
\end{equation}
Otherwise, the concentration compactness principle implies for every $n$ that
$$
u_{r, \Theta_n} \rightarrow 0 \text { in } L^t(\mathbb{R}^N) \quad \text { as } r \rightarrow \infty, \quad \text { for all } 2<t<2^* .
$$
By the diagonal principle, \eqref{eq2.1} and $\left|\lambda_{r, \Theta_n}\right| \leq 2 C$ for large $r$, there exists $r_n \rightarrow \infty$ such that
$$
\int_{\Omega_r}\left|\nabla u_{r_n, \Theta_n}\right|^2 d x \leq C
$$
for some $C$ independent of $n$, contradicting (iii) in Lemma \ref{L5.1} for large $n$. As a consequence \eqref{eq5.11} holds, and there is $z_{r, \Theta_n} \in \Omega_r$ with $\left|z_{r, \Theta_n}\right| \rightarrow \infty$ such that
$$
\int_{B\left(z_{r, \Theta_n}, 1\right)} u_{r, \Theta_n}^2 d x \geq \frac{d_n}{2} .
$$
Moreover, $\operatorname{dist}\left(z_{r, \Theta_n}, \partial \Omega_r\right) \rightarrow \infty$ as $r \rightarrow \infty$ by an argument similar to that in Lemma \ref{L3.6}. Now, for $n$ fixed let $v_r(x)=u_{r, \Theta_n}\left(x+z_{r, \Theta_n}\right)$ for $x \in \Sigma^r:=\left\{x \in \mathbb{R}^N: x+z_{r, \Theta_n} \in \Omega_r\right\}$. It follows from Lemma \ref{L5.3} that there is $v \in H^1(\mathbb{R}^N)$ with $v \neq 0$ such that $v_r \rightharpoonup v$. Observe that for every $\phi \in C_c^{\infty}\left(\mathbb{R}^N\right)$ there is $r$ large such that $\phi\left(\cdot-z_{r, \Theta_n}\right) \in C_c^{\infty}\left(\Omega_r\right)$ due to $\operatorname{dist}\left(z_{r, \Theta_n}, \partial \Omega_r\right) \rightarrow \infty$ as $r \rightarrow \infty$. It follows that
\begin{eqnarray}\label{eq5.12}
&&\int_{\Omega_r} \nabla u_{r, \Theta_n} \nabla \phi\left(\cdot-z_{r, \Theta_n}\right) d x+\int_{\Omega_r} V u_{r, \Theta_n} \phi\left(\cdot-z_{r, \Theta_n}\right) d x+\lambda_{r, \Theta_n} \int_{\Omega_r} u_{r, \Theta_n} \phi\left(\cdot-z_{r, \Theta_n}\right) d x \nonumber\\
&=&\int_{\Omega_r}\left|u_{r, \Theta_n}\right|^{q-2} u_{r, \Theta_n} \phi\left(\cdot-z_{r, \Theta_n}\right) d x+\beta \int_{\Omega_r}f(u_{r, \Theta_n}) \phi\left(\cdot-z_{r, \Theta_n}\right) d x .
\end{eqnarray}
Using $|z_{r,\Theta_n}|\rightarrow\infty$ as $r\rightarrow\infty$, it follows that
\begin{eqnarray*}
\left|\int_{\Omega_r} V u_{r, \Theta_n} \phi\left(\cdot-z_{r, \Theta_n}\right) d x\right| & \leq& \int_{S u p p \phi}\left|V\left(\cdot+z_{r, \Theta_n}\right) v_r \phi\right| d x \\
& \leq&\left\|v_r\right\|_{2^*}\|\phi\|_{2^*}\left(\int_{\mathbb{R}^N \backslash B_{\frac{|z_{r, \Theta_n }|}{2}}}|V|^{\frac{N}{2}} d x\right)^{\frac{2}{N}}  \\
& \rightarrow& 0 \quad \text { as } r \rightarrow \infty.
\end{eqnarray*}
Letting $r \rightarrow \infty$ in \eqref{eq5.12}, we get for every $\phi \in C_c^{\infty}(\mathbb{R}^N)$ :
$$
\int_{\mathbb{R}^N} \nabla v \cdot \nabla \phi d x+\lambda_{\Theta_n} \int_{\mathbb{R}^N} v \phi d x=\int_{\mathbb{R}^N}|v|^{q-2} v \phi d x+\beta \int_{\mathbb{R}^N}f(v) \phi d x.
$$
Therefore $v \in H^1(\mathbb{R}^N)$ is a weak solution of the equation
$$
-\Delta v+\lambda_{\Theta_n} v=\beta f(v)+|v|^{q-2} v \quad \text { in } \mathbb{R}^N
$$
and
$$
\int_{\mathbb{R}^N}|\nabla v|^2 d x+\lambda_{\Theta_n} \int_{\mathbb{R}^N}|v|^2 d x=\beta \int_{\mathbb{R}^N}f(v)v d x+\int_{\mathbb{R}^N}|v|^q d x .
$$
The Pohozaev identity implies
$$
\frac{N-2}{2 N} \int_{\mathbb{R}^N}|\nabla v|^2 d x+\frac{\lambda_{\Theta_n}}{2} \int_{\mathbb{R}^N}|v|^2 d x=\beta \int_{\mathbb{R}^N}F(v) d x+\frac{1}{q} \int_{\mathbb{R}^N}|v|^q d x,
$$
hence
\begin{eqnarray}\label{eq5.13}
&&\frac{\lambda_{\Theta_n}}{N} \int_{\mathbb{R}^N}|v|^2 d x\nonumber\\
&=&\frac{\beta(N-2)}{2 N } \int_{\mathbb{R}^N}\left[\frac{2N}{N-2}F(v)-f(v)v\right] d x+\frac{2 N-q(N-2)}{2 N q} \int_{\mathbb{R}^N}|v|^q d x\nonumber\\
&\geq&\frac{\beta(N-2)}{2 N }\left(\frac{2N}{N-2}-p_1\right) \int_{\mathbb{R}^N}F(v) d x+\frac{2 N-q(N-2)}{2 N q} \int_{\mathbb{R}^N}|v|^q d x.
\end{eqnarray}
We have $\lambda_{\Theta_n}>0$ because of $2<p_1<2+\frac{4}{N}<q<2^*$, which is a contradiction.

\textbf{Case 2} $u_{\Theta_n} \neq 0$ for $n$ large. Note that $u_{\Theta_n}$ satisfies
\begin{equation}\label{eq5.14}
  -\Delta u_{\Theta_n}+V u_{\Theta_n}+\lambda_{\Theta_n} u_{\Theta_n}=\beta f(u_{\Theta_n})+\left|u_{\Theta_n}\right|^{q-2} u_{\Theta_n} .
\end{equation}
If $v_{r, \Theta_n}:=u_{r, \Theta_n}-u_{\Theta_n}$ satisfies
\begin{equation}\label{eq5.15}
  \limsup\limits_{r \rightarrow \infty} \max\limits_{z \in \mathbb{R}^N} \int_{B(z, 1)} v_{r, \Theta_n}^2 d x=0,
\end{equation}
then the concentration compactness principle implies $u_{r, \Theta_n} \rightarrow u_{\Theta_n}$ in $L^t(\mathbb{R}^N)$ for any $2<t<2^*$. It then follows from \eqref{eq2.1} and \eqref{eq5.14} that
\begin{eqnarray*}
\int_{\Omega_r}\left|\nabla u_{r, \Theta_n}\right|^2 d x+\Theta_n \lambda_{r, \Theta_n} & =&\beta \int_{\Omega_r}f(u_{r, \Theta_n})u_{r, \Theta_n} d x+\int_{\Omega_r}\left|u_{r, \Theta_n}\right|^q d x-\int_{\Omega_r} V u_{r, \Theta_n}^2 d x \\
& \rightarrow& \beta \int_{\mathbb{R}^N}f(u_{\Theta_n})u_{r, \Theta_n} d x+\int_{\mathbb{R}^N}\left|u_{\Theta_n}\right|^q d x-\int_{\mathbb{R}^N} V u_{\Theta_n}^2 d x \\
& =&\int_{\mathbb{R}^N}\left|\nabla u_{\Theta_n}\right|^2 d x+\lambda_{\Theta_n} \int_{\mathbb{R}^N} u_{\Theta_n}^2 d x.
\end{eqnarray*}
Using $\lambda_{r, \Theta_n} \rightarrow \lambda_{\Theta_n}$ as $r \rightarrow \infty$, we further have
\begin{equation}\label{eq5.16}
\int_{\Omega_r}\left|\nabla u_{r, \Theta_n}\right|^2 d x+\Theta_n \lambda_{\Theta_n} \rightarrow \int_{\mathbb{R}^N}\left|\nabla u_{\Theta_n}\right|^2 d x+\lambda_{\Theta_n} \int_{\mathbb{R}^N} u_{\Theta_n}^2 d x \quad \text { as } r \rightarrow \infty .
\end{equation}
Using \eqref{eq5.16}, (iii) in Lemma \ref{L5.1} and $\left|\lambda_{\Theta_n}\right| \leq C$ for large $n$, there holds
$$
\int_{\mathbb{R}^N}\left|\nabla u_{\Theta_n}\right|^2 d x \rightarrow \infty \quad \text { as } n \rightarrow \infty .
$$
By \eqref{eq5.14} and the Pohozaev identity
\begin{eqnarray*}
& &\frac{N-2}{2 N} \int_{\mathbb{R}^N}\left|\nabla u_{\Theta_n}\right|^2 d x+\frac{1}{2 N} \int_{\mathbb{R}^N} \widetilde{V} u_{\Theta_n}^2dx+\frac{1}{2} \int_{\mathbb{R}^N} V(x) u_{\Theta_n}^2 d x+\frac{\lambda_\Theta}{2} \int_{\mathbb{R}^N} u_{\Theta_n}^2 d x \nonumber\\
& =&\frac{1}{q} \int_{\mathbb{R}^N}\left|u_{\Theta_n}\right|^q d x+\beta \int_{\mathbb{R}^N}F(u_{\Theta_n}) d x .
\end{eqnarray*}
It holds that
\begin{eqnarray*}
0 &\leq & \frac{(2-q) \lambda_{\Theta_n}}{2 q} \int_{\mathbb{R}^N} u_{\Theta_n}^2 d x \\
&\leq & \frac{(N-2) q-2 N}{2 N q} \int_{\mathbb{R}^N}\left|\nabla u_{\Theta_n}\right|^2 d x+\frac{\|\widetilde{V}\|_{\infty}}{2 N} \Theta_n+\frac{(q-2)\|V\|_{\infty}}{2 q} \Theta_n \\
&\rightarrow & -\infty  \text { as } n \rightarrow \infty .
\end{eqnarray*}
Therefore \eqref{eq5.15}  cannot occur. Consequently there exist $d_n>0$ and $z_{r, \Theta_n} \in \Omega_r$ with $\left|z_{r, \Theta_n}\right| \rightarrow \infty$ as $r \rightarrow \infty$ such that
$$
\int_{B\left(z_{r, \Theta_n}, 1\right)} v_{r, \Theta_n}^2 d x>d_n.
$$
Then $\widetilde{v}_{r, \Theta_n}:=v_{r, \Theta_n}\left(\cdot+z_{r, \Theta_n}\right) \rightharpoonup \widetilde{v}_{\Theta_n} \neq 0$, and $\widetilde{v}_{\Theta_n}$ is a nonnegative solution of
$$
-\Delta v+\lambda_{\Theta_n} v=\beta f(v) v+|v|^{q-2} v \quad \text { in } \mathbb{R}^N .
$$
In fact, we have $\liminf\limits_{r \rightarrow \infty} \operatorname{dist}\left(z_{r, \Theta_n}, \partial \Omega_r\right)=\infty$ by the Liouville theorem on the half space. It follows from an argument similar to that of \eqref{eq5.13} that $\lambda_{\Theta_n}>0$ for large $n$, which is a contradiction.
\end{proof}

\noindent\textbf{Proof of Theorem \ref{t1.3}} The proof is a direct consequence of Lemma \ref{L5.4} and  Lemma \ref{L3.5}.

\section{Proof of Theorem \ref{t1.7}}

In this section we assume that the assumptions of Theorem \ref{t1.7} hold. Define the functional $\mathcal{I}_{r}: S_{r, \Theta} \rightarrow \mathbb{R}$ by
\begin{equation}\label{eq7.1}
  \mathcal{I}_{r}(u)=\frac{1}{2} \int_{\Omega_r}|\nabla u|^2 d x+\frac{1}{2} \int_{\Omega_r} V  u^2 d x-\frac{1}{2^*} \int_{\Omega_r}|u|^{2^*} d x-\beta \int_{\Omega_r}F(u) d x .
\end{equation}
Note that if $u \in S_{r, \Theta}$ is a critical point of $\mathcal{I}_{r, s}$, then there exists $\lambda \in \mathbb{R}$ such that $(\lambda, u)$ is a solution of the equation
\begin{equation}\label{eq7.2}
 \left\{\aligned
& -\Delta u+V u+\lambda u=|u|^{2^*-2}u+\beta f(u), &x\in \Omega_r, \\
& \int_{\Omega_r} |u|^2dx=\Theta,u\in H_0^1(\Omega_r), &x\in \Omega_r.
\endaligned
\right.
\end{equation}
Since $\beta>0$,
\begin{eqnarray*}
\mathcal{I}_r(u)
&\geq&\frac{1}{2}\left(1-\left\|V_{-}\right\|_{\frac{N}{2}} S^{-1}\right) \int_{\Omega_r}|\nabla u|^2 d x-\frac{1}{2^*\cdot S^{
\frac{2^*}{2}}}\left(\int_{\Omega_r}|\nabla u|^2 d x\right)^{\frac{2^*}{2}}\\
&&-\alpha\beta C_{N, p_1} \Theta^{\frac{2 p_1-N(p_1-2)}{4}}\left(\int_{\Omega_r}|\nabla u|^2 d x\right)^{\frac{N(p_1-2)}{4}}\\
&=&\widetilde{h_1}(t),
\end{eqnarray*}
where
\begin{eqnarray*}
\widetilde{h_1}(t)&:=&\frac{1}{2}\left(1-\left\|V_{-}\right\|_{\frac{N}{2}} S^{-1}\right) t^2-\frac{1}{2^*\cdot S^{
\frac{2^*}{2}}}t^{2^*}-\alpha\beta C_{N, p_1} \Theta^{\frac{2 p_1-N(p_1-2)}{4}}t^{\frac{N(p_1-2)}{2}}.
\end{eqnarray*}
Consider
$$
\widehat{\psi}(t):=\frac{1}{2}\left(1-\left\|V_{-}\right\|_{\frac{N}{2}} S^{-1}\right) t^2-\frac{1}{2^*\cdot S^{\frac{2^*}{2}}}t^{2^*}.
$$
Note that $\widehat{\psi}$ admits a unique maximum at
$$
\widehat{t}=\left[\left(1-\|V_{-}\|_{\frac{N}{2}} S^{-1}\right)S^{\frac{2^*}{2}}\right]^{\frac{1}{2^*-2}} .
$$
By a direct calculation, we obtain
\begin{eqnarray*}
\widehat{\psi}(\widehat{t})
=\frac{1}{N}\left(1-\left\|V_{-}\right\|_{\frac{N}{2}} S^{-1}\right)^{\frac{N}{2}}S^{\frac{N}{2}}.
\end{eqnarray*}
Hence,
$$
\widehat{\psi}(\widehat{t})>\alpha\beta C_{N, p_1} \Theta^{\frac{2 p_1-N(p_1-2)}{4}}\widehat{t}^{\frac{N(p_1-2)}{2}}
$$
as long as
\begin{equation*}
  \Theta_V=\left( \frac{1}{N\alpha\beta C_{N, p_1}}\right)^{\frac{4}{2 p_1-N(p_1-2)}}\left(1-\left\|V_{-}\right\|_{\frac{N}{2}} S^{-1}\right)^{\frac{N}{2}}S^{\frac{N}{2}\cdot\frac{4-N(p_1-2)}{2 p_1-N(p_1-2)}}.
\end{equation*}
Now, let $0<\Theta<\Theta_V$ be fixed, we obtain
\begin{equation*}
\widehat{\psi}(\widehat{t})>\alpha\beta C_{N, p_1} \Theta^{\frac{2 p_1-N(p_1-2)}{4}}\widehat{t}^{\frac{N(p_1-2)}{2}}
\end{equation*}
and $\widetilde{h}_1(\widehat{t})>0$. In view of $2<p_1<2+\frac{4}{N}<2^*$, there exist $0<\widetilde{R_1}<\widetilde{T_\Theta}<\widetilde{R_2}$ such that $\widetilde{h}_1(t)<0$ for $0<t<\widetilde{R_1}$ and for $t>\widetilde{R_2}, \widetilde{h}_1(t)>0$ for $\widetilde{R_1}<t<\widetilde{R_2}$, and $\widetilde{h}_1(\widetilde{T_\Theta})=\max\limits_{t \in \mathbb{R}^{+}} \widetilde{h}_1(t)>0$.

Define
$$
\widetilde{\mathcal{V}}_{r, \Theta}=\left\{u \in S_{r, \Theta}:\|\nabla u\|_2^2 \leq \widetilde{T_\Theta}^2\right\} .
$$
Let $\theta$ be the principal eigenvalue of operator $-\Delta$ with Dirichlet boundary condition in $\Omega$, and let $|\Omega|$ be the volume of $\Omega$.
\begin{lemma}\label{L7.1}

(i) If $r<\frac{\sqrt{C \Theta}}{\widetilde{T_\Theta}}$, then $\widetilde{\mathcal{V}}_{r, \Theta}=\emptyset$.

(ii) If
$$
r> \max \left\{\frac{\sqrt{C\Theta}}{\widetilde{T_\Theta}},\left(\frac{ \theta\left(1+\|V\|_{\frac{N}{2}} S^{-1}\right)}{2 \alpha_1\beta} \Theta^{\frac{2-p_2}{2}}|\Omega|^{\frac{p_2-2}{2}}\right)^{\frac{2}{N(p_2-2)+4}}\right\}
$$
then $\widetilde{\mathcal{V}}_{r, \Theta} \neq \emptyset$ and
$$
\widetilde{e}_{r, \Theta}:=\inf _{u \in \widetilde{\mathcal{V}}_{r, \Theta}} \mathcal{I}_r(u)<0
$$
is attained at some interior point $u_r>0$ of $\widetilde{\mathcal{V}}_{r, \Theta}$. As a consequence, there exists a Lagrange multiplier $\lambda_r \in \mathbb{R}$ such that $\left(\lambda_r, u_r\right)$ is a solution of \eqref{eq7.1}. Moreover $\liminf\limits_{r \rightarrow \infty} \lambda_r>0$ holds true.
\end{lemma}
\begin{proof}
(i)The proof is similar to the Lemma \ref{L4.1}. (ii) Let $v_1 \in S_{1, \Theta}$ be the positive normalized eigenfunction corresponding to $\theta$. Setting
\begin{equation}\label{eq7.3}
  r_\Theta=\max \left\{\frac{\sqrt{C\Theta}}{\widetilde{T_\Theta}},\left(\frac{ \theta\left(1+\|V\|_{\frac{N}{2}} S^{-1}\right)}{2 \alpha_1\beta} \Theta^{\frac{2-p_2}{2}}|\Omega|^{\frac{p_2-2}{2}}\right)^{\frac{2}{N(p_2-2)+4}}\right\}.
\end{equation}
Now, we construct for $r>r_\Theta$ a function $u_r \in S_{r, \Theta}$ such that $u_r \in \widetilde{\mathcal{V}}_{r, \Theta}$ and $\mathcal{I}_r\left(u_r\right)<0$.
By \eqref{eq7.3}, \eqref{eq4.3}, \eqref{eq4.4}, $2<p_2<2+\frac{4}{N}$ and a direct calculation, we have $u_r \in \widetilde{\mathcal{V}}_{r, \Theta}$ and
\begin{eqnarray*}
\mathcal{I}_r\left(u_r\right)
&\leq&\frac{1}{2} \int_{\Omega_r}\left|\nabla u_r\right|^2 d x+\frac{1}{2} \int_{\Omega_r} V u_r^2 d x-\frac{1}{2^*} \int_{\Omega_r}\left|u_r\right|^{2^*} d x-\alpha_1\beta \int_{\Omega_r}\left|u_r\right|^{p_2} d x \\
& \leq&\frac{1}{2}\left(1+\|V\|_{\frac{N}{2}} S^{-1}\right) r^{-2} \theta \Theta-\alpha_1\beta r^{\frac{N(2-p_2)}{2}} \Theta^{\frac{p_2}{2}}|\Omega|^{\frac{2-p_2}{2}} \\
&<& 0 .
\end{eqnarray*}
It then follows from the Gagliardo-Nirenberg inequality that
\begin{eqnarray}\label{eq7.6}
\mathcal{I}_r\left(u_r\right)
&\geq& \frac{1}{2}\left( 1-\|V_{-}\|_{\frac{N}{2}} S^{-1}\right) \int_{\Omega_r}|\nabla u|^2 d x-C_{N, p_1} \beta \Theta^{\frac{2 p_1-N(p_1-2)}{4}}\left(\int_{\Omega_r}|\nabla u|^2 d x\right)^{\frac{N(p_1-2)}{4}} \nonumber\\
& &-\frac{1}{2^*\cdot S^{
\frac{2^*}{2}}}\left(\int_{\Omega_r}|\nabla u|^2 d x\right)^{\frac{2^*}{2}} .
\end{eqnarray}
As a consequence $\mathcal{I}_r$ is bounded from below in $\widetilde{\mathcal{V}}_{r, \Theta}$. By the Ekeland principle there exists a sequence $\left\{u_{n, r}\right\} \subset \widetilde{\mathcal{V}}_{r, \Theta}$ such that
\begin{equation}\label{eq7.7}
\mathcal{I}_r(u_{n, r}) \rightarrow \inf\limits_{u \in \widetilde{\mathcal{V}}_{r, \Theta}} \mathcal{I}_r(u),\  \mathcal{I}_r^{\prime}(u_{n, r})|_{T_{u_n, r} S_{r, \Theta}} \rightarrow 0 \ \text {as}\ n \rightarrow \infty
\end{equation}
Consequently there exists $u_r \in H_0^1(\Omega_r)$ such that
\begin{equation*}
  u_{n, r} \rightharpoonup u_r \ \text{in}\ H_0^1(\Omega_r)
\end{equation*}
and
\begin{equation}\label{eq7.8}
  u_{n, r} \rightarrow u_r \ \text{in}\ L^k(\Omega_r) \ \text{for all}\ 2 \leq k<2^*.
\end{equation}

We claim now that the weak limit $u_r$ does not vanish identically. Suppose by contradiction that $u_r \equiv 0$. Since $\left\{u_{n, r}\right\}$ is bounded in $H^1(\Omega_r)$, up to a subsequence we have that $\|\nabla u_{n, r}\|_2^2 \rightarrow \ell \in \mathbb{R}$. Using $(f_2)$, \eqref{eq7.7}, \eqref{eq7.8}, we have
\begin{eqnarray*}
  \langle\mathcal{I}_r^{\prime}(u_{n, r}),u_{n, r}\rangle&=& \int_{\Omega_r}\left|\nabla u_{n, r}\right|^2 d x+  \int_{\Omega_r} V u_{n, r}^2 d x- \int_{\Omega_r}\left|u_{n, r}\right|^{2^*} d x-\beta \int_{\Omega_r}f(u_{n, r})u_{n, r} d x\\
  &\rightarrow&0,
\end{eqnarray*}
hence
$$
\|u_{n, r}\|_{2^*}^{2^*}=\|\nabla u_{n, r}\|_2^2 \rightarrow \ell
$$
as well. Therefore, by the Sobolev inequality $\ell \geq S \ell^{\frac{2}{2^*}}$, and we deduce that either $\ell=0$, or $\ell \geq S^{\frac{N}{2}}$.
Let us suppose at first that $\ell \geq S^{N / 2}$. Since $\mathcal{I}_r(u_{n, r}) \rightarrow \widetilde{e}_{r, \Theta}<0$, we have that
\begin{eqnarray*}
0>\widetilde{e}_{r, \Theta}+o(1) & =&I_r(u_{n, r})\\
& =&\frac{1}{2} \int_{\Omega_r}\left|\nabla u_{n, r}\right|^2 d x+\frac{1}{2} \int_{\Omega_r} V u_{n, r}^2 d x-\frac{1}{2^*} \int_{\Omega_r}\left|u_{n, r}\right|^{2^*} d x-\beta \int_{\Omega_r}F(u_{n, r}) d x \\
& =&\frac{1}{N}\|\nabla u_{n, r}\|_2^2+o(1)=\frac{\ell}{N}+o(1),
\end{eqnarray*}
which is not possible. If instead $\ell=0$, we have $\|u_{n, r}\|_{2^*}\rightarrow0$, $\|\nabla u_{n, r}\|_2\rightarrow0$ and
$F(u_{n, r})\rightarrow0$. But then $\mathcal{I}_r(u_{n, r}) \rightarrow 0\neq\widetilde{e}_{r, \Theta}$, which gives again a contradiction. Thus, $u_r$ does not vanish identically.

Since $\{u_{n, r}\}$ is a bounded minimization sequence for $\mathcal{I}_r(u_{n, r})|_{S_{r, \Theta}} $, there exists $\{\lambda_n\} \subset \mathbb{R}$ such that for every $\varphi \in H^1(\Omega_r)$,
\begin{equation}\label{eq7.9}
  \int_{\Omega_r} \nabla u_{n, r} \cdot \nabla \varphi+\int_{\Omega_r}Vu_{n, r}\varphi+\lambda_n u_{n, r} \varphi-\beta f(u_{n, r})   \varphi-\left|u_{n, r}\right|^{2^*-2} u_{n, r} \varphi=o(1)\|\varphi\|
\end{equation}
as $n \rightarrow \infty$, by the Lagrange multipliers rule. Choosing $\varphi=u_{n, r}$, we deduce that $\left\{\lambda_n\right\}$ is bounded as well, and hence up to a subsequence $\lambda_n \rightarrow \lambda_r \in \mathbb{R}$. Moreover, passing to the limit in \eqref{eq7.9} by weak convergence, we obtain
\begin{equation*}
  -\Delta u_r+V u_r+\lambda_r u_r=|u_r|^{2^*-2}u_r+\beta f(u_r),\ x\in \Omega_r.
\end{equation*}
Recalling that $v_{n, r}=u_{n, r}-u_r \rightharpoonup 0$ in $H_0^1(\Omega_r)$, we know
$$
\|\nabla u_{n, r}\|_2^2=\|\nabla u_{r}\|_2^2+\|\nabla v_{n, r}\|_2^2+o(1).
$$
By the Br$\mathrm{\acute{e}}$zis-Lieb lemma \cite{HB1983}, we have
$$
\|u_{n, r}\|_{2^*}^{2^*}=\|u_{r}\|_{2^*}^{2^*}+\|v_{n, r}\|_{2^*}^{2^*}+o(1) .
$$
Moreover,
$$
\left\|\nabla u_r\right\|_2^2 \leq \liminf _{n \rightarrow \infty}\left\|\nabla u_{n, r}\right\|_2^2 \leq \widetilde{T_\Theta}^2,
$$
that is, $u_r \in \widetilde{\mathcal{V}}_{r, \Theta}$. Note that
$$
\int_{\Omega_r} V u_{n, r}^2 d x \rightarrow \int_{\Omega_r} V u_r^2 d x \ \text {as}\ n \rightarrow \infty,
$$
hence
$$
\|v_{n, r}\|_{2^*}^{2^*}=\|\nabla v_{n, r}\|_2^2 \rightarrow \ell
$$
as well. Therefore, $\ell \geq S \ell^{\frac{2}{2^*}}$, and we deduce that either $\ell=0$, or $\ell \geq S^{\frac{N}{2}}$.
Let us suppose at first that $\ell \geq S^{N / 2}$. Since $\mathcal{I}_r(v_{n, r}) \rightarrow  0$, we have that
\begin{eqnarray*}
o(1) & =&\mathcal{I}_r(v_{n, r})\\
& =&\frac{1}{2} \int_{\Omega_r}\left|\nabla v_{n, r}\right|^2 d x+\frac{1}{2} \int_{\Omega_r} V v_{n, r}^2 d x-\frac{1}{2^*} \int_{\Omega_r}\left|v_{n, r}\right|^{2^*} d x-\beta \int_{\Omega_r}F(v_{n, r}) d x \\
& =&\frac{1}{2}\|\nabla v_{n, r}\|_2^2+o(1)=\frac{\ell}{N}+o(1),
\end{eqnarray*}
which is not possible. If instead $\ell=0$, we have that $u_{n, r} \rightarrow u_r$ in $H_0^1\left(\Omega_r\right)$, so $\mathcal{I}_r(u_r)<0$. Therefore $u $ is an interior point of $\widetilde{\mathcal{V}}_{r, \Theta}$ because $\mathcal{I}_r(u) \geq \widetilde{h_1}(\widetilde{T_\Theta})>0$ for any $u \in \partial \widetilde{\mathcal{V}}_{r, \Theta}$ by \eqref{eq7.6}. The Lagrange multiplier theorem implies that there exists $\lambda_r \in \mathbb{R}$ such that $(\lambda_r, u_r)$ is a solution of \eqref{eq7.1}. Moreover,
\begin{eqnarray}\label{eq7.10}
\lambda_r \Theta
& =&\int_{\Omega_r}\left|u_r\right|^{2^*} d x+\beta \int_{\Omega_r}f(u_r)u_r d x-\frac{2}{2^*} \int_{\Omega_r}|u_r|^{2^*} d x-2\beta \int_{\Omega_r}F(u_r) d x-2 I_r(u_r) \nonumber\\
& =&\frac{2^*-2}{2^*} \int_{\Omega_r}\left|u_r\right|^{2^*} d x+\beta \int_{\Omega_r}[f(u_r)u_r-2F(u_r)] d x-2 I_r(u_r) \nonumber\\
& >&-2 \mathcal{I}_r(u_r)=-2 \widetilde{e}_{r, \Theta} .
\end{eqnarray}
It follows from the definition of $\widetilde{e}_{r, \Theta}$ that $\widetilde{e}_{r, \Theta}$ is nonincreasing with respect to $r$. Hence, $\widetilde{e}_{r, \Theta} \leq \widetilde{e}_{r_\Theta, \Theta}<0$ for any $r>r_\Theta$ and $0<\Theta<\Theta_V$. In view of \eqref{eq7.10}, we have
$$
\liminf\limits_{r \rightarrow \infty} \lambda_r>0 .
$$
Finally, the strong maximum principle implies $u_r>0$.
\end{proof}

\noindent\textbf{Proof of Theorem \ref{t1.7}} The proof is a direct consequence of Lemma \ref{L7.1} and Lemma \ref{L3.5}.

\section{Proof of Theorem \ref{t1.8}}
 In this subsection, we assume $\beta \leq 0$ and the assumptions of Theorem \ref{t1.8} hold. Consider the following equation
 \begin{equation}\label{eeq8.1}
 \left\{\aligned
& -\Delta u+V u+\lambda u=|u|^{2^*-2}u+\beta f(u), &x\in \Omega_r, \\
& \int_{\Omega_r} |u|^2dx=\Theta,u\in H_0^1(\Omega_r), &x\in \Omega_r.
\endaligned
\right.
\end{equation}
For $\frac{1}{2} \leq s \leq 1$, we define the functional $\mathcal{I}_{r, s}: S_{r, \Theta} \rightarrow \mathbb{R}$ by
\begin{equation}\label{eq8.1}
  \mathcal{I}_{r, s}(u)=\frac{1}{2} \int_{\Omega_r}|\nabla u|^2 d x+\frac{1}{2} \int_{\Omega_r} V  u^2 d x-\frac{s}{2^*} \int_{\Omega_r}|u|^{2^*} d x-\beta \int_{\Omega_r}F(u) d x .
\end{equation}
Note that if $u \in S_{r, \Theta}$ is a critical point of $\mathcal{I}_{r, s}$, then there exists $\lambda \in \mathbb{R}$ such that $(\lambda, u)$ is a solution of the equation
\begin{equation}\label{eq8.2}
 \left\{\aligned
& -\Delta u+V u+\lambda u=s|u|^{2^*-2}u+\beta f(u), &x\in \Omega_r, \\
& \int_{\Omega_r} |u|^2dx=\Theta,u\in H_0^1(\Omega_r), &x\in \Omega_r.
\endaligned
\right.
\end{equation}

\begin{lemma}\label{L8.1}
For any $\Theta>0$, there exist $r_\Theta>0$ and $u^0, u^1 \in S_{r_\Theta, \Theta}$ such that

(i) $\mathcal{I}_{r, s}(u^1) \leq 0$ for any $r>r_\Theta$ and $s \in\left[\frac{1}{2}, 1\right]$,
$$
\left\|\nabla u^0\right\|_2^2<\left(1-\left\|V_{-}\right\|_{\frac{N}{2}} S^{-1}\right)^{\frac{N-2 }{2}}S^{\frac{N}{2}}<\left\|\nabla u^1\right\|_2^2
$$
and
$$
\mathcal{I}_{r, s}\left(u^0\right)<\frac{1}{N}\left(1-\left\|V_{-}\right\|_{\frac{N}{2}} S^{-1}\right)^{\frac{N }{2}}S^{\frac{N}{2}}.
$$

(ii) If $u \in S_{r, \Theta}$ satisfies
$$
\|\nabla u\|_2^2=\left(1-\left\|V_{-}\right\|_{\frac{N}{2}} S^{-1}\right)^{\frac{N-2 }{2}}S^{\frac{N}{2}},
$$
then there holds
$$
\mathcal{I}_{r, s}(u) \geq \frac{1}{N}\left(1-\left\|V_{-}\right\|_{\frac{N}{2}} S^{-1}\right)^{\frac{N }{2}}S^{\frac{N}{2}} .
$$

(iii) Set
$$
\widetilde{m}_{r, s}(\Theta)=\inf _{\gamma \in \widetilde{\Gamma}_{r, \Theta}} \sup _{t \in[0,1]} \mathcal{I}_{r, s}(\gamma(t))
$$
with
$$
\widetilde{\Gamma}_{r, \Theta}=\left\{\gamma \in C\left([0,1], S_{r, \Theta}\right): \gamma(0)=u^0, \gamma(1)=u^1\right\} .
$$
Then
$$
\frac{1}{N}\left(1-\left\|V_{-}\right\|_{\frac{N}{2}} S^{-1}\right)^{\frac{N }{2}}S^{\frac{N}{2}} \leq \widetilde{m}_{r, s}(\Theta) \leq \mathbf{h}(\mathbf{h}_\Theta),
$$
where $\mathbf{h}(\mathbf{h}_\Theta)=\max\limits_{t \in \mathbb{R}^{+}} h(t)$, the function $\mathbf{h}: \mathbb{R}^{+} \rightarrow \mathbb{R}$ being defined by
$$
\mathbf{h}(t)=\frac{1}{2}\left(1+\|V\|_{\frac{N}{2}} S^{-1}\right) t^2 \theta\Theta-\alpha\beta C_{N, p_1} \Theta^{\frac{p_1}{2}}\theta^{\frac{N(p_1-2)}{4}}t^{\frac{N(p_1-2)}{2}}-\frac{1}{2\cdot2^*} t^{2^*} \Theta^{\frac{2^*}{2}}\cdot|\Omega|^{\frac{2-2^*}{2}}.
$$
Here $\theta$ is the principal eigenvalue of $-\Delta$ with Dirichlet boundary conditions in $\Omega$, and $|\Omega|$ is the volume of $\Omega$.
\end{lemma}
\begin{proof}
(i) By the H\"older inequality,
\begin{equation}\label{eq8.4}
  \int_{\Omega}|v_1(x)|^{2^*} dx\geq\Theta^{\frac{2^*}{2}}\cdot|\Omega|^{\frac{2-2^*}{2}}.
\end{equation}
For $x \in \Omega_{\frac{1}{t}}$ and $t>0$, define $v_t(x):=t^{\frac{N}{2}} v_1(t x)$. Using \eqref{eq3.3}, \eqref{eq8.4}, \eqref{eq3.5} and $\frac{1}{2}\leq s\leq1$, it holds
\begin{eqnarray}\label{eq8.5}
\mathcal{I}_{\frac{1}{t}, s}\left(v_t\right)&\leq & \frac{1}{2}\left(1+\|V\|_{\frac{N}{2}} S^{-1}\right) t^2 \theta\Theta-\alpha\beta C_{N, p_1} \Theta^{\frac{p_1}{2}}\theta^{\frac{N(p_1-2)}{4}}t^{\frac{N(p_1-2)}{2}}\nonumber\\
&&-\frac{1}{2\cdot2^*} t^{\frac{N(2^*-2)}{2}} \Theta^{\frac{2^*}{2}}\cdot|\Omega|^{\frac{2-2^*}{2}} \nonumber\\
&=: & \mathbf{h}(t) .
\end{eqnarray}
Note that since $2<p_1<2+\frac{4}{N}<q=2^*$ and $\beta\leq0$ there exist $0<\mathbf{h}_\Theta<t_0$ such that $\mathbf{h}\left(t_0\right)=0, \mathbf{h}(t)<0$ for any $t>t_0, \mathbf{h}(t)>0$ for any $0<t<t_0$ and $\mathbf{h}\left(\mathbf{h}_\Theta\right)=\max\limits_{t \in \mathbb{R}^{+}} \mathbf{h}(t)$. As a consequence, there holds
\begin{equation}\label{eq8.6}
  \mathcal{I}_{r, s}\left(v_{t_0}\right)=\mathcal{I}_{\frac{1}{t_0}, s}\left(v_{t_0}\right) \leq \mathbf{h}\left(t_0\right)=0
\end{equation}
for any $r \geq \frac{1}{t_0}$ and $s \in\left[\frac{1}{2}, 1\right]$. Moreover, there exists $0<t_1<\mathbf{h}_\Theta$ such that
\begin{equation}\label{eq8.7}
 \mathbf{h}(t)<\frac{1}{N}\left(1-\left\|V_{-}\right\|_{\frac{N}{2}} S^{-1}\right)^{\frac{N }{2}}S^{\frac{N}{2}}
\end{equation}
for $t \in\left[0, t_1\right]$. On the other hand, it follows from the Sobolev inequality and the H\"older inequality that
\begin{eqnarray}\label{eq8.8}
\mathcal{I}_{r, s}(u)
&\geq& \frac{1}{2}\left(1-\left\|V_{-}\right\|_{\frac{N}{2}} S^{-1}\right) \int_{\Omega_r}|\nabla u|^2 d x-\frac{1}{2^*\cdot S^{
\frac{2^*}{2}}}\left(\int_{\Omega_r}|\nabla u|^2 d x\right)^{\frac{2^*}{2}}.
\end{eqnarray}
Define
$$
\mathbf{g}(t):=\frac{1}{2}\left(1-\left\|V_{-}\right\|_{\frac{N}{2}} S^{-1}\right) t-\frac{1}{2^*\cdot S^{
\frac{2^*}{2}}}t^{\frac{2^*}{2}}
$$
and
$$
\widetilde{t}=\left(1-\left\|V_{-}\right\|_{\frac{N}{2}} S^{-1}\right)^{\frac{N-2 }{2}}S^{\frac{N}{2}},
$$
it is easy to see that $\mathbf{g}$ is increasing on $(0, \widetilde{t})$ and decreasing on $(\widetilde{t}, \infty)$, and
\begin{eqnarray*}
\mathbf{g}(\widetilde{t})
&=&\frac{1}{N}\left(1-\left\|V_{-}\right\|_{\frac{N}{2}} S^{-1}\right)^{\frac{N }{2}}S^{\frac{N}{2}}.
\end{eqnarray*}
For $r \geq \widetilde{r}_\Theta:=\max \left\{\frac{1}{t_1}, \sqrt{\frac{2 \theta \Theta}{\tilde{t}}}\right\}$ we have $v_{\frac{1}{\tilde{r}_\Theta}} \in S_{r, \Theta}$ and
\begin{eqnarray}\label{eq8.9}
\|\nabla v_{\frac{1}{\widetilde{r}_\Theta}}\|_2^2&=&\left(\frac{1}{\widetilde{r}_\Theta}\right)^2\left\|\nabla v_1\right\|_2^2
<\left(1-\left\|V_{-}\right\|_{\frac{N}{2}} S^{-1}\right)^{\frac{N-2 }{2}}S^{\frac{N}{2}}.
\end{eqnarray}
Moreover, there holds
\begin{equation}\label{eq8.10}
 \mathcal{I}_{\widetilde{r}_\Theta, s}\left(v_{\frac{1}{\widetilde{r}_\Theta}}\right) \leq \mathbf{h}\left(\frac{1}{\widetilde{r}_\Theta}\right) \leq \mathbf{h}\left(t_1\right) .
\end{equation}
Setting $u^0=v_{\frac{1}{\tilde{r}_\Theta}}, u^1=v_{t_0}$ and
\begin{equation}\label{eq8.11}
  r_\Theta=\max \left\{\frac{1}{t_0}, \widetilde{r}_\Theta\right\}.
\end{equation}
Since \eqref{eq8.6}, \eqref{eq8.7}, \eqref{eq8.8} and \eqref{eq8.9}, then (i) holds.

(ii) By \eqref{eq8.8} and a direct calculation, (ii) holds.

(iii) Since $\mathcal{I}_{r, s}\left(u^1\right) \leq 0$ for any $\gamma \in \Gamma_{r, \Theta}$, we have
$$
\|\nabla \gamma(0)\|_2^2<\widetilde{t}<\|\nabla \gamma(1)\|_2^2.
$$
It then follows from \eqref{eq8.8} that
$$
\begin{aligned}
\max\limits_{t \in[0,1]} \mathcal{I}_{r, s}(\gamma(t))  \geq\mathbf{ g}(\widetilde{t})
 =\frac{1}{N}\left(1-\left\|V_{-}\right\|_{\frac{N}{2}} S^{-1}\right)^{\frac{N }{2}}S^{\frac{N}{2}}
\end{aligned}
$$
for any $\gamma \in \widetilde{\Gamma}_{r, \Theta}$, hence the first inequality in (iii) holds. Now we define a path $\gamma \in \widetilde{\Gamma}_{r, \Theta}$ by
$$
\gamma(\tau)(x)=\left(\tau t_0+(1-\tau) \frac{1}{\widetilde{r}_\Theta}\right)^{\frac{N}{2}} v_1\left(\left(\tau t_0+(1-\tau) \frac{1}{\widetilde{r}_\Theta}\right) x\right)
$$
for $\tau \in[0,1]$ and $x \in \Omega_r$. Then by \eqref{eq8.5} we have $\widetilde{m}_{r, s}(\Theta) \leq \mathbf{h}(\mathbf{h}_\Theta)$, where $\mathbf{h}(\mathbf{h}_\Theta)=\max\limits_{t \in \mathbb{R}^{+}} \mathbf{h}(t)$. Note that $\mathbf{h}_\Theta$ is independent of $r$ and $s$.
\end{proof}

Using Proposition \ref{p3.1} to $\mathcal{I}_{r, s}$,  it follows that
$$
A(u)=\frac{1}{2} \int_{\Omega_r}|\nabla u|^2 d x+\frac{1}{2} \int_{\Omega_r} V(x) u^2 d x-\beta \int_{\Omega_r}F(u) d x \quad \text { and } \quad B(u)=\frac{1}{2^*} \int_{\Omega_r}|u|^{2^*} d x.
$$
Hence, for almost every $s \in\left[\frac{1}{2}, 1\right]$, there exists a bounded Palais-Smale sequence $\left\{u_n\right\}$ satisfying
$$
\mathcal{I}_{r, s}\left(u_n\right) \rightarrow \widetilde{m}_{r, s}(\Theta) \quad \text { and }\left.\quad \mathcal{I}_{r, s}^{\prime}\left(u_n\right)\right|_{T_{u_n} S_{r, \Theta}} \rightarrow 0.
$$
Next, we are devoted to proving compactness.
\begin{lemma}\label{L8.2}
If $\beta \leq 0$ and the assumptions of Theorem \ref{t1.8} hold, then $\widetilde{m}_{r, s}(\Theta)<\frac{\zeta}{N}S^{\frac{N}{2}}$, where $\zeta=s^{-\frac{2}{2^*-2}}$.
\end{lemma}
\begin{proof}
Let $U_{\varepsilon}$ be defined by
$
U_{\varepsilon}(x):=\left(\frac{\varepsilon}{\varepsilon^2+|x|^2}\right)^{\frac{N-2}{2}}
$
(up to a scalar factor, $U_{\varepsilon}$ is the bubble centered in the origin, with concentration parameter $\varepsilon>0$, defined in \eqref{eq1.3}). Let also $\varphi \in C_c^{\infty}(\Omega_r)$ be a radial cut-off function with $\varphi \equiv 1$ in $B_1, \varphi \equiv 0$ in $B_2^c$, and $\varphi$ radially decreasing. We define
$$
u_{\varepsilon}(x):=\varphi(x) U_{\varepsilon}(x), \quad \text { and } \quad v_{\varepsilon}(x):=\sqrt{\Theta} \frac{u_{\varepsilon}(x)}{\|u_{\varepsilon}\|_2}.
$$
Notice that $u_{\varepsilon} \in C_c^{\infty}(\Omega_r)$, and $v_{\varepsilon} \in S_{r,\Theta}$. Let us recall the following useful estimates (see \cite[Lemma A.1]{{NSJFA2020}}):
\begin{equation}\label{eq8.12}
  \|\nabla u_{\varepsilon}\|_2^2=K_1+O\left(\varepsilon^{N-2}\right),
\end{equation}
\begin{eqnarray}\label{eq8.13}
  \|u_{\varepsilon}\|_{2^*}^2= \begin{cases}K_2+O\left(\varepsilon^N\right), & \text { if } N \geq 4, \\
K_2+O\left(\varepsilon^2\right), & \text { if } N=3,\end{cases}
\end{eqnarray}
\begin{eqnarray}\label{eq8.14}
& \|u_{\varepsilon}\|_2^2= \begin{cases}\varepsilon^2 K_3+O\left(\varepsilon^{N-2}\right), & \text { if } N \geq 5,  \\
\omega \varepsilon^2|\log \varepsilon|+O\left(\varepsilon^2\right), & \text { if } N=4, \\
\omega\left(\int_0^2 \varphi(r) d r\right) \varepsilon+O\left(\varepsilon^2\right), & \text { if } N=3,\end{cases}
\end{eqnarray}
\begin{eqnarray}\label{eq8.15}
& \|u_{\varepsilon}\|_q^q=\varepsilon^{N-\frac{N-2}{2} q}\left(K_4+O\left(\varepsilon^{(N-2) q-N}\right)\right) \nonumber\\
&\text { if } N \geq 4 \text { and } q \in\left(2,2^*\right),
 \text { and if } N=3 \text { and } q \in(3,6).
\end{eqnarray}
as  $\varepsilon \rightarrow 0$. Since $U_{\varepsilon}$ is extremal for the Sobolev inequality, we have that $\frac{K_1 }{ K_2}=S$. Therefore, using $\frac{1}{2} \leq s \leq 1$, we have
\begin{eqnarray*}
\mathcal{I}_{r, s}(tv_\varepsilon)
&\leq&\frac{t^2}{2} \int_{\Omega_r}|\nabla v_\varepsilon|^2 d x+\frac{t^2}{2} \int_{\Omega_r} V  v_\varepsilon^2 d x-\frac{st^{2^*}}{ 2^*} \int_{\Omega_r}|v_\varepsilon|^{2^*} d x- \alpha\beta t^{p_1}  \int_{\Omega_r}|v_\varepsilon|^{p_1} d x\\
&=:&h_3(t).
\end{eqnarray*}
Clearly, $h_3(t)>0$ for $t>0$ small and $h_3(t) \rightarrow-\infty$ as $t \rightarrow \infty$, so $h_3(t)$ attains its maximum at some $t_\varepsilon>0$ with $h_3^{\prime}(t_\varepsilon)=0$. Then, observing that the function
$$
t \mapsto \frac{t^2}{2} \int_{\Omega_r}|\nabla v_\varepsilon|^2 d x+\frac{t^2}{2} \int_{\Omega_r} V  v_\varepsilon^2 d x-\frac{st^{2^*}}{ 2^*} \int_{\Omega_r}|v_\varepsilon|^{2^*} d x- \alpha\beta t^{p_1}  \int_{\Omega_r}|v_\varepsilon|^{p_1} d x
$$
is increasing on the interval of
$$
\left[0,\left[\frac{-2^*\alpha\beta (p_1-2)\|v_\varepsilon\|_{p_1}^{p_1}}{s(2^*-2)\|v_\varepsilon\|_{2^*}^{2^*}}\right]^{\frac{1}{2^*-p_1}}\right].
$$
This fact combined with \eqref{eq8.12}-\eqref{eq8.15} implies that there exist $\delta_1, \delta_2>0$, independent of $\varepsilon>0$, such that
$$
\delta_1 \leq t_\varepsilon \leq \delta_2.
$$
Moreover, observing that the function
$$
t \mapsto \frac{t^2}{2}\int_{\Omega_r}|\nabla v_\varepsilon|^2 d x-\frac{st^{2^*}}{2^*} \int_{\Omega_r}|v_\varepsilon|^{2^*} d x
$$
is increasing on the interval of
$$
\left[0,\left(\frac{\|\nabla v_\varepsilon\|_{2}^2}{s\|v_\varepsilon\|_{2^*}^{2^*}}\right)^{\frac{1}{2^*-2}}\right].
$$
Using \eqref{eq8.12}-\eqref{eq8.15} and the fact that $\frac{K_1 }{ K_2}=S$, if $N=3$, the same estimate holds eventually, using $\|u_{\varepsilon}\|_{2^*}^2=K_2+O(\varepsilon^2)$ instead of $\|u_{\varepsilon}\|_{2^*}^2=K_2+O(\varepsilon^N)$. Therefore, the maximum level is
\begin{eqnarray*}
h_3(t_\varepsilon)&\leq&\frac{t_\varepsilon^2}{2}\|\nabla v_\varepsilon\|_{2}^2+\frac{t_\varepsilon^2}{2} \int_{\Omega_r} V  v_\varepsilon^2 d x-\frac{st_\varepsilon^{2^*}}{ 2^*} \int_{\Omega_r}|v_\varepsilon|^{2^*} d x- \alpha\beta t_\varepsilon^{p_1}  \int_{\Omega_r}|v_\varepsilon|^{p_1} d x\\
&\leq&\frac{1}{Ns^{\frac{2}{2^*-2}}} \left(\frac{\|\nabla v_\varepsilon\|_{2}^2}{\|v_\varepsilon\|_{2^*}^{2}}\right)^{\frac{2^*}{2^*-2}}+\frac{\int_{\Omega_r} V  v_\varepsilon^2 d x}{2} \left(\frac{\|\nabla v_\varepsilon\|_{2}^2}{s\|v_\varepsilon\|_{2^*}^{2^*}}\right)^{\frac{2}{2^*-2}}- \alpha\beta  \|v_\varepsilon\|_{p_1}^{p_1}\left(\frac{\|\nabla v_\varepsilon\|_{2}^2}{s\|v_\varepsilon\|_{2^*}^{2^*}}\right)^{\frac{p_1}{2^*-2}}\\
&=&\frac{1}{Ns^{\frac{2}{2^*-2}}} \left(\frac{\|\nabla u_\varepsilon\|_{2}^2}{\|u_\varepsilon\|_{2^*}^{2}}\right)^{\frac{N}{2}}+\frac{\int_{\Omega_r} V  v_\varepsilon^2 d x}{2s^{\frac{2}{2^*-2}}} \left(\Theta^{\frac{2-2^*}{2}}\cdot\frac{\|\nabla u_\varepsilon\|_{2}^2}{\|u_\varepsilon\|_{2^*}^{2^*}}\cdot\frac{\|u_\varepsilon\|_{2}^{2^*}}{\|u_\varepsilon\|_{2}^{2}}\right)^{\frac{2}{2^*-2}}\\
&&-\frac{ \alpha\beta  }{s^{\frac{p_1}{2^*-2}}}\left(\Theta^{\frac{2-2^*}{2}}\cdot\frac{\|\nabla u_\varepsilon\|_{2}^2}{\|u_\varepsilon\|_{2^*}^{2^*}}\cdot\frac{\|u_\varepsilon\|_{2}^{2^*}}{\|u_\varepsilon\|_{2}^{2}}\right)^{\frac{p_1}{2^*-2}}\cdot\Theta^{\frac{p_1}{2}}\frac{\|u_\varepsilon\|_{p_1}^{p_1}}{\|u_\varepsilon\|_2^{p_1}}\\
&\leq&\frac{1}{Ns^{\frac{2}{2^*-2}}}\left[\frac{K_1+O\left(\varepsilon^{N-2}\right)}{K_2+O\left(\varepsilon^N\right)}\right]^{\frac{N}{2}} +\frac{\max\limits_{x\in\Omega_r}V(x)}{2s^{\frac{2}{2^*-2}}}\cdot\|u_\varepsilon\|_{2}^{2}\cdot \frac{\|\nabla u_\varepsilon\|_{2}^{\frac{4}{2^*-2}}}{\|u_\varepsilon\|_{2^*}^{\frac{2\cdot2^*}{2^*-2}}}\\
&&-\frac{ \alpha\beta  }{s^{\frac{p_1}{2^*-2}}}\cdot\|u_\varepsilon\|_{p_1}^{p_1}\cdot\frac{\|\nabla u_\varepsilon\|_{2}^{\frac{2p_1}{2^*-2}}}{\|u_\varepsilon\|_{2^*}^{\frac{p_1\cdot2^*}{2^*-2}}}\\
&=&\frac{1}{Ns^{\frac{2}{2^*-2}}}S^{\frac{N}{2}}+O\left(\varepsilon^{N-2}\right)+C_1\|u_\varepsilon\|_{2}^{2}+C_2\|u_\varepsilon\|_{p_1}^{p_1}\\
&=&\frac{1}{Ns^{\frac{2}{2^*-2}}}S^{\frac{N}{2}}
\end{eqnarray*}
as $\varepsilon\rightarrow0$, where $\zeta=s^{-\frac{2}{2^*-2}}$ and $C_1\geq 0,\ C_2\geq 0$ because of $\beta\leq 0$. In the penultimate equal sign, we used
\begin{equation*}
  \frac{1}{N}\left[\frac{K_1+O\left(\varepsilon^{N-2}\right)}{K_2+O\left(\varepsilon^N\right)}\right]^{\frac{N}{2}} =\frac{1}{N}\left[\frac{K_1}{K_2}+O\left(\varepsilon^{N-2}\right)\right]^{\frac{N}{2}}=\frac{S^{\frac{N}{2}}}{N}+O\left(\varepsilon^{N-2}\right).
\end{equation*}
This completes the proof.
\end{proof}
\begin{lemma}\label{L8.3}
For any $\Theta>0$, let $r>r_\Theta$, where $r_\Theta$ is defined in Lemma \ref{L8.1}. Then problem \eqref{eq8.2} has a solution $\left(\lambda_{r, s}, u_{r, s}\right)$ for almost every $s \in\left[\frac{1}{2}, 1\right]$. Moreover, $u_{r, s} \geq 0$ and $\mathcal{I}_{r, s}\left(u_{r, s}\right)=\widetilde{m}_{r, s}(\Theta)$.
\end{lemma}
\begin{proof}
Based on the previous analysis, we know that, for almost every $s \in\left[\frac{1}{2}, 1\right]$, there exists a bounded Palais-Smale sequence $\left\{u_n\right\}$ satisfying
\begin{equation}\label{eq8.16}
\mathcal{I}_{r, s}\left(u_n\right) \rightarrow \widetilde{m}_{r, s}(\Theta) \quad \text { and }\left.\quad \mathcal{I}_{r, s}^{\prime}\left(u_n\right)\right|_{T_{u_n} S_{r, \Theta}} \rightarrow 0.
\end{equation}
Then
$$
\lambda_n=-\frac{1}{\Theta}\left(\int_{\Omega_r}\left|\nabla u_n\right|^2 d x+\int_{\Omega_r} V(x) u_n^2 d x-\beta \int_{\Omega_r}f(u_n)u_n d x-s \int_{\Omega_r}\left|u_n\right|^{2^*} d x\right)
$$
is bounded and
\begin{equation}\label{eq8.17}
  \mathcal{I}_{r, s}^{\prime}\left(u_n\right)+\lambda_n u_n \rightarrow 0 \quad \text { in } H^{-1}\left(\Omega_r\right) .
\end{equation}
Moreover, since $\left\{u_n\right\}$ is a bounded Palais-Smale sequence, there exist $u_0 \in H_0^1\left(\Omega_r\right)$ and $\lambda \in \mathbb{R}$ such that, up to a subsequence,
\begin{eqnarray}\label{eq8.18}
\lambda_n &\rightarrow& \lambda \ \text { in }\ \mathbb{R}, \nonumber\\
u_n &\rightharpoonup& u_0 \  \text { in }\ H_0^1(\Omega_r),\nonumber\\
 u_n &\rightarrow& u_0 \  \text { in }\ L^t(\Omega_r) \text { for all } 2 \leq t<2^*,
\end{eqnarray}
where $u_0$ satisfies
$$
\left\{\begin{array}{l}
-\Delta u_0+V u_0+\lambda u_0=s\left|u_0\right|^{2^*-2} u_0+\beta f(u_0) \quad \text { in } \Omega_r, \\
u_0 \in H_0^1\left(\Omega_r\right), \quad \int_{\Omega_r}\left|u_0\right|^2 d x=\Theta.
\end{array}\right.
$$
Using \eqref{eq8.17}, we have
$$
\mathcal{I}_{r, s}^{\prime}\left(u_n\right) u_0+\lambda_n \int_{\Omega_r} u_n u_0 d x \rightarrow 0 \text { as } n \rightarrow \infty
$$
and
\begin{equation*}
  I_{r, s}^{\prime}\left(u_n\right) u_n+\lambda_n \Theta\rightarrow 0 \ \text{as }\  n \rightarrow \infty  .
\end{equation*}
Note that
\begin{eqnarray*}
\lim _{n \rightarrow \infty} \int_{\Omega_r} V(x) u_n^2 d x&=&\int_{\Omega_r} V(x) u_0^2 d x,\\
\lim _{n \rightarrow \infty} \int_{\Omega_r} f(u_n)u_n d x&=&\int_{\Omega_r} f(u_0)u_0 d x,\\
\lim _{n \rightarrow \infty} \int_{\Omega_r} f(u_n)u_0 d x&=&\int_{\Omega_r} f(u_0)u_0  d x.
\end{eqnarray*}
Now, we show that $u_n \rightarrow u_0$ in $H_0^1\left(\Omega_r\right)$. Firstly, note that the weak limit $u_0$ does not vanish identically. Suppose by contradiction that $u_0 \equiv 0$. Since $\left\{u_{n}\right\}$ is bounded in $H^1(\Omega_r)$, up to a subsequence we have that $\|\nabla u_{n}\|_2^2\rightarrow \ell \in \mathbb{R}$. Using $(f_2)$, \eqref{eq8.17}, \eqref{eq8.18}, we have
\begin{eqnarray*}
  \langle\mathcal{I}_{r,s}^{\prime}(u_{n}),u_{n}\rangle&=& \int_{\Omega_r}\left|\nabla u_{n}\right|^2 d x+  \int_{\Omega_r} V u_{n}^2 d x- s\int_{\Omega_r}\left|u_{n}\right|^{2^*} d x-\beta \int_{\Omega_r}f(u_{n})u_{n} d x\\
  &\rightarrow&0,
\end{eqnarray*}
hence
$$
s\|u_{n}\|_{2^*}^{2^*}=\|\nabla u_{n}\|_2^2 \rightarrow \ell
$$
as well. Therefore, $\ell \geq s^{-\frac{2}{2^*}}S \ell^{\frac{2}{2^*}}$, and we deduce that either $\ell=0$, or $\ell \geq s^{-\frac{2}{2^*-2}}S^{\frac{N}{2}}$. Let us suppose at first that $\ell \geq s^{-\frac{2}{2^*-2}}S^{\frac{N}{2}}$. Since $\mathcal{I}_{r,s}(u_{n}) \rightarrow \widetilde{m}_{r, s}(\Theta)<\frac{\zeta}{N}S^{\frac{N}{2}}$, we have that
\begin{eqnarray*}
\frac{\zeta}{N}S^{\frac{N}{2}}>\widetilde{m}_{r, s}(\Theta)&\leftarrow&\mathcal{I}_{r,s}(u_{n})+o(1) \\
& =&\frac{1}{2} \int_{\Omega_r}\left|\nabla u_{n}\right|^2 d x+\frac{1}{2} \int_{\Omega_r} V u_{n}^2 d x-\frac{s}{2^*} \int_{\Omega_r}\left|u_{n}\right|^{2^*} d x-\beta \int_{\Omega_r}F(u_{n}) d x \\
& =&\frac{\ell}{N}\geq s^{-\frac{2}{2^*-2}}\frac{S^{\frac{N}{2}}}{N},
\end{eqnarray*}
which is not possible. If instead $\ell=0$, we have $\|u_{n}\|_{2^*}\rightarrow0$, $\|\nabla u_{n}\|_2\rightarrow0$ and
$F(u_{n})\rightarrow0$. But then $I_{r,s}(u_{n}) \rightarrow 0\neq\widetilde{m}_{r, s}(\Theta)$, which gives again a contradiction. Thus, $u_r$ does not vanish identically.

Since $\{u_{n, r}\}$ is a bounded minimization sequence for $\mathcal{I}_r(u_{n, r})|_{S_{r, \Theta}} $, there exists $\{\lambda_n\} \subset \mathbb{R}$ such that for every $\varphi \in H^1(\Omega_r)$,
\begin{equation}\label{eq8.19}
  \int_{\Omega_r} \nabla u_{n, r} \cdot \nabla \varphi+\int_{\Omega_r}Vu_{n, r}\varphi+\lambda_n u_{n, r} \varphi-\beta f(u_{n, r})   \varphi-s\left|u_{n, r}\right|^{2^*-2} u_{n, r} \varphi=o(1)\|\varphi\|
\end{equation}
as $n \rightarrow \infty$, by the Lagrange multipliers rule. Choosing $\varphi=u_{n, r}$, we deduce that $\left\{\lambda_n\right\}$ is bounded as well, and hence up to a subsequence $\lambda_n \rightarrow \lambda_r \in \mathbb{R}$. Moreover, passing to the limit in \eqref{eq8.19} by weak convergence, we obtain
\begin{equation*}
  -\Delta u_r+V u_r+\lambda_r u_r=s|u_r|^{2^*-2}u_r+\beta f(u_r),\ x\in \Omega_r.
\end{equation*}
Recalling that $v_{n, r}=u_{n, r}-u_r \rightharpoonup 0$ in $H_0^1(\Omega_r)$, we know
$$
\|\nabla u_{n, r}\|_2^2=\|\nabla u_{r}\|_2^2+\|\nabla v_{n, r}\|_2^2+o(1).
$$
By the Br$\mathrm{\acute{e}}$zis-Lieb lemma \cite{HB1983}, we have
$$
\|u_{n, r}\|_{2^*}^{2^*}=\|u_{r}\|_{2^*}^{2^*}+\|v_{n, r}\|_{2^*}^{2^*}+o(1) .
$$
Note that
$$
\int_{\Omega_r} V v_{n, r}^2 d x \rightarrow 0 \ \text {as}\ n \rightarrow \infty,
$$
hence
$$
s\|v_{n, r}\|_{2^*}^{2^*}=\|\nabla v_{n, r}\|_2^2 \rightarrow \ell
$$
as well. Therefore, by the Sobolev inequality $\ell \geq s^{-\frac{2}{2^*}}S \ell^{\frac{2}{2^*}}$, and we deduce that either $\ell=0$, or $\ell \geq s^{-\frac{2}{2^*-2}}S^{\frac{N}{2}}$. Let us suppose at first that $\ell \geq s^{-\frac{2}{2^*-2}}S^{\frac{N}{2}}$. Since $\mathcal{I}_{r,s}(u_{n}) \rightarrow \widetilde{m}_{r, s}(\Theta)<\frac{\zeta}{N}S^{\frac{N}{2}}$, we have that
\begin{eqnarray*}
\frac{\zeta}{N}S^{\frac{N}{2}}>\widetilde{m}_{r, s}(\Theta)&\leftarrow&\mathcal{I}_{r,s}(v_{n})+o(1) \\
& =&\frac{1}{2} \int_{\Omega_r}\left|\nabla v_{n}\right|^2 d x+\frac{1}{2} \int_{\Omega_r} V v_{n}^2 d x-\frac{s}{2^*} \int_{\Omega_r}\left|v_{n}\right|^{2^*} d x-\beta \int_{\Omega_r}F(v_{n}) d x \\
& =&\frac{\ell}{N}\geq s^{-\frac{2}{2^*-2}}\frac{S^{\frac{N}{2}}}{N},
\end{eqnarray*}
which is not possible. If instead $\ell=0$, we have that $u_{n, r} \rightarrow u_r$ in $H_0^1\left(\Omega_r\right)$, so $\mathcal{I}_r(u_r)>0$.

Similar to the proof of Lemma \ref{L3.2}, we also obtain that $u_{r, s} \geq 0$.
\end{proof}

In order to obtain a solution of \eqref{eeq8.1}, we also need to prove a uniform estimate for the solutions of \eqref{eq8.2} established in Lemma \ref{L8.3}. Similar to the proof of Lemma \ref{L3.3} and Lemma \ref{L3.4}, we obtain the following lemmas.

\begin{lemma}\label{L8.4}
If $(\lambda, u) \in \mathbb{R} \times S_{r, \Theta}$ is a solution of \eqref{eq8.2} established in Lemma \ref{L8.3} for some $r$ and $s$, then
$$
 \int_{\Omega_r}|\nabla u|^2 d x \leq \frac{4 N}{N(2^*-2)-4}\left(\frac{2^*-2}{2} \mathbf{h}(\mathbf{h}_\Theta)+\Theta\left(\frac{1}{2 N}\|\widetilde{V}\|_{\infty}+\frac{2^*-2}{4}\|V\|_{\infty}\right)\right),
$$
where the constant $\mathbf{h}(\mathbf{h}_\Theta)$ is defined in (iii) of Lemma \ref{L8.1}  and is independent of $r$ and $s$.
\end{lemma}

\begin{lemma}\label{L8.5}
For every $\Theta>0$, problem \eqref{eq8.2} has a solution $\left(\lambda_r, u_r\right)$ provided $r>r_\Theta$ where $r_\Theta$ is as in Lemma \ref{L8.1}. Moreover, $u_r \geq 0$ in $\Omega_r$.
\end{lemma}
\noindent\textbf{Proof of Theorem \ref{t1.8}} The proof is an immediate consequence of Lemmas \ref{L8.5}  and \ref{L3.5}.

\section{Proof of Theorem \ref{t1.9}}
 In this subsection, we assume $\beta > 0$ and the assumptions of Theorem \ref{t1.8} hold. Consider the following equation
 \begin{equation}\label{eq9.1}
 \left\{\aligned
& -\Delta u+V u+\lambda u=|u|^{2^*-2}u+\beta f(u), &x\in \Omega_r, \\
& \int_{\Omega_r} |u|^2dx=\Theta, u\in H_0^1(\Omega_r), &x\in \Omega_r.
\endaligned
\right.
\end{equation}
For $\frac{1}{2} \leq s \leq 1$, we define the functional $\mathcal{J}_{r, s}: S_{r, \Theta} \rightarrow \mathbb{R}$ by
\begin{equation}\label{eq9.2}
  \mathcal{J}_{r, s}(u)=\frac{1}{2} \int_{\Omega_r}|\nabla u|^2 d x+\frac{1}{2} \int_{\Omega_r} V  u^2 d x-\frac{s}{2^*} \int_{\Omega_r}|u|^{2^*} d x-s\beta \int_{\Omega_r}F(u) d x .
\end{equation}
Note that if $u \in S_{r, \Theta}$ is a critical point of $\mathcal{J}_{r, s}$, then there exists $\lambda \in \mathbb{R}$ such that $(\lambda, u)$ is a solution of the equation
\begin{equation}\label{eq9.3}
 \left\{\aligned
& -\Delta u+V u+\lambda u=s|u|^{2^*-2}u+s\beta f(u), &x\in \Omega_r, \\
& \int_{\Omega_r} |u|^2dx=\Theta, u\in H_0^1(\Omega_r), &x\in \Omega_r.
\endaligned
\right.
\end{equation}
\begin{lemma}\label{L9.1}
For any $\Theta>0$, there exist $\widehat{r}_\Theta>0$ and $u^0, u^1 \in S_{r_\Theta, \Theta}$ such that

(i) For $r>\widehat{r}_\Theta$ and $s \in\left[\frac{1}{2}, 1\right]$ we have $\mathcal{J}_{r, s}\left(u^1\right) \leq 0$ and
$$
\mathcal{J}_{r, s}\left(u^0\right)<\widehat{A}^{-\frac{2}{2^*-2}}\left(1-\left\|V_{-}\right\|_{\frac{N}{2}} S^{-1}\right)^{\frac{2^*}{2^*-2}} \left[\frac{2^*-2}{2}\left(\frac{1}{2^*}\right)^{\frac{2^*}{2^*-2}}\right],
$$
where
$$
\widehat{A}=S^{-
\frac{2^*}{2}}\left[\frac{4(2^*-2)}{N(p_1-2)(4-N(p_1-2))} +\frac{1}{2^*}\right].
$$
Moreover,
$$
\left\|\nabla u^0\right\|_2^2<\left[\frac{\left(1-\left\|V_{-}\right\|_{\frac{N}{2}} S^{-1}\right)}{2^* \widehat{A}}\right]^{\frac{2}{2^*-2}},\ \left\|\nabla u^1\right\|_2^2>\left[\frac{\left(1-\left\|V_{-}\right\|_{\frac{N}{2}} S^{-1}\right)}{2^* \widehat{A}}\right]^{\frac{2}{2^*-2}}.
$$

(ii) If $u \in S_{r, \Theta}$ satisfies
$$
\|\nabla u\|_2^2=\left[\frac{\left(1-\left\|V_{-}\right\|_{\frac{N}{2}} S^{-1}\right)}{2^* \widehat{A}}\right]^{\frac{2}{2^*-2}},
$$
then there holds
$$
\mathcal{J}_{r, s}(u) \geq \widehat{A}^{-\frac{2}{2^*-2}}\left(1-\left\|V_{-}\right\|_{\frac{N}{2}} S^{-1}\right)^{\frac{2^*}{2^*-2}} \left[\frac{2^*-2}{2}\left(\frac{1}{2^*}\right)^{\frac{2^*}{2^*-2}}\right].
$$

(iii) Let
$$
\widehat{m}_{r, s}(\Theta)=\inf _{\gamma \in \widehat{\Gamma}_{r, \Theta}} \sup\limits_{t \in[0,1]} \mathcal{J}_{r, s}(\gamma(t)),
$$
where
$$
\widehat{\Gamma}_{r, \Theta}=\left\{\gamma \in C\left([0,1], S_{r, \Theta}\right): \gamma(0)=u^0, \gamma(1)=u^1\right\}.
$$
Then
$$
\widehat{m}_{r, s}(\Theta) \geq \widehat{A}^{-\frac{2}{2^*-2}}\left(1-\left\|V_{-}\right\|_{\frac{N}{2}} S^{-1}\right)^{\frac{2^*}{2^*-2}} \left[\frac{2^*-2}{2}\left(\frac{1}{2^*}\right)^{\frac{2^*}{2^*-2}}\right]
$$
and
\begin{eqnarray*}
  \widehat{m}_{r, s}(\Theta) &\leq& \frac{N(2^*-2)-4}{2}\left(\frac{\theta\left(1+\|V\|_{\frac{N}{2}} S^{-1}\right)}{N(2^*-2)}\right)^{\frac{N(2^*-2)}{N(2^*-2)-4}}(4\cdot 2^*)^{\frac{4}{N(2^*-2)-4}}|\Omega|^{\frac{2(2^*-2)}{N(2^*-2)-4}}\\
   &&\cdot\Theta^{\frac{N(2^*-2)-2\cdot2^*}{N(2^*-2)-4}},
\end{eqnarray*}
where $\theta$ is the principal eigenvalue of $-\Delta$ with Dirichlet boundary condition in $\Omega$.
\end{lemma}
\begin{proof}
(i) By the H\"older inequality, we know
\begin{equation}\label{eq9.5}
  \int_{\Omega}|v_1(x)|^{2^*} dx\geq\Theta^{\frac{2^*}{2}}\cdot|\Omega|^{\frac{2-2^*}{2}}.
\end{equation}
For $x \in \Omega_{\frac{1}{t}}$ and $t>0$, define $v_t(x):=t^{\frac{N}{2}} v_1(t x)$. Using \eqref{eq5.2}, \eqref{eq9.5}, \eqref{eq4.4} and $\frac{1}{2}\leq s\leq1$, it holds
\begin{eqnarray}\label{eq9.6}
\mathcal{J}_{\frac{1}{t}, s}\left(v_t\right)
&\leq & \frac{1}{2}\left(1+\|V\|_{\frac{N}{2}} S^{-1}\right) t^2 \theta\Theta- \frac{\beta}{2} \alpha_1 t^{\frac{N(p_2-2)}{2}}  \int_{\Omega}|v_1|^{p_2} d x-\frac{1}{2\cdot2^*} t^{\frac{N(2^*-2)}{2}} \Theta^{\frac{2^*}{2}}\cdot|\Omega|^{\frac{2-2^*}{2}} \nonumber\\
&\leq: & \widehat{h}_2(t),
\end{eqnarray}
where
\begin{equation*}
  \widehat{h}_2(t)=\frac{1}{2}\left(1+\|V\|_{\frac{N}{2}} S^{-1}\right) t^2 \theta\Theta-\frac{1}{2\cdot2^*} t^{\frac{N(2^*-2)}{2}} \Theta^{\frac{2^*}{2}}\cdot|\Omega|^{\frac{2-2^*}{2}}.
\end{equation*}
A simple computation shows that $\widehat{h}_2(t_0)=0$ for
$$
t_0:=\left[\left(1+\|V\|_{\frac{N}{2}} S^{-1}\right) 2^* \theta \Theta^{\frac{2-2^*}{2}}|\Omega|^{\frac{2^*-2}{2}}\right]^{\frac{2}{N(2^*-2)-4}}
$$
and $\widehat{h}_2(t)<0$ for any $t>t_0, \widehat{h}_2(t)>0$ for any $0<t<t_0$. Moreover, $\widehat{h}_2(t)$ achieves its maximum at
$$
t_\Theta=\left[\frac{4 \cdot2^*\left(1+\|V\|_{\frac{N}{2}} S^{-1}\right) \theta}{N(2^*-2)} \Theta^{\frac{2-2^*}{2}}|\Omega|^{\frac{2^*-2}{2}}\right]^{\frac{2}{N(2^*-2)-4}} .
$$
This implies
\begin{equation}\label{eq9.7}
  \mathcal{J}_{r, s}(v_{t_0})=J_{\frac{1}{t_0}, s}(v_{t_0}) \leq \widehat{h}_2(t_0)=0
\end{equation}
for any $r \geq \frac{1}{t_0}$ and $s \in\left[\frac{1}{2}, 1\right]$. There exists $0<t_1<t_\Theta$ such that for any $t \in\left[0, t_1\right]$,
\begin{equation}\label{eq9.8}
\widehat{h}_2(t)<A^{-\frac{2}{2^*-2}}\left(1-\left\|V_{-}\right\|_{\frac{N}{2}} S^{-1}\right)^{\frac{2^*}{2^*-2}} \left[\frac{2^*-2}{2}\left(\frac{1}{2^*}\right)^{\frac{2^*}{2^*-2}}\right] .
\end{equation}
On the other hand, it follows from \eqref{eq3.5} that
\begin{eqnarray*}
\mathcal{J}_{r, s}(u)
&\geq&\frac{1}{2}\left(1-\left\|V_{-}\right\|_{\frac{N}{2}} S^{-1}\right) \int_{\Omega_r}|\nabla u|^2 d x-\frac{1}{2^*\cdot S^{
\frac{2^*}{2}}}\left(\int_{\Omega_r}|\nabla u|^2 d x\right)^{\frac{2^*}{2}}\nonumber\\
&&-\alpha\beta C_{N, p_1} \Theta^{\frac{2 p_1-N(p_1-2)}{4}}\left(\int_{\Omega_r}|\nabla u|^2 d x\right)^{\frac{N(p_1-2)}{4}}.
\end{eqnarray*}
Define
\begin{eqnarray*}
\widehat{g}_1(t)&:=&\frac{1}{2}\left(1-\left\|V_{-}\right\|_{\frac{N}{2}} S^{-1}\right) t-\frac{1}{2^*\cdot S^{
\frac{2^*}{2}}}t^{\frac{2^*}{2}}-\alpha\beta C_{N, p_1} \Theta^{\frac{2 p_1-N(p_1-2)}{4}}t^{\frac{N(p_1-2)}{4}}\\
&=&t^{\frac{N(p_1-2)}{4}}\left[\frac{1}{2}\left(1-\left\|V_{-}\right\|_{\frac{N}{2}} S^{-1}\right) t^{\frac{4-N(p_1-2)}{4}}-\frac{1}{2^*\cdot S^{
\frac{2^*}{2}}}t^{\frac{2\cdot2^*-N(p_1-2)}{4}}\right]\\
&&-\alpha\beta C_{N, p_1} \Theta^{\frac{2 p_1-N(p_1-2)}{4}}t^{\frac{N(p_1-2)}{4}}.
\end{eqnarray*}
In view of $2<p<2+\frac{2}{N}<q<2^*$ and the definition of $\widetilde{\Theta}_V$, there exist $0<l_1<l_M<l_2$ such that $ \widehat{g}_1(t)<0$ for any $0<t<l_1$ and $t>l_2, \widehat{g}_1(t)>0$ for $l_1<t<l_2$ and $\widehat{g}_1\left(l_M\right)=\max\limits_{t \in \mathbb{R}^{+}} \widehat{g}_1(t)>0$. Let
$$
t_2=\left(\frac{\alpha\beta C_{N, p_1}S^{
\frac{2^*}{2}}N(p_1-2)(4-N(p_1-2))}{4(2^*-2)}\right)^{\frac{4}{2\cdot2^*-N(p_1-2)}} \Theta^{\frac{2p_1-N(p_1-2)}{2\cdot2^*-N(p_1-2)}} .
$$
Then by a direct calculation, we have $g_1^{\prime \prime}(t) \leq 0$ if and only if $t \geq t_2$. Hence
$$
\max\limits_{t \in \mathbb{R}^{+}} \widehat{g}_1(t)=\max _{t \in\left[t_2, \infty\right)} \widehat{g}_1(t).
$$
Note that for any $t \geq t_2$,
\begin{eqnarray}\label{eq9.10}
g_1(t) & =&\frac{1}{2}\left(1-\left\|V_{-}\right\|_{\frac{N}{2}} S^{-1}\right) t-\frac{1}{2^*\cdot S^{
\frac{2^*}{2}}}t^{\frac{2^*}{2}}-\alpha\beta C_{N, p_1} \Theta^{\frac{2 p_1-N(p_1-2)}{4}}t^{\frac{N(p_1-2)}{4}} \nonumber\\
& =&\frac{1}{2}\left(1-\left\|V_{-}\right\|_{\frac{N}{2}} S^{-1}\right) t-\alpha\beta C_{N, p_1} \Theta^{\frac{2 p_1-N(p_1-2)}{4}}\cdot t^{\frac{N(p_1-2)}{4}}-\frac{1}{2^*\cdot S^{\frac{2^*}{2}}}t^{\frac{2^*}{2}} \nonumber\\
& =&\frac{1}{2}\left(1-\left\|V_{-}\right\|_{\frac{N}{2}} S^{-1}\right) t-\frac{4(2^*-2)}{S^{
\frac{2^*}{2}}N(p_1-2)(4-N(p_1-2))} \cdot t_2^{\frac{2\cdot2^*-N(p_1-2)}{4}}\cdot t^{\frac{N(p_1-2)}{4}}\nonumber\\
&&-\frac{1}{2^*\cdot S^{\frac{2^*}{2}}}t^{\frac{2^*}{2}} \nonumber\\
&\geq&\frac{1}{2}\left(1-\left\|V_{-}\right\|_{\frac{N}{2}} S^{-1}\right) t-S^{-
\frac{2^*}{2}}\left[\frac{4(2^*-2)}{N(p_1-2)(4-N(p_1-2))} +\frac{1}{2^*}\right]t^{\frac{2^*}{2}} \nonumber\\
& =:& \widehat{g}_2(t) .
\end{eqnarray}

Now, we will determine the value of $\widetilde{\Theta}_V$. In fact, $\widehat{g}_1\left(l_M\right)=\max\limits_{t \in \mathbb{R}^{+}} \widehat{g}_1(t)>0$ as long as $\widehat{g}_2(t_2)>0$, that is,
\begin{eqnarray*}
\widehat{g}_2(t_2)&=&\frac{1}{2}\left(1-\left\|V_{-}\right\|_{\frac{N}{2}} S^{-1}\right) t_2-S^{-
\frac{2^*}{2}}\left[\frac{4(2^*-2)}{N(p_1-2)(4-N(p_1-2))} +\frac{1}{2^*}\right]t_2^{\frac{2^*}{2}}\\
&=&\frac{1}{2}\left(1-\left\|V_{-}\right\|_{\frac{N}{2}} S^{-1}\right)\left(\frac{\alpha\beta C_{N, p_1}S^{
\frac{2^*}{2}}}{A_{p_1}}\right)^{\frac{4}{2\cdot2^*-N(p_1-2)}} \Theta^{\frac{2p_1-N(p_1-2)}{2\cdot2^*-N(p_1-2)}}\\
&&-S^{-
\frac{2^*}{2}}\left(A_{p_1}+\frac{1}{2^*}\right)\left(\frac{\alpha\beta C_{N, p_1}S^{
\frac{2^*}{2}}}{A_{p_1}}\right)^{\frac{2\cdot2^*}{2\cdot2^*-N(p_1-2)}} \Theta^{\frac{2^*[2p_1-N(p_1-2)]}{2[2\cdot2^*-N(p_1-2)]}}\\
&>&0,
\end{eqnarray*}
where
$$
A_{p_1 }=\frac{4(2^*-2)}{N(p_1-2)(4-N(p_1-2))}.
$$
Hence, we take
\begin{equation*}
 \widetilde{\Theta}_V=  \left(\frac{\alpha\beta C_{N, p_1}S^{\frac{2^*}{2}}}{A_{p_1}}\right)^{-\frac{4}{2p_1-N(p_1-2)}}\left[\frac{S^{\frac{2^*}{2}}}{2\cdot2^*}\left(1-\left\|V_{-}\right\|_{\frac{N}{2}} S^{-1}\right)(2^*A_{p_1}+1)\right]^{\frac{2[2\cdot2^*-N(p_1-2)]}{(2^*-2)[2p_1-N(p_1-2)]}} .
\end{equation*}

Let
$$
\widehat{A}=S^{-
\frac{2^*}{2}}\left[\frac{4(2^*-2)}{N(p_1-2)(4-N(p_1-2))} +\frac{1}{2^*}\right],\ t_g=\left[\frac{\left(1-\left\|V_{-}\right\|_{\frac{N}{2}} S^{-1}\right)}{2^* \widehat{A}}\right]^{\frac{2}{2^*-2}} ,
$$
so that $t_g>t_2$ by the definition of $\widetilde{\Theta}_V, \max\limits_{t \in\left[t_2, \infty\right)} \widehat{g}_2(t)=\widehat{g}_2(t_g)$ and
\begin{eqnarray*}
\max\limits_{t \in \mathbb{R}^{+}} \widehat{g}_1(t) & \geq& \max\limits_{t \in\left[t_2, \infty\right)} \widehat{g}_2(t)
 =\widehat{A}^{-\frac{2}{2^*-2}}\left(1-\left\|V_{-}\right\|_{\frac{N}{2}} S^{-1}\right)^{\frac{2^*}{2^*-2}} \left[\frac{2^*-2}{2}\left(\frac{1}{2^*}\right)^{\frac{2^*}{2^*-2}}\right]  .
\end{eqnarray*}
Set $\bar{r}_\Theta=\max \left\{\frac{1}{t_1}, \sqrt{\frac{2 \theta \Theta}{t_g}}\right\}$, then $v_{\frac{1}{\bar{r}_\Theta}} \in S_{r, \Theta}$ for any $r>\bar{r}_\Theta$, and
\begin{equation}\label{eq9.11}
  \left\|\nabla v_{\frac{1}{\bar{r}_\Theta}}\right\|_2^2=\left(\frac{1}{\bar{r}_\Theta}\right)^2\left\|\nabla v_1\right\|_2^2<t_g=\left[\frac{\left(1-\left\|V_{-}\right\|_{\frac{N}{2}} S^{-1}\right)}{2^* \widehat{A}}\right]^{\frac{2}{2^*-2}} .
\end{equation}
Moreover,
\begin{equation}\label{eeq9.12}
\mathcal{J}_{\bar{r}_\Theta, s}\left(v_{\frac{1}{\bar{r}_\Theta}}\right) \leq \widehat{h}\left(\frac{1}{\bar{r}_\Theta}\right) \leq \widehat{h}\left(t_1\right) .
\end{equation}
Let $u^0=v_{\frac{1}{\bar{\tau}_\Theta}}, u^1=v_{t_0}$ and
$$
\widetilde{r}_\Theta=\max \left\{\frac{1}{t_0}, \bar{r}_\Theta\right\} .
$$
Then the statement (i) holds by \eqref{eq9.7}, \eqref{eq9.8}, \eqref{eq9.11}, \eqref{eeq9.12}.

(ii) holds by \eqref{eq9.10} and a direct calculation.

(iii) In view of $J_{r, s}\left(u^1\right) \leq 0$ for any $\gamma \in \Gamma_{r, \Theta}$ and the definition of $t_0$, we have
$$
\|\nabla \gamma(0)\|_2^2<t_g<\|\nabla \gamma(1)\|_2^2 .
$$
It then follows from \eqref{eq9.10} that
\begin{eqnarray*}
\max\limits_{t \in[0,1]} J_{r, s}(\gamma(t)) & \geq& g_2\left(t_g\right)
 =\widehat{A}^{-\frac{2}{2^*-2}}\left(1-\left\|V_{-}\right\|_{\frac{N}{2}} S^{-1}\right)^{\frac{2^*}{2^*-2}} \left[\frac{2^*-2}{2}\left(\frac{1}{2^*}\right)^{\frac{2^*}{2^*-2}}\right]
\end{eqnarray*}
for any $\gamma \in \Gamma_{r, \Theta}$, hence the first inequality in (iii) holds. We define a path $\gamma:[0,1] \rightarrow S_{r, \Theta}$ by
$$
\gamma(t): \Omega_r \rightarrow \mathbb{R}, \quad x \mapsto\left(\tau t_0+(1-\tau) \frac{1}{\widetilde{r}}_\Theta\right)^{\frac{N}{2}} v_1\left(\left(\tau t_0+(1-\tau) \frac{1}{\widetilde{r}_\Theta}\right) x\right) .
$$
Then $\gamma \in \Gamma_{r, \Theta}$, and the second inequality in (iii) follows from \eqref{eq9.6}.
\end{proof}

Using Proposition \ref{p3.1} to $\widetilde{J}_{r, s}$,  it follows that
$$
A(u)=\frac{1}{2} \int_{\Omega_r}|\nabla u|^2 d x+\frac{1}{2} \int_{\Omega_r} V(x) u^2 d x \quad \text { and } \quad B(u)=\frac{1}{2^*} \int_{\Omega_r}|u|^{2^*} d x+\beta \int_{\Omega_r}F(u) d x.
$$
Hence, for almost every $s \in\left[\frac{1}{2}, 1\right]$, there exists a bounded Palais-Smale sequence $\left\{u_n\right\}$ satisfying
$$
\mathcal{J}_{r, s}\left(u_n\right) \rightarrow \widetilde{m}_{r, s}(\Theta) \quad \text { and }\left.\quad \mathcal{J}_{r, s}^{\prime}\left(u_n\right)\right|_{T_{u_n} S_{r, \Theta}} \rightarrow 0.
$$
Similar to the proof of Lemmas \ref{L8.2} and \ref{L8.3}, we have the following lemmas.
\begin{lemma}\label{L9.2}
If $\beta > 0$ and the assumptions of Theorem \ref{t1.9} hold, then $\widetilde{m}_{r, s}(\Theta)<\frac{\zeta}{N}S^{\frac{N}{2}}$, where $\zeta=s^{-\frac{2}{2^*-2}}$.
\end{lemma}
\begin{lemma}\label{L9.3}
Assume $0<\Theta<\widetilde{\Theta}_V$ where $\widetilde{\Theta}_V$ is given in Theorem \ref{t1.9}, let $r>r_\Theta$, where $r_\Theta$ is defined in Lemma \ref{L9.1}. Then problem \eqref{eq9.3} has a solution $\left(\lambda_{r, s}, u_{r, s}\right)$ for almost every $s \in\left[\frac{1}{2}, 1\right]$. Moreover, $u_{r, s} \geq 0$ and $\mathcal{J}_{r, s}\left(u_{r, s}\right)=\widetilde{m}_{r, s}(\Theta)$.
\end{lemma}

In order to obtain a solution of \eqref{eq9.1}, we also need to prove a uniform estimate for the solutions of \eqref{eq9.3} established in Lemma \ref{L9.3}.
\begin{lemma}\label{L9.4}
For fixed $\Theta>0$ the set of solutions $u \in S_{r, \Theta}$ of \eqref{eq9.3} is bounded uniformly in $s$ and $r$.
\end{lemma}
\begin{proof}
Since $u$ is a solution of \eqref{eq9.3}, we have
\begin{equation*}
  \int_{\Omega_r}|\nabla u|^2 d x+\int_{\Omega_r} V u^2 d x=s \int_{\Omega_r}|u|^{2^*} d x+s\beta \int_{\Omega_r}f(u)u d x-\lambda \int_{\Omega_r}|u|^2 d x .
\end{equation*}
The Pohozaev identity implies
\begin{eqnarray*}
&&\frac{N-2}{2 N} \int_{\Omega_r}|\nabla u|^2 d x+\frac{1}{2 N} \int_{\partial \Omega_r}|\nabla u|^2(x \cdot \mathbf{n}) d \sigma+\frac{1}{2 N} \int_{\Omega_r}\widetilde{V}(x) u^2+\frac{1}{2} \int_{\Omega_r} V u^2 d x \\
&=&-\frac{\lambda}{2} \int_{\Omega_r}|u|^2 d x+\frac{s}{2^*} \int_{\Omega_r}|u|^{2^*} d x+s\beta \int_{\Omega_r}F(u) d x
\end{eqnarray*}
where $\mathbf{n}$ denotes the outward unit normal vector on $\partial \Omega_r$. It then follows from $\beta >0$ and $(f_2)$ that
\begin{eqnarray*}
&&\frac{1}{N} \int_{\Omega_r}|\nabla u|^2 d  x-\frac{1}{2 N} \int_{\partial \Omega_r}|\nabla u|^2(x \cdot \mathbf{n}) d \sigma-\frac{1}{2 N} \int_{\Omega_r}(\nabla V \cdot x) u^2 d x \\
& \leq&\frac{2^*-2}{2}\left(\frac{1}{2} \int_{\Omega_r}|\nabla u|^2 d x+\frac{1}{2} \int_{\Omega_r} V u^2 d x-\widehat{m}_{r, s}(\Theta)\right)+  s\frac{\beta (p_2-2^*) }{2  }  \int_{\Omega_r}F(u) d x .
\end{eqnarray*}
Using Gagliardo-Nirenberg inequality, \eqref{eq3.5} and (iii) in Lemma \ref{L9.1}, we have
 \begin{eqnarray*}
\frac{2^*-2}{2} \widehat{m}_{r, s}(\Theta)
&\geq & \frac{N(p_2-2)-4}{4 N} \int_{\Omega_r}|\nabla u|^2 d x-\Theta\left(\frac{1}{2 N}\|\nabla V \cdot x\|_{\infty}+\frac{p_2-2}{4}\|V\|_{\infty}\right)\\
&&+\frac{s\alpha\beta (p_2-2^*) }{2  }C_{N, p_1} \Theta^{\frac{2 p_1-N(p_1-2)}{4}}\left(\int_{\Omega_r}|\nabla u|^2 d x\right)^{\frac{N(p_1-2)}{4}}.
 \end{eqnarray*}
Since $2<p_1<2+\frac{4}{N}$, we can bound $\int_{\Omega_r}|\nabla u|^2 d x$ uniformly in $s$ and $r$.
\end{proof}
\begin{lemma}\label{L9.5}
Assume $0<\Theta<\widetilde{\Theta}_V$, where $\widetilde{\Theta}_V$ is given in Theorem \ref{t1.9}, and let $r>\widetilde{r}_\Theta$, where $\widetilde{r}_\Theta$ is defined in Lemma \ref{L9.1}. Then equation \eqref{eq9.3} admits a solution $\left(\lambda_{r, \Theta}, u_{r, \Theta}\right)$ for every $r>\widetilde{r}_\Theta$ such that $u_{r, \Theta}>0$ in $\Omega_r$.
\end{lemma}
\begin{proof}
The proof of lemma is similar to the Lemma \ref{L8.5}.
\end{proof}
\noindent\textbf{Proof of Theorem \ref{t1.9}} The proof is an immediate consequence of Lemmas \ref{L9.5}  and \ref{L3.5}.

\section{Mass critical case}

\subsection{Proof of Theorem \ref{t1.10}}
This subsection considers the case of $p_1=2+\frac{4}{N}$, so we need to modify the proof of Theorem \ref{t1.3}.
\begin{lemma}\label{L10.1}
For $0<\Theta<\widetilde{\Theta}_V$ where $\widetilde{\Theta}_V$ is defined in Theorem \ref{t1.10}, there exist $\widetilde{r}_\Theta>0$ and $u^0, u^1 \in S_{r_\Theta, \Theta}$ such that

(i) For $r>\widetilde{r}_\Theta$ and $s \in\left[\frac{1}{2}, 1\right]$ we have $J_{r, s}\left(u^1\right) \leq 0$ and
$$
J_{r, s}\left(u^0\right)<\frac{(N(q-2)-4)\left(1-\left\|V_{-}\right\|_{\frac{N}{2}} S^{-1}-2\alpha\beta C_N\Theta^{\frac{2}{N}}\right)^{\frac{N(q-2)}{N(q-2)-4}}}{2 [N(q-2)]^{\frac{N(q-2)}{N(q-2)-4}}}\left[\frac{2 q}{ C_{N, q} \Theta^{\frac{2 q-N(q-2)}{4}}}\right]^{\frac{4}{N(q-2)-4}}.
$$
Moreover,
$$
\left\|\nabla u^0\right\|_2^2<\left[\frac{2 q}{N(q-2) C_{N, q}}\left(1-\left\|V_{-}\right\|_{\frac{N}{2}} S^{-1}-2\alpha\beta C_N\Theta^{\frac{2}{N}}\right) \Theta^{\frac{q(N-2)-2 N}{4}}\right]^{\frac{4}{N(q-2)-4}}
$$
and
$$
\left\|\nabla u^1\right\|_2^2>\left[\frac{2 q}{N(q-2) C_{N, q}}\left(1-\left\|V_{-}\right\|_{\frac{N}{2}} S^{-1}-2\alpha\beta C_N\Theta^{\frac{2}{N}}\right) \Theta^{\frac{q(N-2)-2 N}{4}}\right]^{\frac{4}{N(q-2)-4}}.
$$

(ii) If $u \in S_{r, \Theta}$ satisfies
$$
\|\nabla u\|_2^2=\left[\frac{2 q}{N(q-2) C_{N, q}}\left(1-\left\|V_{-}\right\|_{\frac{N}{2}} S^{-1}-2\alpha\beta C_N\Theta^{\frac{2}{N}}\right) \Theta^{\frac{q(N-2)-2 N}{4}}\right]^{\frac{4}{N(q-2)-4}},
$$
then there holds
$$
J_{r, s}(u) \geq \frac{(N(q-2)-4)\left(1-\left\|V_{-}\right\|_{\frac{N}{2}} S^{-1}-2\alpha\beta C_N\Theta^{\frac{2}{N}}\right)^{\frac{N(q-2)}{N(q-2)-4}}}{2 [N(q-2)]^{\frac{N(q-2)}{N(q-2)-4}}}\left[\frac{2 q}{ C_{N, q} \Theta^{\frac{2 q-N(q-2)}{4}}}\right]^{\frac{4}{N(q-2)-4}}.
$$

(iii) Let
$$
m_{r, s}(\Theta)=\inf _{\gamma \in \Gamma_{r, \Theta}} \sup\limits_{t \in[0,1]} J_{r, s}(\gamma(t)),
$$
where
$$
\Gamma_{r, \Theta}=\left\{\gamma \in C\left([0,1], S_{r, \Theta}\right): \gamma(0)=u^0, \gamma(1)=u^1\right\}.
$$
Then
$$
m_{r, s}(\Theta) \geq\frac{(N(q-2)-4)\left(1-\left\|V_{-}\right\|_{\frac{N}{2}} S^{-1}-2\alpha\beta C_N\Theta^{\frac{2}{N}}\right)^{\frac{N(q-2)}{N(q-2)-4}}}{2 [N(q-2)]^{\frac{N(q-2)}{N(q-2)-4}}}\left[\frac{2 q}{ C_{N, q} \Theta^{\frac{2 q-N(q-2)}{4}}}\right]^{\frac{4}{N(q-2)-4}}
$$
and
\begin{equation*}
  m_{r, s}(\Theta) \leq \frac{N(q-2)-4}{2}\left(\frac{\theta\left(1+\|V\|_{\frac{N}{2}} S^{-1}\right)}{N(q-2)}\right)^{\frac{N(q-2)}{N(q-2)-4}}(4 q)^{\frac{4}{N(q-2)-4}}|\Omega|^{\frac{2(q-2)}{N(q-2)-4}} \Theta^{\frac{N(q-2)-2 q}{N(q-2)-4}}.
\end{equation*}
where $\theta$ is the principal eigenvalue of $-\Delta$ with Dirichlet boundary condition in $\Omega$.
\end{lemma}
\begin{proof}
We only need to modify the proof of Lemma \ref{L5.1}. There exists $0<t_1<t_\Theta$ such that for any $t \in\left[0, t_1\right]$,
\begin{equation}\label{eq10.1}
h_2(t)<\frac{(N(q-2)-4)\left(1-\left\|V_{-}\right\|_{\frac{N}{2}} S^{-1}-2\alpha\beta C_N\Theta^{\frac{2}{N}}\right)^{\frac{N(q-2)}{N(q-2)-4}}}{2 [N(q-2)]^{\frac{N(q-2)}{N(q-2)-4}}}\left[\frac{2 q}{ C_{N, q} \Theta^{\frac{2 q-N(q-2)}{4}}}\right]^{\frac{4}{N(q-2)-4}}.
\end{equation}
On the other hand, it follows from \eqref{eq3.5}, the Gagliardo-Nirenberg inequality and the H\"older inequality that
\begin{eqnarray}\label{eq10.2}
J_{r, s}(u)
&=&\left[\frac{1-\left\|V_{-}\right\|_{\frac{N}{2}} S^{-1}}{2}-\alpha\beta C_{N} \Theta^{\frac{2}{N}}\right] \|\nabla u\|_2^2-\frac{C_{N, q} \Theta^{\frac{2 q-N(q-2)}{4}}}{q}\|\nabla u\|_2^{\frac{N(q-2)}{2}}.
\end{eqnarray}
Define
\begin{eqnarray*}
g_1(t):=\left[\frac{1-\left\|V_{-}\right\|_{\frac{N}{2}} S^{-1}}{2}-\alpha\beta C_{N} \Theta^{\frac{2}{N}}\right] t-\frac{C_{N, q} \Theta^{\frac{2 q-N(q-2)}{4}}}{q}t^{\frac{N(q-2)}{4}}
\end{eqnarray*}
and
$$
t_g=\left[\frac{2 q}{N(q-2) C_{N, q}}\left(1-\left\|V_{-}\right\|_{\frac{N}{2}} S^{-1}-2\alpha\beta C_N\Theta^{\frac{2}{N}}\right) \Theta^{\frac{q(N-2)-2 N}{4}}\right]^{\frac{4}{N(q-2)-4}},
$$
it is easy to see that $g_1$ is increasing on $(0, t_g)$ and decreasing on $(t_g, \infty)$, and
$$
g_1(t_g)=\frac{(N(q-2)-4)\left(1-\left\|V_{-}\right\|_{\frac{N}{2}} S^{-1}-2\alpha\beta C_N\Theta^{\frac{2}{N}}\right)^{\frac{N(q-2)}{N(q-2)-4}}}{2 [N(q-2)]^{\frac{N(q-2)}{N(q-2)-4}}}\left[\frac{2 q}{ C_{N, q} \Theta^{\frac{2 q-N(q-2)}{4}}}\right]^{\frac{4}{N(q-2)-4}} .
$$
Set $\bar{r}_\Theta=\max \left\{\frac{1}{t_1}, \sqrt{\frac{2 \theta \Theta}{t_g}}\right\}$, then $v_{\frac{1}{\bar{r}_\Theta}} \in S_{r, \Theta}$ for any $r>\bar{r}_\Theta$, and
\begin{eqnarray}\label{eq10.3}
&&\left\|\nabla v_{\frac{1}{\bar{r}_\Theta}}\right\|_2^2=\left(\frac{1}{\bar{r}_\Theta}\right)^2\left\|\nabla v_1\right\|_2^2\nonumber\\
&<&t_g=\left[\frac{2 q}{N(q-2) C_{N, q}}\left(1-\left\|V_{-}\right\|_{\frac{N}{2}} S^{-1}-2\alpha\beta C_N\Theta^{\frac{2}{N}}\right) \Theta^{\frac{q(N-2)-2 N}{4}}\right]^{\frac{4}{N(q-2)-4}} .
\end{eqnarray}
Then the statement (i) holds by \eqref{eq5.5}, \eqref{eq10.1}, \eqref{eq5.9}, \eqref{eq10.3}.

(ii) holds by \eqref{eq10.2} and a direct calculation.

(iii) In view of $J_{r, s}\left(u^1\right) \leq 0$ for any $\gamma \in \Gamma_{r, \Theta}$ and the definition of $t_0$, we have
$$
\|\nabla \gamma(0)\|_2^2<t_g<\|\nabla \gamma(1)\|_2^2 .
$$
It then follows from \eqref{eq10.2} that
\begin{eqnarray*}
&&\max\limits_{t \in[0,1]} J_{r, s}(\gamma(t))  \geq g_1\left(t_g\right) \\
& =&\frac{(N(q-2)-4)\left(1-\left\|V_{-}\right\|_{\frac{N}{2}} S^{-1}-2\alpha\beta C_N\Theta^{\frac{2}{N}}\right)^{\frac{N(q-2)}{N(q-2)-4}}}{2 [N(q-2)]^{\frac{N(q-2)}{N(q-2)-4}}}\left[\frac{2 q}{ C_{N, q} \Theta^{\frac{2 q-N(q-2)}{4}}}\right]^{\frac{4}{N(q-2)-4}}
\end{eqnarray*}
for any $\gamma \in \Gamma_{r, \Theta}$, hence the first inequality in (iii) holds. We define a path $\gamma:[0,1] \rightarrow S_{r, \Theta}$ by
$$
\gamma(t): \Omega_r \rightarrow \mathbb{R}, \quad x \mapsto\left(\tau t_0+(1-\tau) \frac{1}{\widetilde{r}}_\Theta\right)^{\frac{N}{2}} v_1\left(\left(\tau t_0+(1-\tau) \frac{1}{\widetilde{r}_\Theta}\right) x\right) .
$$
Then $\gamma \in \Gamma_{r, \Theta}$, and the second inequality in (iii) follows from \eqref{eq5.4}.
\end{proof}
\begin{lemma}\label{L10.2}
For fixed $\Theta>0$ the set of solutions $u \in S_{r, \Theta}$ of \eqref{eq5.1} is bounded uniformly in $s$ and $r$.
\end{lemma}
\begin{proof}
We only need to modify the proof of Lemma \ref{L5.3}. Using Gagliardo-Nirenberg inequality, \eqref{eq3.5} and (iii) in Lemma \ref{L5.1}, we have
 \begin{eqnarray*}
\frac{q-2}{2} m_{r, s}(\Theta)
&\geq & \left[\frac{N(p_2-2)-4}{4 N}-\frac{s\alpha\beta (q-p_2) }{2  }C_{N} \Theta^{\frac{2}{N}}\right] \int_{\Omega_r}|\nabla u|^2 d x\\
&&-\Theta\left(\frac{1}{2 N}\|\nabla V \cdot x\|_{\infty}+\frac{p_2-2}{4}\|V\|_{\infty}\right).
 \end{eqnarray*}
Since $0<\Theta<\widetilde{\Theta}_V$, we can bound $\int_{\Omega_r}|\nabla u|^2 d x$ uniformly in $s$ and $r$.
\end{proof}
\noindent\textbf{Proof of Theorem \ref{t1.10}} The proof is an immediate consequence of Lemmas \ref{L5.4}  and \ref{L3.5}.
\subsection{Proof of Theorem \ref{t1.11}}
Firstly, we modify the proof of Lemma \ref{L3.1}. Using \eqref{eq3.3}, \eqref{eq3.4}, \eqref{eq3.5} and $\frac{1}{2}\leq s\leq1$, it holds
\begin{eqnarray*}
I_{\frac{1}{t}, s}\left(v_t\right)
&\leq & \frac{1}{2}\left(1+\|V\|_{\frac{N}{2}} S^{-1}-2\alpha\beta C_{N} \Theta^{\frac{2}{N}}\right) t^2 \theta\Theta-\frac{1}{2q} t^{\frac{N(q-2)}{2}} \Theta^{\frac{q}{2}}\cdot|\Omega|^{\frac{2-q}{2}} \nonumber\\
&=: & h(t) .
\end{eqnarray*}
Note that since $2+\frac{4}{N}<q<2^*$ and $\beta\leq0$, there exist $0<T_\Theta<t_0$ such that $h(t_0)=0, h(t)<0$ for any $t>t_0, h(t)>0$ for any $0<t<t_0$ and $h(T_\Theta)=\max\limits_{t \in \mathbb{R}^{+}} h(t)$.
\begin{lemma}\label{L10.4}
Let $\left(\lambda_{r, \Theta}, u_{r, \Theta}\right)$ be the solution of \eqref{eq1.1} from Lemma \ref{L3.4}. If $\|\widetilde{V}_{+}\|_{\frac{N}{2}} < 2 S$, then there exists $\bar{\Theta}>0$ such that
$$
\liminf\limits_{r \rightarrow \infty} \lambda_{r, \Theta}>0 \quad \text { for } 0<\Theta<\bar{\Theta} .
$$
\end{lemma}
\begin{proof}
We only need to modify the proof of Lemma \ref{L3.6}. It follows from \eqref{eq3.5}, \eqref{eq3.28}, \eqref{eq3.29}, \eqref{eq3.30}, $(f_2)$ and $2+\frac{4}{N}<q<2^*$ that
\begin{eqnarray*}
 \left(\frac{1}{q}-\frac{1}{ 2}\right)\lambda_\Theta \int_{\mathbb{R}^N} u_\Theta^2 d x
& \leq& \left[\frac{(N-2) q-2 N}{2 N q}-\frac{\beta(q-p_2)\alpha C_{N}}{q}   \Theta^{ \frac{2 }{N}} \right] \int_{\mathbb{R}^N}\left|\nabla u_\Theta\right|^2 d x \\
&&+\frac{\|\widetilde{V}\|_{\infty}}{2 N} \Theta+\frac{(q-2)\|V\|_{\infty}}{2 q} \Theta\\
& \rightarrow&-\infty \ \text { as } \Theta \rightarrow 0,
\end{eqnarray*}
since $0<\Theta<\widehat{\Theta}_V$.
\end{proof}
\noindent\textbf{Proof of Theorem \ref{t1.11}} The proof is an immediate consequence of Lemmas \ref{L3.4}, \ref{L3.5} and \ref{L10.4}.
\subsection{Proof of Theorem \ref{t1.12}}
Similarly, we only need to modify the proof of Theorem \ref{t1.2}. Since $\beta>0$, it follows from the Gagliardo-Nirenberg inequality and the H\"older inequality that
\begin{eqnarray*}
I_r(u)&=&\frac{1}{2} \int_{\Omega_r}|\nabla u|^2 d x+\frac{1}{2} \int_{\Omega_r} V(x) u^2 d x-\frac{N}{2N+4} \int_{\Omega_r}|u|^{2+\frac{4}{N}} d x-\beta \int_{\Omega_r}F(u) d x\\
&\geq&\frac{1}{2}\left(1-\left\|V_{-}\right\|_{\frac{N}{2}} S^{-1}-\frac{NC_{N}\Theta^{\frac{2}{N}}}{N+2}\right) \int_{\Omega_r}|\nabla u|^2 d x\\
&&-\alpha\beta C_{N, p_1} \Theta^{\frac{2 p_1-N(p_1-2)}{4}}\left(\int_{\Omega_r}|\nabla u|^2 d x\right)^{\frac{N(p_1-2)}{4}}\\
&=&h_1(t),
\end{eqnarray*}
where
\begin{eqnarray*}
h_1(t):=\frac{1}{2}\left(1-\left\|V_{-}\right\|_{\frac{N}{2}} S^{-1}-\frac{NC_{N}\Theta^{\frac{2}{N}}}{N+2}\right)t^2-\alpha\beta C_{N, p_1} \Theta^{\frac{2 p_1-N(p_1-2)}{4}}t^{\frac{N(p_1-2)}{2}}.
\end{eqnarray*}
In view of $2<p_1<2+\frac{4}{N}$, there exists $T_\Theta>0$ such that $h_1(t)<0$ for $0<t<T_\Theta$ and $h_1(t)>0$ for $t>T_\Theta$.

\noindent\textbf{Proof of Theorem \ref{t1.12}} The proof is a direct consequence of Lemma \ref{L4.1} and Lemma \ref{L3.5}.
\section{Final comments}
Some similar result(Theorems \ref{t1.1}, \ref{t1.2}, \ref{t1.3}, but there are subtle changes in the assumptions) can be proved for the following class of problem
\begin{equation*}
 \left\{\aligned
& -\Delta u+V(x)u+\lambda u=w(u)+\beta |u|^{p-2}u, &x\in \Omega, \\
& \int_{\Omega} |u|^2dx=\Theta,u\in H_0^1(\Omega), &x\in \Omega,
\endaligned
\right.
\end{equation*}
where $\Omega \subset \mathbb{R}^N$ is either all of $\mathbb{R}^N$ or a bounded smooth convex domain, $N \geq 3$, $2<p< 2+\frac{4}{N}$, the mass $\Theta>0$ and the parameter $\beta \in \mathbb{R}$ are prescribed. Nonlinearity $w$ satisfies:
\begin{itemize}
\item[$(W_1)$]\ $w\in C^1(\mathbb{R}, \mathbb{R})$ and $w$ is odd.

\item[$(W_2)$]\ There exists some $(p_1, p_2) \in \mathbb{R}_{+}^2$ satisfying $2+\frac{4}{N}<p_2 \leq p_1<2^*$ such that
$$
p_2 W(\tau) \leq w(\tau) \tau \leq p_1 W(\tau) \text { with } W(\tau)=\int_0^\tau w(t) d t .
$$
\end{itemize}

\end{document}